\documentclass[11pt,reqno]{amsart}

\usepackage[T1]{fontenc}
\usepackage[utf8]{inputenc}
\usepackage{lmodern} 
\usepackage[final]{microtype}
\usepackage{csquotes}

\usepackage[dvipsnames]{xcolor}

\usepackage[a4paper,margin=1.1in]{geometry}
\usepackage{amsmath,amssymb,amsthm,mathtools} 
\usepackage{bm}
\usepackage{shuffle}
\usepackage{mathrsfs}
\usepackage{thmtools}
\usepackage{tikz-cd}

\usepackage{comment}
\usepackage{enumitem}
\usepackage{todonotes}

\usepackage{booktabs}
\usepackage{tabularx}
\usepackage{array}
\usepackage{xparse}

\usepackage{booktabs}
\usepackage{tabularx}
\usepackage{array}
\usepackage{xparse}

\NewDocumentEnvironment{notationtable}{O{Notation} +b}
{%
  \begin{table}[h!]
  \centering
  \small
  \begin{tabularx}{0.96\textwidth}{%
    @{}%
    >{\raggedright\arraybackslash}p{0.22\textwidth}%
    >{\raggedright\arraybackslash}X%
    >{\raggedleft\arraybackslash}p{0.18\textwidth}%
    @{}}
  \toprule
  \textbf{Symbol} & \textbf{Description} & \textbf{Reference} \\
  \midrule
  #2
  \bottomrule
  \end{tabularx}
  \caption{#1}
  \label{tab:Notation}
  \end{table}
}
{}

\NewDocumentCommand{\notation}{m m m}{%
  \ensuremath{#1}\dotfill & #2 & #3 \\
}

\usepackage[dvipsnames]{xcolor}
\usepackage[
  colorlinks,
  linkcolor=MidnightBlue,
  citecolor=ForestGreen,
  urlcolor=BrickRed,
  pdfauthor={}
]{hyperref}
\usepackage[nameinlink,capitalize]{cleveref}

\usepackage[
  backend=biber,
  style=alphabetic,
  maxbibnames=99,
  url=false,
  doi=true,
  isbn=false,
  giveninits=true
]{biblatex}
\usepackage{orcidlink} 

\newcommand{\R}{\mathbb{R}}
\newcommand{\N}{\mathbb{N}}

\newcommand{\cD}{\mathcal{D}}
\newcommand{\cG}{\mathcal{G}}
\newcommand{\cH}{\mathcal{H}}
\newcommand{\cI}{\mathcal{I}}
\newcommand{\cL}{\mathcal{L}}

\newcommand{\frg}{\mathfrak{g}}

\newcommand{\floor}[1]{\lfloor #1 \rfloor}

\newcommand{\var}{\textrm{var}}

\newcommand{\up}{ {\floor{p}} }

\newcommand{\cost}{ \theta }
\newcommand{\mcost}{ \Theta }

\DeclareMathOperator{\id}{id}

\DeclareMathOperator{\eva}{ev}

\DeclareMathOperator{\proj}{proj}

\newcommand{\rev}[1]{\overset{\leftarrow}{#1}}
\newcommand{\met}{\textbf{LipMet}_1}
\newcommand{\metlift}{\textbf{MetExt}_1}

\newcommand{\pvdi}[4]{\left\|#1\right\|_{#2\text{-var},#3,#4}}
\newcommand{\pvd}[3]{\left\|#1\right\|_{#2\text{-var},#3}}
\newcommand{\pvi}[3]{\left\|#1\right\|_{#2\text{-var},#3}}
\newcommand{\pv}[2]{\left\|#1\right\|_{#2\text{-var}}}

\newcommand{\control}{\omega}
\newcommand{\CC}{{\textrm{CC}}} 

\newcommand{\norm}[1]{\left\lVert#1\right\rVert}

\newcommand{\beq}{\begin{equation}}
\newcommand{\eeq}{\end{equation}}

\setlist[enumerate,1]{label=(\roman*)}

\theoremstyle{definition}
\declaretheorem[name=Theorem,numberwithin=section]{theorem}

\declaretheorem[name=Proposition,sibling=theorem]{proposition}
\declaretheorem[name=Lemma,sibling=theorem]{lemma}
\declaretheorem[name=Corollary,sibling=theorem]{corollary}
\declaretheorem[name=Definition,style=definition,sibling=theorem]{definition}

\declaretheorem[name=Claim,style=definition,sibling=theorem]{claim}
\declaretheorem[name=Example,style=definition,sibling=theorem]{example}
\declaretheorem[name=Counterexample,style=definition,sibling=theorem]{counterexample}
\declaretheorem[name=Remark,style=remark,sibling=theorem]{remark}

\declaretheorem[name=Theorem,numberwithin=section]{specialtheorem}

\declaretheorem[name=Corollary,sibling=specialtheorem]{specialcorollary}

\crefname{theorem}{Theorem}{Theorems}
\crefname{conjecture}{Conjecture}{Conjecture}
\crefname{lemma}{Lemma}{Lemmas}
\crefname{proposition}{Proposition}{Propositions}
\crefname{corollary}{Corollary}{Corollaries}
\crefname{definition}{Definition}{Definitions}
\crefname{assumption}{Assumption}{Assumptions}
\crefname{example}{Example}{Examples}

\crefname{counterexample}{Counterexample}{Counterexamples}
\crefname{remark}{Remark}{Remarks}
\crefname{equation}{}{}
\numberwithin{equation}{section}
\crefname{specialtheorem}{Theorem}{Theorems}
\crefname{specialcorollary}{Corollary}{Corollaries}

\title[Signature group]{
Metric Geometry of the Signature Group for \texorpdfstring{$p$}{p}--Variation Rough Paths}

\author{Felix Medwed} 
\address{Institute of Mathematics, University of Potsdam, D-14469 Potsdam, Germany}
\email{felix.medwed@uni-potsdam.de}

\author{Sylvie Paycha} 
\address{Institute of Mathematics, University of Potsdam, D-14469 Potsdam, Germany}
\email{paycha@math.uni-potsdam.de}

\author{Alexander Schmeding} 
\address{Norwegian University of Science and Technology, NO-7491 Trondheim, Norway}
\email{alexander.schmeding@ntnu.no}

\date{\today}

\begin{document}

\begin{abstract}

The signatures of $p$-rough paths form a subgroup of sufficiently high-level truncated tensor algebras, whose inverse limit is a subgroup of the full tensor algebra. For $p \geq 1$, we  provide a {\it top-down description} of the signature group as the inverse limit of finite-dimensional Carnot--Carathéodory geometries in the $p$-variation setting.
We show that every compatible choice of metrics induces a topological tree structure on the inverse-limit group, under which the signature group is not a topological group. This extends the results of Enrico Le Donne and Roland Züst from bounded variation to rough paths. We also characterise the dependence of the inverse-limit groups and their metric completions on the choice of metric, identifying them with the tree-reduced path group of Horatio Boedihardjo, Xiang Geng, Terry Lyons, and Danyu Yang.

\end{abstract}

\maketitle

\tableofcontents



\section{Introduction}

The signature group is a strict subgroup of the group-like elements of the tensor series algebra. It is isomorphic to the group of unparametrised tree-reduced paths, with path concatenation and reversal corresponding to multiplication and inversion. This correspondence was established for bounded variation paths over  $\R^n$ in the seminal work of \cite{HamblyLyonsUniqueness_2010} and later extended to weakly geometric $p$-rough paths for all $p\geq 1$ in \cite{BOEDIHARDJO_LYONS_2016720}. These results identify the signature group as a faithful representation of the underlying path group, making it a natural object of study. Recent work has focused on its topological \cite{cass2024topologiesunparameterisedroughpath} and geometric properties.

As observed in \cite{HamblyLyonsUniqueness_2010}, the signature group admits several natural topologies, each leading to different geometric features. At the finite level, the $k$-truncated signature group is the free nilpotent Lie group of step $k$, whose canonical geometry is given by the Carnot--Carathéodory metric (see  \cite[Chapter 7]{Friz_Victoir_2010_fullbook} and \cite[Chapter 11]{LeDonne2025}).  Building on this viewpoint,  the authors of \cite{LeDonneZuest_public}  realised the bounded variation signature group as the inverse limit of Carnot groups endowed with their Carnot--Carathéodory metrics. The resulting metric space is an $\R$-tree, reflecting the tree-reduced nature of signatures, but right translations fail to be continuous, so the signature group is not a topological group. 

The $\R$-tree structure was already observed in \cite[Proposition 4.1]{BOEDIHARDJO_LYONS_2016720}. In the present paper, we revisit it from the metric-geometric perspective of \cite{LeDonneZuest_public} and extend the inverse-limit construction to weakly geometric  $p$-rough paths for all $p\geq 1$. This provides a {\it top-down description} of the signature group as the inverse limit of finite-dimensional Carnot--Carathéodory geometries in the $p$-variation setting.

To highlight the non-triviality of this extension, we give subsequently three examples that motivate the main proposition below. The first recalls the construction of Le Donne and Z\"ust, the second explains the failure of a direct extension of the limiting metric to $p>1$, and the third illustrates how this obstruction can be resolved. We presuppose the notation introduced in the later sections.

\begin{example}\label{ex:Diagram-for-simple-path}
Let $e_1,e_2$ be the standard basis vectors of $\mathbb{R}^2$ endowed with the norm $\|\cdot\|$. Let $X\colon [0,1] \to \mathbb R^2$ be the path starting at the origin that first follows $e_1$ and then $e_2$, given by
\begin{equation*}
        X_t := \left\{\begin{array}{lr}
            2t e_1, & \text{ if } t \in [0,1/2] \\
            e_1 + (2t-1) e_2, & \text{ if } t \in [1/2,1]
        \end{array}\right.
    \end{equation*}
    This path is tree-reduced with length $\ell := \|e_1\| + \|e_2\|$ for any chosen norm $\|\cdot\|$ on $\R^2$ (i.e.~euclidean length $\ell = 2$). Then, the path $Y$ given by the Lyons lift is
    \begin{equation*}
        Y_t := \left\{\begin{array}{lr}
            \exp(2t e_1), & \text{ if } t \in [0,1/2] \\
            \exp(e_1) \exp((2t-1) e_2), &\text{ if } t \in [1/2,1]
        \end{array}\right.
    \end{equation*}
   The path $Y_t$ is injective with signature
    \begin{equation*}
        S(X) = Y_1 = \exp(e_1) \exp(e_2).
    \end{equation*}
    On the other hand one may consider the inverse limit object
    \begin{equation*}
        Z = ((1,e_1 + e_2), (1,e_1+e_2,\exp_2(e_1)\exp_2(e_2)),...) =: (g_k)_{k \in \N}.
    \end{equation*}
   By construction $(\Pi^1)^{-1}(S(X)) = Z$ where $\Pi^1$ is defined in \eqref{eq:Inverse-limit-to-product-conversion-map-p}. 
   The corresponding limit of the $\CC$ metrics (ref. \eqref{eq:CC-distance-p=1}) is then derived from \cref{prop:Isomorphisms} for $p=1$ (ref.~\eqref{eq:Inverse-Metric}). 
   Combining \cite[Corollary 1.6 with its proof]{HamblyLyonsUniqueness_2010} (or see \cite[Remark 4.1]{BOEDIHARDJO_LYONS_2016720}) and using \eqref{eq:Diagram-p-variation-LZ-precise}, we then infer that 
    \begin{equation*}
        \ell = \sup_{k \in \N} d^{1-\CC}_k(1_k,\exp_k(e_1)\exp_k(e_2)) = d^1(1,S(X)) = D^1([O]_\circlearrowleft, [Y]_\circlearrowleft).
    \end{equation*}
    where one can replace the supremum by the limit due to the 1-Lipschitz property of the canonical projections between truncations of $S(X)$.
    Note that a direct computation of the $\CC$ length is not fruitful due to the growth of complexity of the underlying control problems. \\
    In this way, we provided an explicit example for the diagram (ref. \cref{thm:Main-Theorem-Commutative-Diagram}) below.
\end{example}

\begin{example}[Continuation of \cref{ex:Diagram-for-simple-path}: Blow-Up] \label{ex:Blow-Up-CC-metric}
    Let us now consider for every $k \in \N\backslash\{1\}$ the elements
    \begin{equation*}
        g_k = \exp_k([e_1,e_2]) = \left\{\begin{array}{rl}
            \left(1,0,\frac{[e_1,e_2]}{1!},0,\frac{[e_1,e_2]^{\otimes_k 2}}{2!}, ...,\frac{[e_1,e_2]^{\otimes_k (k-1)}}{((k-1)/2)!},0\right) &, \text{ if $k$ odd}  \\
            \left(1,0,\frac{[e_1,e_2]}{1!},0,\frac{[e_1,e_2]^{\otimes_k 2}}{2!}, ...,\frac{[e_1,e_2]^{\otimes_k k}}{(k/2)!}\right) &, \text{ if $k$ even}
        \end{array}\right.
    \end{equation*}
    formed by the Lie bracket (ref. \hyperref[not:Lie-Bracket-exponential-function]{p.~\pageref*{not:Lie-Bracket-exponential-function}}) in contrast to $\exp(e_1)\exp(e_2) = \exp(e_1+e_2 + [e_1,e_2]/2 + ...)$ by the Baker--Campbell--Hausdorff formula. Then, as $k \to \infty$, the corresponding group-like element is the signature of a genuine rough path, the \emph{pure area path} $t \mapsto \exp_2(t [e_1,e_2])$ due to the interpretation of the second level of the tensor series (ref.~\cite{Friz2020_chapter2}). We now show that the sequence of $\CC$-metrics (w.r.t.~finite euclidean length), indexed by the truncation level $k$, blows up. \\
    Indeed, suppose w.l.o.g.~there was a continuous path $X\colon [0,1] \to \mathbb R^2$ of bounded variation, s.t.~$S(X) = \exp([e_1,e_2])$. Then, by the inequality for admissible tensor norms (ref.~\hyperref[not:admissible-tensor-norms]{p.~\pageref*{not:admissible-tensor-norms}})
    \begin{equation*}
        \|\proj_k S(X)\|_k \leq \frac{\ell^k}{k!}, \quad \forall k \in \N,
    \end{equation*}
    where $\ell = \operatorname{Length}(X)$ is again the length of the path $X$ w.r.t.~the here chosen euclidean norm.
    Then, since we consider the euclidean structure for $\R^2$, for every $k = 2m$ with $m \in \N$ we get by $S(X) = \exp([e_1,e_2])$ for the Hilbert-Schmidt norms
    \begin{equation*}
        \frac{\| [e_1,e_2]^{\otimes m} \|_{2m}}{m!} = \frac{\| [e_1,e_2] \|^{m}_2}{m!} = \frac{\sqrt{2}^m}{m!} \leq \frac{\ell^{2m}}{(2m)!}
    \end{equation*}
    and by rearranging and taking the infimum on the upper bound over all bounded variation paths that achieve the same $2m$-truncated signature, we get
    \begin{equation*}
        d_{2m}^{1\textnormal{-}\CC}(1_{2m},g_{2m}) \geq 2^{1/4} \left( \frac{2m!}{m!} \right)^{\frac{1}{2m}} \xrightarrow{m \to \infty} \infty.
    \end{equation*}
    But $\sup_{k \in \mathbb{N}} d_k^{1-\CC}(1_k,g_k) \leq \ell < \infty$, since $X$ was assumed to be of bounded variation. Hence, we arrive at a contradiction. Another such phenomenon can be seen (ref. \cite[Exercise 2.10]{Friz2020_chapter2}) by the approximation sequence
    \begin{equation*}
        X^{(n)}_t = \frac{\cos(2\pi n^2 t)}{n} e_1 + \frac{\sin(2 \pi n^2 t)}{n} e_2
    \end{equation*}
    for which its euclidean length $\ell^{(n)} := \operatorname{Length}(X^{(n)}) = 2\pi n \to \infty$ as $n \to \infty$. Using the fundamental bounds for the $\CC$-metrics, we again obtain a blow up of the limiting $\CC$-metric.
\end{example}

\begin{example}[Continuation of \cref{ex:Blow-Up-CC-metric}: Higher variation]
    Using the max type metric (ref. \eqref{eq:Max-metric-for-any-p}) we see that for the Hilbert-Schmidt norms the $p$-variation (ref. \eqref{eq:General-p-variation-for-pointed-paths})
    \begin{align*}
        \pv{X \colon t \mapsto \exp_2(t [e_1,e_2])}{p}^p &= \sup_{\mathcal{D}} \sum_{[a,b] \in \mathcal{D}} \left(\max\left\{0, \beta_p^{p/2} ((2/p)! \|(b-a)[e_1,e_2]\|_2)^{1/2}  \right\}\right)^p \\
        &= \beta_p^{p/2} \sup_{\mathcal{D}} \sum_{[a,b] \in \mathcal{D}} ((2/p)!)^{\frac{p}{2}} (b-a)^{p/2} 2^{p/4}
    \end{align*}
    and this is finite if and only if $p \geq 2$, where $\beta_p>0$ is some number only depending on $p$. Furthermore, the Lyons lift $Y$ of $X$ is then $Y\colon t \mapsto \exp(t[e_1,e_2])$, which is injective. Its signature is the element $S(X) = Y_1 = \exp([e_1,e_2])$ and the limiting object
    \begin{equation*}
        Z = ((1,0,\exp_2[e_1,e_2]),(1,0,\exp_2([e_1,e_2]),0)...) =: (g_k)_{k \in \N \backslash \{1\}}
    \end{equation*}
    can be mapped again to the signature $S(X)$ under the now corresponding $\Pi^2$ (or any $\Pi^p$ with $p \in [2,3)$). Hence, by \cref{ex:Blow-Up-CC-metric} the $\CC$-metrics have to be generalised to the $p$-$\CC$ metrics (ref.~\eqref{eq:CC-p-variation-norm} and \cref{lem:properties-of-dpCC}), so that one may apply again \cref{prop:Isomorphisms} (under technical assumptions provided in \cref{sec:Canonical-metric-relations-goals}) to gain the analogous correspondence between $D^p$ and $d^p$ in the case of $p$-variation. In this way, one should regard the pure area rough path above as the corresponding higher variation example for the diagram \cref{thm:Main-Theorem-Commutative-Diagram} below.
\end{example}

Together with the preceding discussion, this naturally raises the following questions:
\begin{itemize}
\item Can the construction of Le Donne and Züst be extended from the rectifiable setting ($p=1$) to weakly geometric rough paths ($p>1$) within the framework of \cite{BOEDIHARDJO_LYONS_2016720}?

\item How is the inverse-limit construction related to the $\mathbb R$-tree structure established in \cite[Proposition~4.1]{BOEDIHARDJO_LYONS_2016720}?

\item Can one modify the finite-level metric geometries so as to obtain a different inverse-limit topology, or is the resulting tree structure intrinsic to every construction compatible with the relevant path-lifting structure (cf.~\cite[Remark~8.20]{Schmeding_2022})?
\end{itemize}

The final question is motivated by the possibility that the $\mathbb R$-tree topology is an artifact of the chosen finite-level metrics rather than an intrinsic feature of the construction. If this were the case, an alternative compatible choice of metrics might induce a different inverse-limit topology under which the signature group becomes a topological group.

Our main result shows that all three questions admit a common answer, encapsulated by the following commutative diagram.

\begin{specialtheorem} \label{thm:Main-Theorem-Commutative-Diagram}
    Let $p \geq 1$ and let $V$ be a finite-dimensional real Banach space. Then the following commutative diagram holds in the category of pointed metric groups (see \cref{sec:Paths-over-pointed-metric-groups} and \cite[Sections~2--3]{LeDonneZuest_public}):
  \begin{equation} \label{eq:Diagram-p-variation-LZ}
            \begin{tikzcd}
            {\begin{tabular}{@{}l@{}}Reduced $p$-rough \\ path group\end{tabular}} \arrow[rr, "{\begin{tabular}{@{}l@{}}\textbf{Lyons lift}\end{tabular},\sim}"]\arrow[dd, "\sim"] \arrow[rrdd, "{\begin{tabular}{@{}l@{}} \textbf{Signature} \end{tabular},\sim}" description] & {} & {\begin{tabular}{@{}l@{}}Injective paths of \\ finite $p$-variation\end{tabular}} \arrow[dd, "{\begin{tabular}{@{}l@{}}
  \textbf{Terminal time} \\
  \textbf{evaluation}
\end{tabular},\sim}"] \\
            & \\
            {\begin{tabular}{@{}l@{}}Inverse limit of \\ nilpotent Lie groups \end{tabular}} \arrow[rr, "{\sim}"] & {} & {\begin{tabular}{@{}l@{}}$p$-Signature group\end{tabular}}
        \end{tikzcd}
    \end{equation}
\end{specialtheorem}

The relevant category and the objects appearing in the diagram are introduced in the subsequent sections. A fully precise version of Diagram~\eqref{eq:Diagram-p-variation-LZ}, using the notation of \cref{tab:Notation}, is given in \eqref{eq:Diagram-p-variation-LZ-precise}.

Diagram~\eqref{eq:Diagram-p-variation-LZ} combines the uniqueness results of \cite{HamblyLyonsUniqueness_2010,BOEDIHARDJO_LYONS_2016720}, represented by the upper-right triangle, with the inverse-limit construction of \cite{LeDonneZuest_public} (for the case $p = 1$), represented by the lower-left triangle.

The principal contributions of this paper are the following:
\begin{enumerate}
    \item We extend the inverse-limit construction of Le Donne--Züst from rectifiable paths to weakly geometric rough paths of finite \(p\)-variation for every \(p\ge1\).

    \item We investigate the emergence of the associated $\mathbb R$-tree from a topological perspective and relate it to the corresponding metric-geometric literature.

    \item We establish the corresponding approximation result \cite[Corollary 4.5]{LeDonneZuest_public} of Le Donne and Züst in finite $p$-variation for every $p \geq 1$.

    \item We determine the extent to which the construction depends on the choice of finite-level metrics, proving that the resulting tree structure is already intrinsic at the underlying topological level.
\end{enumerate}

As a consequence, the inverse-limit topology necessarily carries an intrinsic $\mathbb R$-tree structure. Combined with the \emph{\hyperref[thm:NoGo-Theorem]{No-Go Theorem~\ref*{thm:NoGo-Theorem}}} (\cite[Theorem~5.1]{LeDonneZuest_public}), this immediately yields the following corollary.

\begin{specialcorollary}[ref. \protect{\cref{cor:SGRtree}}]
Consider Diagram~\eqref{eq:Diagram-p-variation-LZ}, and assume that the underlying real Banach space has dimension at least two. Then the topology induced on the $p$-signature group (cf.~\eqref{eq:Signature-Group-Definition}) by the metric identifications in the diagram is that of a topological tree. Consequently, the $p$-signature group is not a topological group.
\end{specialcorollary}

All morphisms appearing in Diagram~\eqref{eq:Diagram-p-variation-LZ} are explicit and introduced throughout the paper (see also \cref{tab:Notation} for a first exposition). Their identification as isomorphisms in the category of pointed metric groups is established in \cref{prop:Isomorphisms,prop:Kernel-signature-characterisation,lem:Image-of-ev1}. In particular, the diagram makes the persistence of the underlying tree topology transparent.

Furthermore, as in \cite{LeDonneZuest_public}, our construction yields an approximation theorem (ref.~\cref{cor:LeDonne-Cor}) for homogeneous metrics. It shows that every continuous path in the free nilpotent Lie group of step $\floor{p}$ can be approximated arbitrarily well by the projection of a Hölder geodesic with respect to $p$-variation (see \cref{def:p-geodesic}) in the $p$-variation geometry on $G^k(V)$, where $k\geq\floor{p}$. By interpolation, this further implies two approximation results for finite-variation paths; see \cref{cor:q-variation-approximation-via-p-geodesics,cor:cor:q-variation-approximation-via-1-geodesics}.

We thereby answer the three questions above by establishing two fundamental results.

First, we identify the weakly geometric signature group ($p \geq 1$) with an inverse limit of free nilpotent Lie groups endowed with Carnot--Carathéodory geometry. Using the topological characterization of $\mathbb R$-trees due to \cite{Mayer1990_topological_characterization_of_R_trees}, recalled in \cref{sec:Paths-over-pointed-metric-groups}, we prove that this inverse limit is itself a topological tree. \\

\paragraph{\textbf{Outline.}}

\cref{sec:Paths-over-pointed-metric-groups} introduces the fundamental concepts from topology, metric geometry, and $p$-variation that will be used throughout the remainder of the work. After briefly introducing the categories of interest, we study the construction of the reduced path group from a purely topological and metric-geometric perspective. In doing so, we highlight several obstacles that naturally lead us to rough path theory as the appropriate framework in which to understand the reduced path group. Readers familiar with $\R$-trees and $p$-variation may skip this section.

\cref{sec:Rough-Path-Theory} provides a survey of the necessary tools of rough path theory and establishes the right-hand side of the diagram \eqref{eq:Diagram-p-variation-LZ}. We introduce the signature group together with the main results of \cite{BOEDIHARDJO_LYONS_2016720}, while developing an abstract  $p$-variation version of \cite[Sections 2 and 3]{LeDonneZuest_public}, for which rough path theory provides both motivation and a natural application. In addition, we use the perspective developed in \cref{sec:Paths-over-pointed-metric-groups} to reinterpret several statements of \cite{BOEDIHARDJO_LYONS_2016720} from a metric-geometric viewpoint. This perspective makes the underlying tree topology particularly transparent. We conclude the section by deriving further consequences concerning the continuity of the signature with respect to the topologies considered in \cite{LeDonneZuest_public} and \cite{cass2024topologiesunparameterisedroughpath}.

\cref{sec:Canonical-metric-relations-goals} begins with the definition of the generalised $d_k^p\textnormal{-}\CC$ Carnot--Carathéodory-type metrics. We establish analogues of the results in \cite[Section 7]{Friz_Victoir_2010_fullbook}, which allow us to recover the results of \cite[Section 4]{LeDonneZuest_public}. These results not only establish the left-hand side of the diagram but also yield the corresponding approximation corollaries. Furthermore, we strengthen the connection between \cite{BOEDIHARDJO_LYONS_2016720} and \cite{LeDonneZuest_public} by discussing the $\R$-tree structures in light of the geometric perspective developed in \cref{sec:Paths-over-pointed-metric-groups}. We also provide, in this setting, an answer to the question raised in \cite[Remark 8.20]{Schmeding_2022}.

Finally, \cref{sec:Character-Groups} provides an exposition of the generalisation of the preceding results to character groups of suitable Hopf algebras. The previous arguments extend to this setting in a largely straightforward manner. In particular, since free nilpotent Lie groups of finite step are precisely the truncated character groups of the shuffle algebra, we obtain the same inverse system, together with its corresponding limit, as in the case of Carnot groups. This allows us to extend the diagram \eqref{eq:Diagram-p-variation-LZ} to this setting by replacing its objects accordingly.

\subsection{Notation}

We denote by $\N := \{1,2,3,...\}$ the set of natural numbers and set $\N_0 := \N \cup\{0\}$. For $p \in [1,\infty)$ let $\floor{p}$ be the largest $m \in \N$ such that $m\leq p$. We abbreviate $\N_p := \N \cap [{\floor{p}},\infty)$. To clarify, note that $\N_p$ is just a convenience as $\N$ is in bijection to $\N_p$ by the index shift $k \mapsto {\floor{p}} - 1 + k$ and hence one may at any point re-index during the discussion. Given two topological spaces $E$ and $F$, we consider their Cartesian product $E \times F$ to be equipped with the product topology and any $S \subseteq E\times F$ to be taken with its respective subspace topology unless stated otherwise. \\

\begin{notationtable}
  \notation{((G,e),d) , \ (G,d)}{Pointed metric group with abbreviation}{\hyperref[not:Metric-Group]{p.~\pageref*{not:Metric-Group}}}
  \notation{\sim}{Tree-like equivalence}{\cref{def:Tree-like-equivalence}}
  \notation{[\cdot]_\tau}{Equivalence notation (if $\sim$ is equivalence relation)}{}
  \notation{[o]_\tau}{Unit w.r.t.~$\sim$ with $o$ as the constant path at the underlying group unit}{\hyperref[not:RPG-unit]{p.~\pageref*{not:RPG-unit}}}
  \notation{X^\tau}{Tree-reduced path representative of $[X]_\tau$}{}
  \notation{\rho^p_k}{Max type metric}{\eqref{eq:Max-metric-for-any-p}}
  \notation{WG\Omega^p}{The set of weakly geometric $p$-rough paths w.r.t.~$\rho^p_k$}{\hyperref[not:Weakly-geometric-RP]{p.~\pageref*{not:Weakly-geometric-RP}}}
  \notation{\mathcal{C}^p}{Reduced path group \(WG\Omega^p/\sim\)}{}
  \notation{\delta^p}{Tree topology metric on \(\mathcal{C}^p\)}{}
  \notation{G_{\rho^p\textnormal{-}p.r.c.}}{Codomain of the signature under the choice of \(\rho^p=(\rho^p_k)_{k \in \N_p}\)}{Ref. \eqref{eq:Codomain-Exponential-decay}}
  \notation{\mathcal{I}^p}{Injective path group on \(G_{\rho^p\textnormal{-}p.r.c.}\)}{}
  \notation{[O]_\circlearrowleft}{Unit w.r.t.~loop erasure with $O$ as the constant path at unit of the group-like elements}{\cref{rem:Injectivity-reason}}
  \notation{D^p}{Tree topology metric on \(\cI^p\)}{\cref{lem:Image-of-ev1}}
  \notation{G^\infty_p(V)}{Inverse limit group of free real nilpotent groups}{\eqref{eq:Chain-for-any-p}}
  \notation{1_\infty}{Unit of the inverse limit group}{\eqref{eq:Group-Structure}}
  \notation{d_{\infty}^{p\textnormal{-}\CC}}{Inverse limit metric w.r.t.~the $p$-$\CC$ analogue}{\eqref{eq:Inverse-Limit-Metric}}
  \notation{SG^p}{Signature group for a fixed $p\geq 1$}{\eqref{eq:Signature-Group-Definition}}
  \notation{1}{(Group) Unit of the group-like elements}{\hyperref[not:Uni-Grouplike-Elements]{p.~\pageref*{not:Uni-Grouplike-Elements}}}
  \notation{d^p}{Tree metric on the signature group}{\eqref{eq:Inverse-Metric}}
\end{notationtable}

\section*{Acknowledgement}

Felix Medwed and Sylvie Paycha gratefully acknowledge funding by the Deutsche Forschungsgemeinschaft (DFG, German Research Foundation) – CRC/TRR 388 "Rough Analysis, Stochastic Dynamics and Related Fields" – Project ID 516748464. Felix Medwed thanks the NTNU IMF for hospitality during the author's research stay in November 2025. The authors would like to thank Peter K.~Friz for many discussions related to this work. Alexander Schmeding thanks Nikolas Tapia for enlightening comments on Hopf algebra valued rough paths.

\section{Pointed path space}

\label{sec:Paths-over-pointed-metric-groups}

We assume some familiarity with \cite{LeDonneZuest_public,Lyons1998,BOEDIHARDJO_LYONS_2016720} as well as \cite[Chapter 7]{Friz_Victoir_2010_fullbook} and \cite[Chapter 8]{Schmeding_2022}. Since the first goal includes the $p$-variation analogue of \cite[Section 2 and 3]{LeDonneZuest_public}, we provide a reminder of the necessary material following \cite[Chapter 5]{Friz_Victoir_2010_fullbook} in the abstract setting. \\

\paragraph{\textbf{Categorical setup.}} Let $(E,d)$ be a metric space.
Given a chosen point $e \in E$, termed \emph{basepoint}, we call $(E,e)$ a pointed space and $((E,e),d)$ a \emph{pointed metric space}. If $E$ carries a group structure, we write $G$ instead. 
For any $g$ and $h$ in $G$, we denote their product by $gh$ and define the left translation $L_g \colon h \mapsto gh$ and the right translation $R_g \colon h\mapsto hg$.
Since the group acts transitively on itself, without loss of generality, we choose the basepoint to be the identity $e \in G$ and we call $(G,e)$ a pointed group. We shall only consider pointed groups $(G,e)$, unless stated otherwise. We call a pointed group $G$ equipped with a metric $d$ for which $L_g$ is an isometry onto itself a \emph{pointed metric group} $(G,d)$ and $d$ is called \emph{left-invariant}\phantomsection\label{not:Metric-Group}. Since $L_g$ is an isometry the map $g \mapsto d(e,g)$ fully determines $d$ and we set $d(g,h) := d(e,g^{-1}h)$ (comp.~\cite[Remark 8.20]{Schmeding_2022}). Left-invariance further implies the symmetry $d(e,g^{-1}) = d(e,g)$.  
Similarly, replacing $d$ in the couple $((E,e),d)$ (resp.~$(G,d)$) with a general (not necessarily metrisable) topology $\mathcal{T}$, we define a \emph{pointed topological space} $((E,e),\mathcal{T})$ and \emph{pointed topological group} $(G,\mathcal{T})$ analogously. 
The group operations of a (pointed) topological group are taken to be continuous w.r.t.~this topology. \\

To define the categories of pointed metric spaces $\textbf{Met}_\star$, pointed topological spaces $\textbf{Top}_\star$, pointed metric groups $\textbf{MetGrp}_\star$ and pointed topological groups $\textbf{TopGrp}_\star$ it remains to specify the morphisms. For $\textbf{Met}_\star$ these shall be given as base point preserving\footnote{A map $\pi\colon (E,e) \to (F,f)$ between pointed spaces is base point preserving if $\pi(e)=f$.} $1$-Lipschitz maps, for $\textbf{Top}_\star$ the morphisms are basepoint preserving continuous maps, for $\textbf{MetGrp}_\star$ the morphisms are $1$-Lipschitz group homomorphisms and for $\textbf{TopGrp}_\star$ the morphisms are continuous group homomorphisms. Recall that group homomorphisms automatically preserve the identity and hence the base point. Note that the identity map for each object and composition of morphisms exist and are well-defined. \\
Recalling that isometries are morphisms of a category which have a left and right inverse w.r.t.~the composition of morphisms, they are given for $\textbf{MetGrp}_\star$ by surjective isometries which are also group homomorphisms. \\

Recall that a metric space $(E,d)$ is
\begin{enumerate}
    \item \emph{geodesic} if for every two points $x,y \in E$ there exists a function $X \colon [0,1] \to E$ with $X_0 = x$ and $X_1 = y$ termed a \emph{geodesic}, such that $d(X_s,X_t) = |t-s| d(x,y)$ for any $s,t \in [0,1]$,
    \item \emph{proper} if all closed balls
\begin{equation*}
    \bar{B}_d(x,r) = \{ y \in E \mid d(x,y) \leq r \}
\end{equation*}
in $E$ are compact.
\end{enumerate}
Any proper metric space is complete, that is, all Cauchy sequences converge. \\
Consider the functor $\mathcal{F} \colon\textbf{Met}_\star \to \textbf{Top}_\star$ which assigns to each pointed metric space the pointed topological space given by the same set and basepoint together with the induced open ball topology. Then the obtained topological spaces under $\mathcal{F}$ are by construction metrisable Hausdorff topological spaces. We shall abbreviate the application of $\mathcal{F}$ to a (pointed) metric space by saying the (pointed) metric space is \emph{considered} as a topological space. One might intuitively expect the existence of a functor $\textbf{MetGrp}_\star \to \textbf{TopGrp}_\star$ such that one may consider any pointed metric group as a pointed topological group. However, the NoGo-Theorem (\cite[Theorem 5.1]{LeDonneZuest_public} or see \cref{thm:NoGo-Theorem} later on) shows that such a functor cannot exist. In other words, the group operations of a metric group do not need to be continuous w.r.t.~the open ball topology if the underlying metric space is considered as a topological space. \\

\paragraph{\textbf{Topological preliminaries.}} Let $E$ be a topological space. We call any continuous function $X \colon [0,1] \to E$ a path. If  $X$ is  a topological embedding, then $X$ is termed an \emph{arc}. 
Throughout, we shall loosely refer to $X$ itself or its image as an arc, which corresponds to 
consider parametrised or unparametrised paths. 
A topological space $E$ is \emph{pathwise} (resp.~\emph{arcwise}) \emph{connected} if for any two distinct points $x,y$ in $E$ there exists a path (resp.~arc) $X \colon [0,1]\to E$ with $X_0 = x$ and $X_1 = y$.
\begin{theorem}[comp.~\protect{\cite{borger1992make,Brazas_2024}}] \label{thm:loop-erasure}
    Let $E$ be a Hausdorff space. If $E$ is pathwise connected, then $E$ is arcwise connected. In fact, for any continuous $X \colon [0,1] \to E$ with $X_0 \neq X_1$, there exist a closed subset $A \subset [0,1]$, a continuous and order-preserving map $q \colon [0,1] \to [0,1]$, and an injective continuous map $Y \colon [0,1] \to E$ with the following properties:
    \begin{enumerate}[label = (\roman*)]
        \item $Y_0 = X_0, Y_1 = X_1.$ 
        \item $Y \circ q|_A = X|_A$.
        \item $q|_A \colon A \to [0,1]$ is surjective.
    \end{enumerate}
\end{theorem}
Thus, in a Hausdorff space, a path with distinct endpoints always contains an injective representative with distinct endpoints. This can be achieved by erasing loops, ref.~\cite{Brazas_2024}, however the resulting arc is not unique in general.
\begin{remark}
    To talk about injective paths between two points we need to assert that the latter are distinct. 
    However, without loss of generality we can relax  this requirement by assigning to the case $X_0=X_1$ the constant path $t \mapsto X_0$ which we refer to as the trivial case.
\end{remark}

\paragraph{\textbf{Topological trees}}. A topological space $E$ can further be equipped with special path properties:
\begin{itemize}
    \item If for any two distinct points the connecting arc is unique, then we term $E$ \emph{uniquely arcwise connected} (UAC) and denote the image of the arc connecting $x,y \in E$ by $[x,y]$.
    \item If for any $x$ in $E$ and neighbourhood $U$ of $x$ there exists an arcwise connected open neighbourhood $V$ contained in $U$, then $E$ is called \emph{locally arcwise connected} (LAC).
\end{itemize}
A \emph{topological tree} is a metrisable topological space that is UAC and LAC (ref. \cite[Section 4]{LeDonneZuest_public}). We call an $\mathbb R$-tree a geodesic metric space that, considered as a topological space, is UAC\footnote{It is common to strengthen the definition of arc in the metric setting to mean an isometric map $X \colon [0,1] \to E$. However, we will refrain from doing so and add the extra geodesic requirement as in \cite{LeDonneZuest_public}.}.

\begin{theorem}[comp.~\protect{\cite[Theorem 5.1]{Mayer1990_topological_characterization_of_R_trees}}] \label{thm:Toptree-Rtree}
    A topological space is a topological tree \textit{iff} it is homeomorphic to an $\R$-tree.
\end{theorem}
\begin{theorem}[NoGo-Theorem; comp.~\protect{\cite[Theorem 5.1]{LeDonneZuest_public}}] \label{thm:NoGo-Theorem}
    A topological tree carrying a group structure, such that all left and right translations are continuous, and containing more than one point is homeomorphic to $\R$.
\end{theorem}
\begin{remark}
   The NoGo-Theorem shows that the metric-topology assignment does not in general define a functor $\mathbf{MetGrp}_\star\longrightarrow \mathbf{TopGrp}_\star.$
\end{remark}
A \emph{rooted topological tree (resp.~$\mathbb R$-tree)} is any topological tree (resp.~$\R$-tree) $E$ with a chosen basepoint $e \in E$, called the \emph{root}, which we take to be the basepoint for any pointed space. We define the branch point $x\land_e y$ w.r.t.~the root by $[e,x] \cap [e,y] = [e,x\land_e y]$ which is well defined by the UAC property. Let $X \colon [0,1] \to E$ be a homeomorphism with image $[e,x]$ and $Y\colon [0,1] \to E$ a homeomorphism with image $[e,y]$, then similarly we define the corresponding \emph{branch arc}, written as $X \land_e Y$ as a homeomorphism with image $[e,x \land_e y]$.

\begin{theorem}[comp.~\protect{\cite[Theorem 5.1]{Mayer1990_topological_characterization_of_R_trees}}] \label{thm:Height-Function-existence-Rtree-Toptree}
    Let $E$ be a rooted topological tree with root $e$. Then there exists a function
    \begin{equation*}
        H_e\colon E\to [0,\infty),
        \qquad H_e(e)=0,
    \end{equation*}
    which is continuous and strictly increasing along every non-trivial branch issuing from $e$, such that
    \begin{equation*}
        d_{H_e}(x,y) := H_e(x)+H_e(y)-2H_e(x\land_e y)
    \end{equation*}
    defines a metric compatible with the topology of $E$, and $((E,e),d_{H_e})$ is a rooted $\R$-tree.
    We call $H_e$ a \emph{height function} with respect to $e$.
\end{theorem}
If  $E$ carries a group structure, such that $(E,d)$ is a pointed metric group, we simply write $H$ for the height function instead as the basepoint is canonically chosen.
\begin{remark}
    \begin{enumerate}
        \item The choice of the height function is not unique.
        \item The unit interval $[0,1]$ with its metric induced by the absolute value is an $\R$-tree, the trivial $\R$-tree.
    \end{enumerate}
\end{remark}

\paragraph{\textbf{Operations on marked paths.}} Let $((E,e),d)$ be a pointed metric space. If a continuous path $X \colon [0,1] \to E$ starts at the basepoint, i.e.~$X_0=e$, then $X$ is referred to as a \emph{pointed path}. We write $\mathcal P_e(E)$ for the set of all pointed paths in $E$, and simply write $\mathcal P_e$ if the target space is clear from the context. Let $(G,d)$ be a pointed metric group with identity $e$. Then for $X,Y \in \mathcal{P}_e(G)$ we define the path concatenation
\begin{equation*} 
    (X\sqcup Y)_t := \left\{\begin{array}{cc}
        X_{2t}, & t\in[0,1/2] \\
        L_{X_{1}}(Y_{2t-1}), & t \in [1/2,1]
    \end{array}\right.
\end{equation*}
and the path reversal by
\begin{equation*}
    \rev{X}_t := L_{(X_{1})^{-1}}( X_{1-t}).
\end{equation*}
The paths $X \sqcup Y$ and $\rev{X}$ above are again elements of $\mathcal{P}_e(G)$. Note, that the start- and endpoint gluing in the definition of $\sqcup$ of pointed paths demands the existence of a left translation to be well defined between any two pointed paths.\\
Then for any $X,Y,Z\in \mathcal P_e(G)$ we can compute
\begin{equation*}
    \rev{\rev{X}} = X,
    \qquad
    \overleftarrow{X\sqcup Y}=\rev{Y}\sqcup \rev{X}.
\end{equation*}
Furthermore, concatenation fulfils
\begin{equation} \label{eq:Associativity-Failure}
    (X\sqcup Y)\sqcup Z
    =
    \bigl(X\sqcup (Y\sqcup Z)\bigr)\circ \sigma,
\end{equation}
where $\sigma\colon [0,1]\to [0,1]$ is an \emph{ordered reparametrisation}, which is a continuous surjective order-preserving function. A possible instance fulfilling \eqref{eq:Associativity-Failure} is
\begin{equation*}
    \sigma(t)
    :=
    \begin{cases}
        2t, & t\in [0,1/4],\\
        t+1/4, & t\in [1/4,1/2],\\
        t/2+1/2, & t\in [1/2,1]
    \end{cases}
\end{equation*}
which measures the failure of associativity. \\

\paragraph{\textbf{Tree-like equivalence.}}
Due to the algebraic properties of $\sqcup$ and $\rev{\cdot}$, a natural question is if there exists a substructure, i.e.~a quotient of $\mathcal{P}_e$ with corresponding congruence relation $\sim$, such that $\sim$ is an equivalence relation on the set of pointed paths and the induced concatenation and reversal constitute genuine group operations. Supposing such a $\sim$ exists, let us denote the corresponding equivalence classes by $[X]$ and the identity by $[o]$, then $[X] = [Y]$ for any two paths $X,Y \in \mathcal{P}_e$ implies $[\rev{X} \sqcup Y] = [o]$ up to the choice of left and right multiplication. After observing that for any arc $X$, the path $\rev{X} \sqcup X$ is traversed first along $\rev{X}$ and then back along $X$ to $e$, notice that $[o]$ has to be the equivalence class of the constant path $o \colon t \mapsto e$. In addition, we can factor $\rev{X} \sqcup X = \psi \circ \phi$ through the trivial $\R$-tree. This observation  can be formalised as follows: \phantomsection\label{not:RPG-unit}
\begin{definition}[Comp. \protect{\cite[Definition 1.1]{BOEDIHARDJO_LYONS_2016720}}] \label{def:Tree-like-equivalence}
    A pointed path $X$ is called \emph{tree-like} if there exists an $\R$-tree $\tau$, a continuous map $\phi\colon [0,1] \to \tau$ with $\phi(0) = \phi(1)$ and a map $\psi\colon \tau \to G$, such that $X = \psi \circ \phi$. Then if $X$ and $Y$ are two pointed paths, we define the relation
    \begin{equation}
        X \sim Y : \Leftrightarrow \rev{X} \sqcup Y \textrm{ is tree-like}.
    \end{equation}
    We call any $X$ and $Y$ fulfilling the relation $\sim$ \emph{tree-like equivalent}. A path $X$ is \emph{tree-reduced} if there exists no subinterval $[s,t] \subseteq [0,1]$ s.t.~$X|_{[s,t]}$ is tree-like.
\end{definition}
By the arguments found in \cite{HamblyLyonsUniqueness_2010} and \cite{lee2020pathsignaturesliegroups} it is immediate by definition of tree-like equivalence that $X \sim X$ and $X \sim Y \Leftrightarrow Y \sim X$ hold. In addition, $X \sim X \circ \sigma$ is true for any ordered reparametrisation $\sigma$ and hence the induced concatenation on the quotient will be associative. Furthermore, if $\sim$ exists as an equivalence relation on $\mathcal{P}_e$, the induced operations are well-defined on the quotient and $\sim$ therefore defines a congruence relation with respect to the induced path concatenation and reversal. Hence, the only remaining issue is to check the transitivity property of the relation (ref.~\cite{HamblyLyonsUniqueness_2010}).
\begin{counterexample} \label{counterexample:Regularity-of-paths-excludes-tree-like-insertions}
    The counterexample constructed in \cite{brazas2026treeliketransitiverelationpaths} over the group $(\R^2,+)$ shows by a limit argument that $\sim$ defined as in \cref{def:Tree-like-equivalence} is not transitive. This is achieved by constructing (pointed) paths $X$ and $Z$ that form the edges of a isosceles right triangle, where $X$ is the path along the adjacent and opposite of the hypotenuses and $Z$ follows the hypotenuse. Both points start from a corner attached to the hypotenuse with the same orientation. Under successive tree-like insertions\footnote{Let $X\colon [0,1] \to (G,d)$ be a continuous path. A $C$-tree-like insertion at $t \in [0,1]$ is $X \sim A \sqcup B \mapsto (A \sqcup (\rev{C} \sqcup C)) \sqcup B$ where $\rev{C} \sqcup C$ is a tree-like piece. Here, $A,B \colon [0,1] \to (G,d)$ are be taken to be $A\colon s \mapsto X_{st}$ and $B\colon s \mapsto X_{t + (1-t) s}$. Due to tree-like equivalence the concatenation with $C$ order does not matter.} into $X$ and $Z$, {\color{purple} one can} construct paths $Y^{X,{(n)}}$ and $Y^{Z,(n)}$ which factor through an $\R$-tree and agree in the appropriate limit, giving a continuous pointed path $Y$. By construction one then gets $X \sim Y$ and $Y \sim Z$, but $X \nsim Z$ making the transitivity property over pointed paths in $(\R^2,+)$ false.
\end{counterexample}
The counterexample exploits (countably) infinitely many tree-like insertions, which are untracked if one only imposes the continuity property on pointed paths. 
Therefore, an obvious way to exclude counterexamples of this kind is to consider a subset of $\mathcal{P}_e$ subject to regularity properties, which are able to quantify and ultimately prevent such counterexamples from being constructed. One such property is the roughness of paths measured in $p$-variation in the sense of Wiener. \\

\paragraph{\textbf{$p$-Variation of paths.}} Following \cite[Chapter 5]{Friz_Victoir_2010_fullbook}, let $X$ be a (pointed) path with values in a metric space $(E,d)$ and fix a $p \geq 1$ throughout, then the \emph{$p$-variation w.r.t.~$d$} restricted to $[s,t] \subseteq [0,1]$ is defined by
\begin{equation} \label{eq:General-p-variation-for-pointed-paths}
    \pvdi{X}{p}{d}{[s,t]} := \pvd{X|_{[s,t]}}{p}{d} := \left( \sup_{\cD: s = s_0 < s_1 < ... < s_n = t} \sum_{j = 1}^{n-1} d(X_{s_{j-1}},X_{s_{j}})^p \right)^{1/p}.
\end{equation}
Here $\mathcal{D}$ is any partition of the interval $[s,t]$ into subintervals $[s_i,s_{i+1}]$.
Should $[s,t] = [0,1]$ hold, then we simply write $\pvd{X}{p}{d}$ and if $d$ is clear from context, we will drop it. Should $\pvd{X}{p}{d} < \infty$, then we say that $X$ is of \emph{finite $p$-variation w.r.t.~$d$} or of finite $p$-$d$-variation. If $p=1$ we say that $X$ is \emph{rectifiable} or of \emph{bounded variation}. Due to \cite[Proposition 5.9]{Friz_Victoir_2010_fullbook},  replacing $d(X_{s_{j-1}},X_{s_j})$ in \eqref{eq:General-p-variation-for-pointed-paths} with $\pvdi{X}{p}{d}{[s_{j-1},s_j]}$ for a path $X$ of finite $p$-$d$-variation will result again in $\pvdi{X}{d}{p}{[s,t]}$ for any $[s,t] \subseteq [0,1]$. 
Furthermore, for any $s,u,t \in [0,1]$ with $s \leq u \leq t$, we have by \cite[Proposition 5.8]{Friz_Victoir_2010_fullbook}
\begin{equation} \label{eq:p-Variation-inequalities-super}
    \pvdi{X}{p}{d}{[s,t]}^p \geq \pvdi{X}{p}{d}{[s,u]}^p + \pvdi{X}{p}{d}{[u,t]}^p,
\end{equation}
making the map $(s,t) \mapsto \pvdi{X}{p}{d}{[s,t]}^p$ super-additive. On the contrary, by the Minkowski-inequality
\begin{equation} \label{eq:p-Variation-inequalities-sub}
    \pvdi{X}{p}{d}{[s,t]} \leq \pvdi{X}{p}{d}{[s,u]} + \pvdi{X}{p}{d}{[u,t]}
\end{equation}
and therefore $(s,t) \mapsto \pvdi{X}{p}{d}{[s,t]}$ is sub-additive.
In the case $p=1$, the above inequalities \eqref{eq:p-Variation-inequalities-super} and \eqref{eq:p-Variation-inequalities-sub} coincide and become equalities.\\
Taking any $0 \leq s \leq u \leq t \leq w \leq 1$, it holds that
\begin{equation*}
    \pvdi{X}{p}{d}{[u,t]} \leq \pvdi{X}{p}{d}{[s,w]}.
\end{equation*}
and the function $t \mapsto \pvdi{X}{p}{d}{[s,t]}$ is continuous and monotonically increasing.
Moreover, by \cite[Lemma 5.12]{Friz_Victoir_2010_fullbook} the map $X \mapsto \pvdi{X}{p}{d}{[s,t]}$ is lower semi-continuous w.r.t.~pointwise convergence in $d$. \\
The set of all pointed paths $X$ into a pointed metric group $((E,e),d)$ of finite $p$-$d$-variation will be denoted by $C^{p\textnormal{-}\var}_d([0,1],E) \subset \mathcal{P}_e(E)$. 
\begin{remark}
    Recall the setting of \cref{counterexample:Regularity-of-paths-excludes-tree-like-insertions} in \cite{brazas2026treeliketransitiverelationpaths} and suppose that $Y^{Z,(n)}$ is the $n$th iteration of the path constructed for the hypotenuse $Y^{Z,(0)} := Z$ with initial $p$-variation $\pvd{Z}{p}{} =: \ell > 0$. At each step $n$ of the construction of the $\R$-tree for $Y^{Z,(n)}$ and for each arc with length $\ell_n$, one inserts a new vertex at half the length $\ell_{n+1}=\ell_n/2$, dividing each arc into two arcs of half the previous length. Then for any new inserted vertex, one attaches two new arcs of length $\ell_{n+1}$ ending at two new vertices. In this way an arc of length $\ell_n$ becomes $4$ new arcs, each of length $\ell_{n+1}=\ell_n/2$. \\
    Hence, we can for $p < 2$ conclude, by picking the partition $\{t_1,...,t_m\}$ such that each $Y_{t_j}^{Z,(n)}$ is a corner point, that
    \begin{equation*}
         \pv{Y^{Z,(n)}}{p}^p \geq 4^n \left(\frac{\ell}{2^n} \right)^p = \ell^p 2^{(2-p)n}\rightarrow \infty, \quad \text{ for } n \rightarrow \infty
    \end{equation*}
    since the sequence of the lower bound is monotonically increasing. To treat the case $p \geq 2$, we first notice that $[0,1]$ is divided into $4^n$ intervals $I_{n,k}, k\in [[1, 4^n]]$ of length $4^{-n}$ and we note that $\operatorname{diam}(Y(I_{n,k})) \leq \ell 2^{-n}$ by construction. Take any $0 \leq s < t \leq 1$ and choose $n$ such that $4^{-(n+1)} < |t-s| \le 4^{-n}$. This gives rise to three subsets $I_{n,k'}$, $I_{n,k'+1}$ and $I_{n,k'+2}$, such that $[s,t] \subseteq I_{n,k'} \cup I_{n,k'+1} \cup I_{n,k'+2}$. Thus, $\|Y_t - Y_s\| \leq \sum_{a = 0}^2 \operatorname{diam}(Y(I_{n,k'+ a}))$ and $\|Y_t - Y_s\| \leq 3 \ell 2^{-n} < 6 \ell |t-s|^{1/2}$ by the choice of $s$ and $t$. Consequently, for any partition
    \begin{equation*}
        \sum_{j = 1}^m \| Y^{Z,(n)}_{t_j} - Y^{Z,(n)}_{t_{j-1}}\|^p \leq (6 \ell)^p \sum_{j = 1}^m |t-s|^{p/2} \leq (6 \ell)^p.
    \end{equation*}
    Taking the supremum and limit of $Y^{Z,(n)} \to Y$ gives the finite $p$-variation of the limiting path $Y$. \\
A necessary condition for the counterexample to be excluded in $(\R^2,+,d_{\text{euclid}})$, where $d_{\text{euclid}}$ is the euclidean length, is therefore $1 \leq p<2$. That this regularity is also sufficient to give transitivity is the main result of \cite{HamblyLyonsUniqueness_2010}.
\end{remark}
Before proceeding, let us recall more on $p$-variation.
\begin{definition}[Control function, \protect{\cite[comp.~Definition 1.17]{Friz_Victoir_2010_fullbook}}]
    A \emph{control (function)} is a two-parameter map $$\control \colon \Delta = \{ (s,t) \in [0,1]^2 \mid 0 \leq s \leq t \leq 1 \} \to [0,\infty),$$ where
    \begin{enumerate}
        \item $\control$ is super-additive, $\forall \, 0 \leq s \leq u \leq t \leq 1: \control(s,u) + \control(u,t) \leq \control(s,t)$,
        \item $\control$ is continuous.
    \end{enumerate}
    A (pointed) path $X$ of finite $p$-$d$-variation is said to be controlled by $\control$ if
    \begin{equation*}
        \pvdi{X}{p}{d}{[s,t]}^p \leq \control(s,t).
    \end{equation*}
\end{definition}
\begin{remark} \label{rem:more-on-p-variation}
    \begin{enumerate}
        \item If $\control$ is a control function, also $C \control$ for any $C > 0$ is a control function, where the constant can be absorbed into the definition of $\control$. 
        \item By \cite[Proposition 5.8]{Friz_Victoir_2010_fullbook}, if $X$ has finite $p$-variation, then $\control_X(s,t) := \pvdi{X}{p}{d}{[s,t]}^p$ for a (pointed) path $X$ is a control function.
        \item \label{rem:more-on-p-variation:Upper-bound-of-metric} If $(G,d)$ is a metric group and $X \in C^{p\textnormal{-}\var}_d([0,1],G)$ is controlled by $\control$, then by \cite[Proposition 5.10]{Friz_Victoir_2010_fullbook} for any $(s,t) \in \Delta$ the estimate
        \begin{equation} \label{eq:Increment-Estimate-Control}
            d(X_s,X_t)^p \leq \control(s,t)
        \end{equation}
        holds.
        \item For each path $X$ controlled by $\control_X$ \cite[Proposition 5.14]{Friz_Victoir_2010_fullbook} shows that there exists a reparametrisation given implicitly by
        \begin{equation*}
            \theta(t) = \frac{\control_X(0,t)}{\control_X(0,1)},
        \end{equation*}
        such that $\Tilde{X}$ defined by $X = \Tilde{X} \circ \theta$ is a Hölder path of regularity $1/p$ with Hölder constant $\control_X(0,1)^{1/p}$, if $\control_X(0,1) > 0$. If $\control_X(0,1) = 0$, then $X$ is constant. In case $p = 1$ this parametrisation is commonly referred to as arc length parametrisation. Therefore, we term the $p>1$ analogue the \emph{$1/p$-Hölder reparametrisation}.
        \item \label{rem:more-on-p-variation-nesting-of-variations} Let $p < q < \infty$ and $X$ be a path into a metric space with finite $p$-variation, then from \cite[Proposition 5.3]{Friz_Victoir_2010_fullbook} it follows that
        \begin{equation*}
            \|X\|_{q\textnormal{-}\var} \leq \|X\|_{p\textnormal{-}\var}
        \end{equation*}
        Hence any rectifiable path has finite $q$-variation for any $q \geq 1$. 
        \item Let $X,Y \in C^{p\textnormal{-}\var}_d([0,1],G)$, then by the properties of $p$-variation, we obtain
        \begin{equation*}
            \pvd{X \sqcup Y}{p}{d}^p \geq \pvd{X}{p}{d}^p + \pvd{Y}{p}{d}^p
        \end{equation*}
        and
        \begin{equation*}
            \pvd{X \sqcup Y}{p}{d} \leq \pvd{X}{p}{d} + \pvd{Y}{p}{d}.
        \end{equation*}
        Furthermore, by the symmetry of the metric $d$ follows
        \begin{equation*}
            \|\rev{X}\|_{p\textnormal{-}\var,d} = \pvd{X}{p}{d}
        \end{equation*}
    \end{enumerate}
\end{remark}
Let $(E,d)$ be a geodesic metric space and $X \colon [0,1]\to E$ be a path with $X_0 = x$ and $X_1 = y$ for two distinct points $x,y \in E$ satisfying $d(X_s,X_t) = d(x,y) |t-s|$ for any $s,t \in [0,1]$. Hence, also $d(X_0,X_1) = d(x,y)$. Then it is a fact of metric geometry that
\begin{equation} \label{eq:Geodesic-relation-of-metric}
    d(x,y) = \min_{\substack{X \colon [0,1] \to E \\ X_0 = x, \, X_1 = y}} \pvd{X}{1}{d},
\end{equation}
where the minimisation occurs due to the inequality $\pvd{Y}{1}{d} \geq d(x,y)$ for any other path $Y$ with $Y_0=x$ and $Y_1 = y$ and the attainment of the minimiser follows from $(E,d)$ being a geodesic metric space by assumption. One may then distinguish a (proper) subset $\mathcal{A}(x,y)$ which consists of uniquely defined paths between two distinct points $x$ and $y$ in $E$. We  call it \emph{the set of admissible paths connecting $x$ to $y$}. We define
\begin{equation*}
    \mathcal{A}(E) = \bigcup_{\substack{x,y \in E \\ x \neq y}} \mathcal{A}(x,y)
\end{equation*}
to be the set of all admissible paths over $E$, such that $(E,d)$ is a geodesic space w.r.t.~$\mathcal{A}(E)$, with \eqref{eq:Geodesic-relation-of-metric} now considered only over $\mathcal{A}(x,y)$ (assuming it is non-empty). One may weaken \eqref{eq:Geodesic-relation-of-metric} by imposing only the equality w.r.t.~the infimum, for which $(E,d)$ is then called a \emph{length space} where $d$ needs to be considered as an extended metric.\footnote{This is a metric $d \colon E\times E \to [0,\infty) \cup \{\infty\}$.} Since \cref{counterexample:Regularity-of-paths-excludes-tree-like-insertions} included the $p$-variation for $p>1$, one may define an analogous relation for $p$-variation, i.e.~ref.~\cite{Chistyakov1998_MinimalpVariation}.
\begin{definition} \label{def:p-geodesic}
    Let $(E,d)$ be a metric space and let $\mathcal{A}(E)$ be the set of all admissible paths of finite $p$-$d$-variation. Define
    \begin{equation*}
        \begin{aligned}
            d_p(x,y) &:= \inf_{X \in \mathcal{A}(x,y) \subset \mathcal{A}(E)} \pvd{X}{p}{d},
        \end{aligned}
    \end{equation*}
    with $d_p(x,y) = +\infty$ if $\mathcal{A}(x,y) = \varnothing$. Then a path $X^\star \in \mathcal{A}(x,y)$ attaining the infimum, such that
    \begin{equation*}
        d_p(x,y) = \pvd{X^\star}{p}{d}
    \end{equation*}
    is satisfied and $X^\star$ is parametrised to be $1/p$-Hölder, is called a \emph{Hölderian geodesic path w.r.t.~$p$-($d$-)variation}. The corresponding space is denoted by $\mathcal{A}(x,y)$. If for any two distinct points $x,y$ in $E$ there always exists a Hölderian geodesic $X^*$ w.r.t.~$p$-variation and $\mathcal{A}(x,y)$ connecting them such that $d(x,y) = d_p(x,y)$, then we call $(E,d)$ \emph{$p$-variation geodesic space w.r.t.~$\mathcal{A}(E)$}. In case $\mathcal{A}(E)$ is clear from context or just along all paths (or arcs) that connect any two points, we drop it.
\end{definition}
Note, that one may have redefined the minimisation to be over the control $\omega_X([0,1]) = \pvd{X}{p}{d}^p$ instead. The reason why $p$-variation is preferred over the control $\control_X$ is to preserve the triangle inequality as can be seen from \eqref{eq:p-Variation-inequalities-sub}, whereas the control only would fulfil \eqref{eq:p-Variation-inequalities-super}. The distinction is merely a technicality, since the function $x \mapsto x^{1/p}$ is monotonically increasing on the domain $[0,\infty)$ and hence any minimiser of the control corresponds to a minimiser of the $p$-variation. We will comment on this more in \cref{sec:Canonical-metric-relations-goals} and prove there -- for completeness -- under a specific choice of the set of all admissible paths that $d_p$ is a metric in \cref{lem:properties-of-dpCC}.
\begin{remark} \label{rem:p-geodesic}
 Let us stress that
    \begin{enumerate}
        \item for any $p > 1$ we do not recover the local additivity property for a $p$-variation geodesic space, but in general remain with the inequality
        \begin{equation*}
            d(X_s,X_t) \leq \pvdi{X}{p}{d}{[s,t]},
        \end{equation*}
        \item in case $(E,d)$ is a length space, then for any $q > 1$, we have $d(x,y) = d_q(x,y)$. Indeed, choosing a rectifiable path and noticing the lower bound coming from the trivial partition $\{0,1\}$ together with \ref{rem:more-on-p-variation-nesting-of-variations} of \cref{rem:more-on-p-variation}, we get
        \begin{equation*}
            d(x,y) \leq d_q(x,y) \leq d_1(x,y) = d(x,y).
        \end{equation*}
        Thus, since we recover $d$ itself, the induced geometry does not change.
    \end{enumerate}
\end{remark}

\paragraph{\textbf{Representations of paths.}} A direct proof of the transitivity property $\sim$ as in \cref{def:Tree-like-equivalence} for pointed paths in full generality is not known to the authors and will in general depend on the considered group and metric in the first place. However, we want to remark that the essential step to conclude transitivity is to require, for $X,Y,B \in C^{p\textnormal{-}\var}_d([0,1],G)$ and $B$ a tree-like path, that
\begin{equation} \label{eq:Reduction-Step}
    \rev{X} \sqcup B \sqcup Y \sim \rev{X} \sqcup Y
\end{equation}
holds. 
Following the convention used by Chen in his work and stated in \cite{lee2020pathsignaturesliegroups}, \eqref{eq:Reduction-Step} shall be called a \emph{reduction step}. \\
A possible workaround to prove transitivity of $\sim$ is to study homomorphisms (understood w.r.t.~to the subsequent equations)
\begin{equation*}
    \Psi \colon C^{p\textnormal{-}\var}_d([0,1],G)\to H
\end{equation*}
into a group $H$, such that
\begin{equation*}
    \ker \Psi = \{ X \in C^{p\textnormal{-}\var}_d([0,1],G)\mid X \text{ is tree-like}\}.
\end{equation*}
This is the perspective of \cite{HamblyLyonsUniqueness_2010} and is visible again e.g.~in \cite{lee2020pathsignaturesliegroups}. Indeed, if such a $\Psi$ exists, then
\begin{equation*}
    X \sim Y
    \quad \Longleftrightarrow \quad
    \Psi(\rev{X}\sqcup Y)=e_H
    \quad \Longleftrightarrow \quad
    \Psi(X)=\Psi(Y),
\end{equation*}
and transitivity follows from equality in $H$. The following explains why looking for homomorphism out of path spaces is a natural question.

Assume for the moment that $X$ is a smooth path with values in $\R^n$. Then one natural class of candidates is given by endpoints of group-valued controlled differential equations. Let $H$ be a Lie group with Lie algebra $\mathfrak h$, and let $A\colon \R^n\to \mathfrak h$ be a linear map. In left-invariant notation, consider
\begin{equation*}
   \mathrm dY_t = T_eL_{Y_t}(A(\mathrm dX_t)), \qquad Y_0=e_H.
\end{equation*}
By the elementary existence and uniqueness theorem for ODEs, the endpoint map $\operatorname{End}_A \colon X \mapsto Y_1 \in H$ is well-defined.
With this invariant convention, solving the above differential equation along $X\sqcup Y$ amounts to multiplying the endpoints in $H$, and reversal corresponds to inversion:
\begin{equation} \label{eq:Endpoint-Map-Properties}
    \operatorname{End}_A(X\sqcup Y) = \operatorname{End}_A(X)\operatorname{End}_A(Y),
    \qquad
    \operatorname{End}_A(\rev{X}) = \operatorname{End}_A(X)^{-1}.
\end{equation}
It remains however non-trivial to obtain the desired characterisation of tree-like paths by $\ker \operatorname{End}_A$.

\section{A \textquote{metric perspective} on rough path theory}

\label{sec:Rough-Path-Theory}

Rough path theory, first introduced by Terry Lyons \cite{Lyons1998}, aimed to make sense pathwise of controlled differential equations
\begin{equation} \label{eq:controlled-differential-equation-rp}
    \mathrm{d}Y_t = f(Y_t)\mathrm{d}X_t, \qquad Y_0 = \xi,
\end{equation}
for rough drivers $X$ in a Banach setting beyond the Young regime,  \cite{Young_1936_Original}. The basic insight is that, for $p \geq 2$, the path $X$ alone is not enough to determine the solution of the controlled differential equation \eqref{eq:controlled-differential-equation-rp} above. Hence, one has to prescribe additional higher order information about the increments of $X$, which in the $p=1$ variation case is given by iterated Riemann--Stieltjes integrals. However, in the $p \geq 2$ case, the integration theory needs to be generalised via sewing arguments, which in the $p < 2$ case amounts to the Young setting (ref.~\cite{Friz2020_chapter2}).

The algebraic structures underlying the higher order information for finite dimensional Banach spaces are best described via the tensor series algebra, or equivalently the space of formal power series with finitely many non-commutative indeterminates. Depending on the setting, the existence of partial integration, seen best for $p=1$, introduces symmetries on the coefficients of the formal series and is best captured using the language of character groups of graded connected Hopf algebras of finite type and their truncations, ref. \cite[Section 8.4]{Schmeding_2022}. The character groups also serve as a generalisation as one is able to encode different nestings of iterated integrals via different choices of Hopf algebras \cite[Section 8.4]{Schmeding_2022}, which change the underlying combinatorics. \\
The free (real) nilpotent (Lie) groups are the truncated character groups of the shuffle algebra \cite[Section 4]{Schmeding_2022} and are the default \emph{weakly geometric} case, which addresses the symmetries introduced by partial integration for the $p=1$ case. As they carry a natural geometric meaning \cite[Section 10]{LeDonne2025}, we shall only address them in this and the next section. \\

\paragraph{\textbf{Nilpotent Lie groups.}}
Let $V$ be a real finite-dimensional Banach space. Let $k$ in $\mathbb N$ and let $G^k(V)$ be the free real nilpotent group over $V$ of step $k$ and rank $\dim V$, considered as a subset of the truncated tensor group $T^k_1(V) = \{1\} \times \prod_{j = 1}^k V^{\otimes j}$ with induced Cauchy product $\otimes_k$ and unit $1_k$ where by convention $V^{\otimes 1} := V$, as constructed in \cite[Chapter 7]{Friz_Victoir_2010_fullbook} and \cite[Section 8.2]{Schmeding_2022} with each $V^{\otimes j}$ for $j \leq k$ taken as a subset of $T^k(V)$. Let
\begin{equation*}
    \frg^k(V)
    =
    \bigoplus_{j=1}^k \mathfrak v_j, \qquad \mathfrak v_{j+1} = [V, \mathfrak v_j], \qquad \forall j = 1,...,k-1, \qquad \mathfrak v_{k+1} = \{0\},
    \qquad
    \mathfrak v_1:=V
\end{equation*}
be the Lie algebra of $G^k(V)$ with Lie bracket given by the commutator w.r.t.~$\otimes_k$. Further, let $\exp_k \colon \frg^k(V) \to G^k(V)$ be the exponential map\phantomsection\label{not:Lie-Bracket-exponential-function} \cite[Corollary 8.3]{Schmeding_2022}. Recalling that each element $g$ in $ T^k_1(V)$ can be written as $g=(1,g^1,\ldots,g^k)$, let $\widetilde \proj_j\colon g\mapsto g^j$ be the coordinate map. For $j\leq k$ let
\begin{equation*}
    \tilde\pi_{k+1}\colon T^{k+1}_1(V)\to T^{k}_1(V), \qquad \tilde\pi^{k}_j := \tilde\pi_{j+1} \circ ... \circ \tilde\pi_{k}
    \qquad
    \tilde\pi^k_k:=\id_{T^k_1(V)},
\end{equation*}
be the canonical projections (ref. \cite[Section 8.2]{Schmeding_2022}) between the truncated tensor algebras. We then set $\pi^{k}_j := \tilde \pi|_{G^{k}(V)}$, for which $\pi^k_j(G^k(V)) = G^j(V)$ as well as $\pi^k_k:= \tilde\pi^k_k|_{G^k(V)}$ and $\proj_j = \widetilde \proj_j|_{G^k(V)}$. If $G^k(V)$ is equipped with its vector space topology, then $G^k(V)$ is a finite-dimensional Lie group.
It was shown in \cite{Hebisch1990} that each $G^k(V)$ considered as a Lie group, equipped with its natural family of dilations $(\delta_\lambda)_{\lambda>0}$, defined via
\begin{equation*}
    \proj_j(\delta_\lambda g)
    =
    \lambda^j\proj_j(g),
    \qquad
    j=1,\ldots,k,
\end{equation*}
admits a left-invariant homogeneous metric $d$, which is a metric fulfilling
\begin{equation*}
    d(\delta_\lambda g,\delta_\lambda h)
    =
    \lambda d(g,h).
\end{equation*}
The generalisation to $\lambda$ in $\R$ is also possible by requiring absolute homogeneity w.r.t.~the dilations. Notice that a homogeneous metric on the abelian group $G^1(V)$ is induced by a norm. We identify throughout $G^1(V)$ with the additive group $(V,+)$. \\

\paragraph{\textbf{Inverse limit and compatible metrics.}} We define the  inverse system (comp.~\cite[Proposition 8.8]{Schmeding_2022}), here also termed a \emph{chain},
\begin{equation} \label{eq:Chain-for-any-p}
    \begin{aligned}
    \cG_p :=
    &\bigl((G^k(V),d_k),(\pi^k_j)_{j={\floor{p}}}^k\bigr)_{k\in\N_p}, \\
    &\begin{tikzcd}
        {(G^{\floor{p}}(V),d_{\floor{p}})} & ... \arrow[l] & {(G^k(V),d_k)} \arrow[l] \arrow["\pi^k_k = \operatorname{id}_{G^k(V)}"', loop, distance=2em, in=125, out=55] & {(G^{k+1}(V),d_{k+1})} \arrow[l, "\pi^{k+1}_k"'] & ... \arrow[l]
    \end{tikzcd}
    \end{aligned}
\end{equation}
If the metrics $(d_k)_{k \in \N_p}$ are chosen such that all projections are $1$-Lipschitz, then $\cG_p$ is an inverse system in the category of pointed metric groups. We shall denote by $\met(\cG_p)$ all sequences $(d_k)_{k \in \N_p}$ of left-invariant metrics on $\cG_p$, such that the projections $(\pi^{k}_j)_{j={\floor{p}}}^k$ for all $k \in \N_p$ are $1$-Lipschitz maps. The set $\met(\cG_p)$, equipped with componentwise addition and comparison as well as multiplication by weights $(w_k)_{k \in \N_p} \subset (0,\infty)$ that preserve the $1$-Lipschitz property has an (ordered) convex cone structure. Note that the $1$-Lipschitz property is preserved if $(w_k)_{k \in \N_p}$ is monotonically increasing.

Then, by \cite[Lemma 2.2, Lemma 3.1]{LeDonneZuest_public} we can find the inverse limit of any pointed metric spaces as well as metric groups. Hence, the inverse limit $(G_p^\infty(V),d_\infty)$ of the chain $\cG_p$ with identity $1_\infty$ is explicitly given as the set of all $(g_k)_{k\in\N_p}\in\prod_{k\in\N_p}G^k(V)$ with $\pi^k_j(g_k) = g_j$ for all $k \geq j \geq \floor{p}$, with group structure
\begin{equation} \label{eq:Group-Structure}
    (g_k)_{k \in \N_p}(h_k)_{k \in \N_p}=(g_kh_k)_{k \in \N_p},
    \qquad
    (g_k)_{k \in \N_p}^{-1}=(g_k^{-1})_{k \in \N_p},
    \qquad
    1_\infty=(1_k)_{k \in \N_p}
\end{equation}
as well as the finiteness condition
\begin{equation} \label{eq:Inverse-Limit-Metric}
    d_\infty(1_\infty,(g_k)_{k \in \N_p})
    :=
    \sup_{k\in\N_p} d_k(1_k,g_k)
    =
    \lim_{k\to\infty} d_k(1_k,g_k)
    <\infty.
\end{equation}
The last equality follows from the $1$-Lipschitz property of the projections and the monotonicity of the sequence. The induced left-invariant limit metric is given by
\begin{equation*}
    d_\infty((g_k)_{k \in \N_p},(h_k)_{k \in \N_p})
    =
    d_\infty((1_k)_{k \in \N_p},(g_k^{-1} h_k)_{k \in \N_p})
\end{equation*}
Furthermore, the existence of the limit $(G^\infty_p(V),d_\infty)$ asserts the existence of canonical projection maps
\begin{equation*}
    \pi^\infty_k\colon G_p^\infty(V)\to G^k(V), \qquad (g_k)_{k \in \N_p} \mapsto g_k.
\end{equation*}
Notice that the corresponding group $G_p^\infty(V)$ differs from the group obtained from the limit object taken in the category of pointed topological groups. Indeed, while the corresponding group of the limit object in the topological groups is also a subset of the product $\prod_{k\in\N_p}G^k(V)$ where $g_j = \pi^k_j(g_k)$ for any $(g_k)_{k\in \N_p} \in \prod_{k\in\N_p}G^k(V)$, it is equipped with the product topology, which is coarser than the limiting metric topology induced by $d_\infty$. \\

\paragraph{\textbf{Group-like elements.}} The group $\Pi_{k \in \N_p} G^k(V)$ where $\pi^{k+1}_k(g_{k+1}) = g_k$ for every $k \geq \floor{p}$ and $(g_k)_{k \in \N_p}$, can be identified algebraically with the group-like elements (i.e.~ref.~\cite{Lyons1998}) of the extended tensor algebra $G((V)) \subsetneq T((V)) = \prod_{k \in \N_0} V^{\otimes k}$ \phantomsection\label{not:Group-like-Elements} equipped with its restricted Cauchy product. An explicit group isomorphism achieving this identification is given by (comp.~\cite[Section 4]{LeDonneZuest_public} for $p=1$) 
\begin{equation}\label{eq:Inverse-limit-to-product-conversion-map-p}
    \begin{array}{lccc}
        \Pi^p \colon
        &
        \prod_{k\in\N_p}G^k(V)
        &
        \to
        &
        \prod_{j\in\N_0}V^{\otimes j},
        \\[0.3em]
        &
        (g_{\floor{p}},g_{{\floor{p}}+1},\ldots)
        &
        \mapsto
        &
        (1,\proj_1(g_{\floor{p}}),\ldots,\proj_{\floor{p}}(g_{\floor{p}}),
        \proj_{{\floor{p}}+1}(g_{{\floor{p}}+1}),\ldots).
    \end{array}
\end{equation}
Note that the group-like elements have a corresponding Lie series $\mathfrak g((V))$ and exponential function $\exp\colon \mathfrak g((V)) \to G((V))$. We denote by $1$ \phantomsection\label{not:Uni-Grouplike-Elements} the unit of the group-like elements. Then, $\Pi^p$ can also be understood in the category of pointed metric groups. However, its properties depend on the choice of $\cG_p$ w.r.t.~$d $ in $\met(\cG_p)$. \\

\paragraph{\textbf{Equivalence of homogeneous metrics.}} Usual analytic results on $G^k(V)$, such as in \cite[Chapter 9]{Friz_Victoir_2010_fullbook}, exploit the fact that for fixed $k $ in $\N$ all homogeneous metrics are pairwise equivalent and induce the standard topology \cite[Proposition 7.44,Proposition 7.45]{Friz_Victoir_2010_fullbook} on $G^k(V)$.
However, to the misfortune of the analyst, this equivalence breaks when one starts to consider limit metrics. In fact, the proof of the finite-level equivalence involves -- similarly to the equivalence of norms on finite dimensional vector spaces -- homogeneity and the extreme value theorem of real functions. Thus, the Lipschitz constants depend in general on the step $k$ that is considered. To clarify, we provide a counterexample to the equivalence in the limit by analogous considerations as for the $\ell^1$ and $\ell^\infty$ spaces.
\begin{counterexample} \label{ex:counterexample}
    Let $k $ in $\N$ and consider a family $(\| \cdot \|_{k})_{k\in \N}$ of cross-norms on the tensor spaces $(V^{\otimes k})_{k \in \N}$, which are norms that satisfy for any $n,m$ in $\N$ and $u $ in $V^{\otimes n},v $ in $V^{\otimes m}$ the property $\| u \otimes v \|_{m+n} = \|u\|_n \|v\|_m$. \\
    Define for each $g \in G^k(V)$ an analogue of the $\ell^\infty$ norm and $\ell^1$ norm to be
    \begin{align} \label{eq:Max-metric-p=1-case-in-counterexample}
        \rho_k(1_k,g) &:= \max_{j = 1,...,k} (j! \| \proj_j(g) \|_{j})^{1/j},& \rho_k(g,h) &:= \rho_k(1_k,g^{-1}h)
        \\
        \label{eq:Sum-metric-p=1-case-in-counterexample}
        \sigma_k(1_k,g) &:= \sum_{j = 1}^k \rho_j(1_j,\pi_j^k(g)),& \sigma_k(g,h) &:= \sigma_k(1_k,g^{-1}h),
    \end{align}
    where the choice of $j!$ is to obtain the triangle inequality of $\rho_k$ by the binomial formula.
    Then, we can estimate
    \begin{equation*}
        \rho_k(1_k,g) \leq \sigma_k(1_k,g) \leq k\,\rho_k(1_k,g)
    \end{equation*}
    and it is straightforward to check that for the element $g = \exp_k(v)$, with $v $ in $V \backslash \{0\}$, we get
    \begin{equation*}
        j! \| \proj_j(\exp_k(v)) \|_j = \| v^{\otimes j} \|_j = \|v\|^j > 0
    \end{equation*}
    and hence
    \begin{equation*}
        \rho_k(1_k,\exp_k(v)) = \|v\|
    \end{equation*}
    by the cross-norm property. Then the two norms \eqref{eq:Max-metric-p=1-case-in-counterexample} and \eqref{eq:Sum-metric-p=1-case-in-counterexample} cease to be equivalent in the limit $k \to \infty$.
\end{counterexample}

\paragraph{\textbf{Canonical rough path metric.}} In rough path theory one usually considers tensor norms $\|\cdot\|_k$ on each $V^{\otimes k}$, which satisfy the Banach inequality
\begin{equation*}
    \| u \otimes v\|_{n+m} \leq \|u\|_n \|v\|_m, \qquad \forall u \in V^{\otimes n}, \forall v \in V^{\otimes m}
\end{equation*}
and permutation invariance on elementary tensors
\begin{equation*}
    \|u_1 \otimes ... \otimes u_n\|_n = \|u_{\sigma(1)} \otimes ... \otimes u_{\sigma(n)}\|_n, \qquad \forall u_1,...,u_n \in V, \forall \sigma \in \operatorname{Sym}_n
\end{equation*}
where $\operatorname{Sym}_n$ is the group of all permutations of $\{1,...,n\}$. Such tensor norms are referred to as \emph{admissible tensor norms}\phantomsection\label{not:admissible-tensor-norms} and we can define a max-type metric
\begin{equation} \label{eq:Max-metric-for-any-p}
    \rho^p_k(1_k,g)
    :=
    \max_{j=1,\ldots,k}
    \left(
        \beta_p (j/p)! \,\|\proj_j(g)\|_j
    \right)^{1/j},\qquad \rho^p_k(g,h) := \rho^p_k(1_k,g^{-1} h).
\end{equation}
Here $(j/p)!:=\Gamma(j/p+1)$  defined via the $\Gamma$-function, is chosen to counteract the decay of iterated integrals (also for $p>1$), and $\beta_p \geq p$ is chosen so that the neo-classical inequality \cite{NeoClassicalInequality}, and hence triangle inequality, also holds  for $p \geq 2$. In this way, $\rho^p_k$ indeed  yields a metric for any $k \geq \floor p$, is left-invariant, symmetric (ref.~\cite[Lemma 8.18]{Schmeding_2022}) and responsible for the levelwise control of path increments $X_{s,t}:=X_s^{-1}X_t$, ref.~\cite[Lemma 2.1.1]{Lyons1998}, due to the maximum function. In addition, it grants the $1$-Lipschitz property to the canonical projections $\pi^k_j$, so $(\rho^p_k)_{k \in \N_p} \in \met(\cG_p)$. Moreover, since each $\|\cdot\|_j$ is a norm, $\rho^p_k$ is a homogeneous metric. For $p=1$ and the choice of cross-norms, it reduces to the $\rho_k$ metric of \cref{ex:counterexample}. Since $G^k(V)$ (resp.~$T^k_1(V)$) is a finite-dimensional manifold for each $k\in\N$, we can conclude by homogeneity that $(G^k(V),\rho^p_k)$ (resp.~$T^k_1(V)$) for any $k \geq \floor p$ is a proper metric space. \\

\paragraph{\textbf{Submetries.}} Since for drivers of roughness $p \geq 2$ in \eqref{eq:controlled-differential-equation-rp}, higher order information up to level $\floor p$ is needed to construct solutions requiring that the corresponding iterated integrals are predefined.
Importantly, the higher order information for any level $n > \floor p$ should be uniquely determined by all the prescribed information up to level $\floor p$. In the case $p=1$ higher order information was uniquely given by iterated Riemann-Stieltjes integrals. The choice $\rho^p_k$ recovers the Lyons extension theorem \cite[Theorem 2.2.1]{Lyons1998} together with the improved bound for $\beta_p$ using \cite{NeoClassicalInequality}. For general metric spaces, Le Donne and Züst abstracted this property by means of the notion of
\emph{unique path lifting} \cite[Definition 2.14]{LeDonneZuest_public}.
However, their statement requires the projections to be submetries.

\begin{definition}[Submetry, comp.~\protect{\cite[Definition 2.1]{LeDonneZuest_public}, \cite[Definition 3.1.23]{LeDonne2025}}]
    Let $(E,d_E)$ and $(F,d_F)$ be metric spaces. A map $\pi\colon E\to F$ is called a \emph{submetry} if it is $1$-Lipschitz and if for every $x\in E$ and every $y\in F$ there exists $x'\in E$ such that
    \begin{equation*}
        \pi(x')=y,
        \qquad
        d_E(x,x')=d_F(\pi(x),y).
    \end{equation*}
    Equivalently, $\pi(\bar B(x,r)) = \bar B(\pi(x),r), \ \forall x\in X, \  \forall r>0$.
\end{definition}
The submetry assumption has two purposes. 
First, submetries admit a characterisation theorem w.r.t.~geodesic spaces, which is can be checked to also hold for $p$-variation geodesic spaces along the same proof lines as in \cite[Proposition 3.1.28]{LeDonne2025}. Second, if each $\pi^{k+1}_k$ and hence each $\pi^k_j$ is a submetry, by the arguments given in \cite[Lemma 2.2]{LeDonneZuest_public}, then also the limit projections $\pi^\infty_k$ are submetries. However, while it is easy to show that $\tilde \pi^k_j \colon (T^k_1(V),\rho^p_k) \to (T^j_1(V),\rho^p_j)$ are submetries, due to the existence of the trivial inclusion
\begin{equation*}
    T^j_1(V) \ni (1,g^1,...,g^{j}) \hookrightarrow (1,g^1,...,g^{j},\underbrace{0,...,0}_{(k-j) \text{ times}}) \in T^k_1(V),
\end{equation*}
we can show that this is not the case for $\pi^k_j \colon (G^k(V),\rho^p_k) \to (G^j(V),\rho^p_j)$.

\begin{counterexample} \label{counterexample:Submetry-Rho}
    Recall that $\rho^p_k \in \met(\cG_p)$. Then, notice that by left-invariance we can w.l.o.g.~examine $g \mapsto \rho^{p}_k(1_k,g)$ for $g \in G^k(V)$. Choose $V = \mathbb{R}^2$ generated by the orthonormal vectors $e_1$ and $e_2$ and let $\|\cdot \|_j$ denote the Hilbert-Schmidt norm on each $(\R^2)^{\otimes j}$ induced by the corresponding euclidean inner product $\langle\cdot,\cdot\rangle$ on $V$ (ref. \cite[Exercise 8.2.1]{Schmeding_2022}).  Notice that if each $\pi^{k+1}_k$ is a submetry, then so is $\pi^{k}_j$ for $j < k$ where defined. Then, we consider $\pi^3_2\colon (G^3(V),\rho^1_3) \to (G^2(V),\rho^1_2)$. Notice that then for the Lie algebra
    \begin{equation*}
        \mathfrak v_1 = V = \operatorname{span}(e_1,e_2), \quad \mathfrak v_2 = \operatorname{span}(\underbrace{[e_1,e_2]}_{=: e_{12}}), \quad \mathfrak v_3 = \operatorname{span}(\underbrace{[e_1,[e_1,e_2]]}_{=:e_{112}},\underbrace{[e_2,[e_1,e_2]]}_{=:e_{212}}).
    \end{equation*}
    Then, the choice
    \begin{equation*}
        g_2 = \exp_2\left(e_1 + \frac{1}{2} e_{12}\right) = 1 + e_1 + \frac{1}{2} (e_1e_1 + e_{12})
    \end{equation*}
    shows that $\pi^3_2$ cannot be a submetry. Indeed, notice that $\langle e_1 e_1, e_{12} \rangle = 0$, $\|e_1e_1\|_2 = 1$, $\|e_{12}\|_2^2 = 2$ and hence
    \begin{equation*}
        \rho^1_2(1_2,g_2) = \max\left\{ 1 ,  \left(2  \frac{\sqrt{3}}{2} \right)^{1/2} \right\} = \max\left\{\sqrt{3}^{1/2}, 1\right\} = 3^{1/4}.
    \end{equation*}
    Any $g_3 \in G^3(V)$ with $\pi^3_2(g_3)=g_2$ is of the form
    \begin{equation*}
        g_3(u,v) = \exp_3\left(e_1 + \frac{1}{2} e_{12} + u e_{112} + v e_{212} \right)
    \end{equation*}
    with appropriate $u,v \in R$ and we get
    \begin{equation*}
        \proj_3(g_3(u,v)) = u\, e_{112} + v\, e_{212} + \frac{1}{4} (e_1 e_{12} + e_{12} e_1) + \frac{1}{6} e_1 e_1 e_1.
    \end{equation*}
    Checking that
    \begin{equation*}
        \langle (e_1 e_{12} + e_{12} e_1), e_{112} \rangle = \langle (e_1 e_{12} + e_{12} e_1), e_{212} \rangle = \langle e_{1}e_1e_1 , e_{112} \rangle = \langle e_1e_1e_1, e_{212}\rangle = 0
    \end{equation*}
    shows that $w = a (e_1 e_{12} + e_{12} e_1) + b e_1e_1e_1$ is orthogonal to the space $\mathfrak v_3$. Therefore, it follows by the Hilbert-Schmidt norm property for orthogonal vectors that
    \begin{equation*}
        \begin{aligned}
            \min_{(u,v) \in \R^2}\|\proj_3(g_3(u,v))\|_3^2 &= \|\proj_3(g_3(0,0))\|_3^2 \\
            &= \| \frac{1}{4} (e_1 e_{12} + e_{12} e_1)\|_3^2 + \| \frac{1}{6} e_1 e_1 e_1\|_3^2 \\
            &= \frac{1}{16} \cdot 2 + \frac{1}{36} = \frac{11}{72}.
        \end{aligned}
    \end{equation*}
    The submetry property asserts that
    \begin{equation*}
        (3! \|\proj_3(g_3)\|_3)^{1/3} \leq 3^{1/4} \Rightarrow \left(3!\|\proj_3(g_3(0,0))\|_3\right)^{1/3} = \left(\frac{11}{2}\right)^{1/6} \leq 3^{1/4}
    \end{equation*}
    but $3^{1/4} < 1.32 < (11/2)^{1/6}$ leading to a contradiction.
\end{counterexample}
\begin{remark} \label{rem:Properties-submetry}
    Let $d=(d_k)_{k\in\N_p}\in\met(\cG_p)$. Suppose each $\pi^{k+1}_k \colon (G^{k+1}(V),d_{k+1}) \to (G^k(V),d_k)$ is a submetry. Then, if $(w_k)_{k\in \N_p} \subset (0,\infty)$ are weights, $(w_k d_k)_{k \in \N_p}$ will preserve the submetry properties of the projections if and only if $(w_k)_{k \in \N_p}$ is a constant sequence.
\end{remark}

\paragraph{\textbf{Path lifting.}} By \cref{counterexample:Submetry-Rho} and Lyons' extension theorem, which by \cite[Corollary 3.9]{OnRoughIntegration_Cass2016} remains valid for the restriction to the geometric case,  we get a straightforward generalisation of the required definition of the unique path lifting property. However, we shall refrain from imposing the submetry property at this stage and we shall instead consider it as an extra property one needs to check.
\begin{definition} \label{def:LeDonne-unique-path-lifting-p-variation}
    Let $(E,d_E)$ and $(F,d_F)$ be non-empty metric spaces. A $1$-Lipschitz $\pi\colon E\to F$ is said to have the \emph{unique $p$-path lifting property} if for every path $X\colon[0,1]\to F$ of finite $p$-$d_F$-variation and every point $x\in\pi^{-1}(X_0)$ there exists a path $Y\colon[0,1]\to E$ such that
    \begin{enumerate}[label=(\roman*)]
        \item $\pvd{Y}{p}{d_E}<\infty$, \label{def:LeDonne-unique-path-lifting-p-variation-finite-p-variation}
        \item $Y_0=x$,
        \item $\pi(Y_t)=X_t$ for all $t\in[0,1]$, \label{def:LeDonne-unique-path-lifting-p-variation-projection}
        \item if a path $Z\colon[0,1]\to E$ satisfies (i), (ii), and (iii), then $Z=Y$, \label{def:LeDonne-unique-path-lifting-p-variation-uniqueness}
        \item $\pvd{Y}{p}{d_E}=\pvd{X}{p}{d_F}$. \label{def:LeDonne-unique-path-lifting-p-variation-isometry}
    \end{enumerate}
    Any such path $Y$ is called the \emph{lift} or \emph{extension} of $X$ starting at $x$. If the lift exists for all such $X$ and starts at a given distinguished point, we denote the corresponding lifting map by $\cL_{F\to E}.$\footnote{Note that $\cL_{F\to E}$ depends on the choice of $d_E,d_F$, which is implicit.}
\end{definition}
\begin{remark}
We record that by Lyons' extension theorem for the choice of $(\rho^p_k)_{k \in \N_p}$ each $\pi^k_j$ for $\floor p \leq j \leq k$ carries the unique $p$-path lifting property. Indeed, by choosing the control function
\begin{equation*}
    \control_X(s,t)
    :=
    \pvdi{X}{p}{\rho^p_{\floor{p}}}{[s,t]}^p,
\end{equation*}
the levelwise control of the increments of the extended path $Y$ can be translated back to the metric $\rho^p_k$ using \cite[Lemma 2.2.1]{Lyons1998}. 
The logical chain is:
\begin{align*}
    \pvdi{X}{p}{\rho^p_{\floor{p}}}{[s,t]}^p \leq \control_X(s,t)
    \Longleftrightarrow&
    \|\proj_j(X_s^{-1}X_t)\|_j
    \leq
    \frac{\control_X(s,t)^{j/p}}{\beta_p (j/p)!},
    \quad
    \forall j\leq{\floor{p}},
    \\
    \xRightarrow{\textrm{Extension}}&
    \|\proj_j(Y_s^{-1}Y_t)\|_j
    \leq
    \frac{\control_X(s,t)^{j/p}}{\beta_p (j/p)!},
    \quad
    \forall j\leq k,
    \\
    \Longleftrightarrow&
    \pvdi{Y}{p}{\rho^p_k}{[s,t]}^p
    \leq
    \control_X(s,t),
\end{align*}
where $\pi^k_{\floor{p}}\circ Y=X$.
\end{remark}
\begin{remark} \label{rem:LeDonneZuest-proof-2.6-and-2.7}
    \begin{enumerate}
    \item Let $\pi\colon (E,d_E)\to(F,d_F)$ be a $1$-Lipschitz map with the unique $p$-path lifting property, and let $Y$ be the lift of $X$ starting at $x\in\pi^{-1}(X_0)$. Then, the direct $p$-variation analogue of \cite[Lemma 2.6]{LeDonneZuest_public}, which follows the same lines of proof, holds. Namely,
    \begin{enumerate}[label=(\alph*)]
        \item for every $0\leq s<t\leq1$,
        \begin{equation*}
            \pvdi{Y}{p}{d_E}{[s,t]}
            =
            \pvdi{X}{p}{d_F}{[s,t]},
        \end{equation*}
        \item since $X$ admits a $1/p$-Hölder parametrisation, then the lifted path $Y$ admits a $1/p$-Hölder parametrisation with the same Hölder constant,
        \item if $Z\colon[0,1]\to E$ has finite $p$-$d_E$-variation, then $Z$ is the unique lift of $\pi\circ Z$ starting at $Z_0$.
    \end{enumerate}
    \item Assuming that each $\pi^k_j$ satisfies the unique $p$-path lifting property, we can conclude along the same lines as \cite[Lemma 2.7]{LeDonneZuest_public}, that also $\pi^\infty_k$ satisfies the unique $p$-path lifting property.
    \item The estimate
        \begin{equation} \label{eq:p-variation-1-Lipschitz-inequality}
            \pvd{\pi^k_j(X)}{p}{d_j}
            \leq
            \pvd{X}{p}{d_k}
        \end{equation}
        holds for every path $X\colon[0,1]\to G^k(V)$ of finite $p$-$d_k$-variation and every ${\floor{p}}\leq j\leq k$.
    \end{enumerate}
\end{remark}
The equivalence of homogeneous metrics for finite step $k \geq \floor p$ implies that any choice of homogeneous metric $d_k$ for $G^k(V)$ will give for each $X$ over $G^{\floor p}(V)$ of finite $p$-variation a unique path $Y_k$ over $G^k(V)$. However, the bi-Lipschitz estimate coming from the equivalence of homogeneous metrics cannot fulfill \ref{def:LeDonne-unique-path-lifting-p-variation-isometry} of \cref{def:LeDonne-unique-path-lifting-p-variation}. Therefore, it is useful to distinguish the special set of metrics that do imply the unique $p$-path lifting property for the projections.
\begin{definition} \label{eq:Set-of-left-invariant-metrics-with-Lifts}
    Let $\metlift(\cG_p)\subseteq\met(\cG_p)$ be the set of sequences of metrics such that all projections
    \begin{equation*}
        \pi^k_j\colon (G^k(V),d_k)\to(G^j(V),d_j),
        \qquad
        {\floor{p}}\leq j\leq k,
    \end{equation*}
    admit the unique $p$-path lifting property of \cref{def:LeDonne-unique-path-lifting-p-variation}. The corresponding lifting maps are denoted by $\cL_k^{k+1} := \cL_{G^k(V) \to G^{k+1}(V)}$ with the convention
    \begin{equation*}
        \cL_k := \cL_{G^{k-1}(V) \to G^k(V)} \circ ... \circ\cL_{G^{\floor{p}}(V)\to G^{\floor{p}+1}(V)} .
    \end{equation*}
\end{definition}
\begin{remark}
    Since by Lyons' extension theorem the $G^{\floor{p}}(V)$-valued paths determine all paths with values in $G^k(V)$ for each $k \in \N_p$, let us distinguish $WG\Omega^p_{d_{\floor p}} := C^{p\textnormal{-}\var}_{d_{\floor p}}([0,1],G^{\floor p}(V))$ for any choice $d \in \metlift(\cG_p)$ and for the choice $\rho^p := (\rho^p_k)_{k \in \N_p}$ we set $WG\Omega^p := WG\Omega^p_{\rho^p_{\floor p}}$. We call $WG\Omega^p$ (resp.~$WG\Omega^p_{d_{\floor p}}$) the space of \emph{weakly geometric $p$-rough paths} (w.r.t.~$d$). \phantomsection\label{not:Weakly-geometric-RP}
    \begin{enumerate}
        \item by construction, $\rho^p = (\rho^p_k)_{k\in\N_p}$ lies in $\metlift(\cG_p)$ for $\beta_p$ above an appropriate threshold, see \cite{Lyons1998,NeoClassicalInequality}.
        \item item \ref{def:LeDonne-unique-path-lifting-p-variation-isometry} of \cref{def:LeDonne-unique-path-lifting-p-variation} forces each sequence of weights $(w_k)_{k \in \N_p} \subset (0,\infty)$ to be constant due to $w_k d_k \mapsto \pvd{X}{p}{w_kd_{k}} = w_k \pvd{X}{p}{d_{k}}$.
    \end{enumerate}
\end{remark}
One may now ask the following question: \\

\begin{center}
\textit{Which are the sequences of metrics $d$ in $\met(\cG_p)$ that belong to $\metlift(\cG_p)$?}\phantomsection \label{question:Metric-characterisation}
\end{center}
\phantom{.}\\
The above open question aims to better understand the possible metric structures behind the extension from a metric perspective. Since $\rho^p = (\rho^p_k)_{k \in \N_p} $ lies in $\metlift(\cG_p)$ for suitable $\beta_p$, the question is well-defined, but without further restrictions is too broad to grasp. We return to this question in \cref{sec:Canonical-metric-relations-goals}, in particular \cref{prop:homogeneous-metric-characterisaton}, where the intrinsic $p$-variation metrics allow for a complete answer within the subclass of homogeneous metrics.

\begin{remark}
    Since $\pi^k_j \colon G^k(V) \to G^j(V)$ is surjective, point lifts exist and are of the form
    \begin{equation*}
        \begin{aligned}
        \iota &\colon \frg^k(V) \to \frg^{k+1}(V),
            \qquad
            (0,u^1,\ldots,u^k) \mapsto (0,u^1,\ldots,u^k,0), \\
            Y^z_t &= \exp_{k+1}(\iota(u_t)+z_t), \qquad z_t\in\mathfrak v_{k+1}.
        \end{aligned}
    \end{equation*}
    
    Further, the submetry property together with compactness assumptions as in \cite[Proposition 3.1.24]{LeDonne2025} does not guarantee the existence of finite $p$-variation lifts for $p>1$. Indeed, while the submetry property still allows to choose consecutive lifted increments along each partition with corresponding bound \ref{rem:more-on-p-variation:Upper-bound-of-metric} of \cref{rem:more-on-p-variation}, passing to a path lift requires an additional uniform control on all lifted increments instead on consecutive ones.
\end{remark}

\paragraph{\textbf{Homomorphism property of lifts.}} Similar to the discussion of the properties of flow maps and endpoint maps over a Lie group at the end of \cref{sec:Paths-over-pointed-metric-groups}, from the definition and properties of the unique $p$-path lifting, we derive algebraic properties of the lifting map.

\begin{lemma} \label{lem:Homomorphism-of-lifts}
    Let
    \begin{equation*}
        \pi\colon (G,d_G)\to(H,d_H)
    \end{equation*}
    be a $1$-Lipschitz homomorphism of pointed metric groups with the unique $p$-path lifting property.
    The lifting map $\cL_{H\to G}$ is a bijection onto the pointed finite $p$-variation paths in $G$ starting at $e_G$.
    Moreover,
    \begin{equation*}
        \cL_{H\to G}(X\sqcup Y)
        =
        \cL_{H\to G}(X)\sqcup \cL_{H\to G}(Y), \qquad \cL_{H\to G}(\rev{X}) = \overleftarrow{\cL_{H\to G}(X)}
    \end{equation*}
\end{lemma}
\begin{proof}
    Injectivity follows from
    \begin{equation*}
        \pi\circ\cL_{H\to G}(X)=X.
    \end{equation*}
  In order to check the surjectivity property, we consider a pointed finite $p$-variation path $Z$ in $G$. Since $\pi$ is $1$-Lipschitz, $\pi\circ Z$ has finite $p$-variation in $H$, and by uniqueness of lifts,
    \begin{equation*}
        Z=\cL_{H\to G}(\pi\circ Z).
    \end{equation*}
    Finally, since $\pi$ is a homomorphism,
    \begin{equation*}
        \pi\circ\bigl(\cL_{H\to G}(X)\sqcup \cL_{H\to G}(Y)\bigr)
        =
        X\sqcup Y.
    \end{equation*}
    The claim follows again from the uniqueness of the lift of $X\sqcup Y$. In the same manner, we derive $\pi \circ \overleftarrow{\cL_{H\to G}(X)} = \rev{X}$ and hence $\overleftarrow{\cL_{H\to G}(X)} = \cL_{H \to G}(\rev{X})$.
\end{proof}

Denote by $\eva_1$ the evaluation map at time $1$ of paths with values in a given set.
\begin{lemma} \label{lem:Homomorphism-And-Image-of-endpoint-map}
We keep the notations and assumptions of \cref{lem:Homomorphism-of-lifts} and define the end point map
    \begin{equation*}
        S_{H\to G}:=\eva_1\circ\cL_{H\to G}.
    \end{equation*}
    Then, $S_{H\to G}$ is a homomorphism, i.e.~it satisfies
    \begin{equation*}
        S_{H\to G}(X\sqcup Y)
        =
        S_{H\to G}(X)S_{H\to G}(Y),
    \end{equation*}
    and
    \begin{equation*}
        S_{H\to G}(\rev X)=S_{H\to G}(X)^{-1},
    \end{equation*}
    and is invariant under ordered reparametrisations. Furthermore, if every point $g\in G$ can be joined to $e_G$ by a path of finite $p$-variation, then $S_{H\to G}$ is surjective.
\end{lemma}
\begin{proof}
    The homomorphism follows by the definition of concatenation as well as the homomorphism property of $\cL_{G\to H}$ under concatenation and $\eva_1$ under multiplication. Hence,
    \begin{equation*}
        S_{H\to G}(X \sqcup Y) = \eva_1 \circ \cL_{H\to G}(X \sqcup Y) = \eva_1 \circ (\cL_{H\to G}(X) \sqcup \cL_{H\to G}(Y)) = S_{H\to G}(X)S_{H\to G}(Y).
    \end{equation*}
    Furthermore,
    \begin{equation*}
        \begin{aligned}
            S_{H\to G}(\rev{X}) &= \eva_1 \circ \cL_{H\to G}(\rev{X}) = \eva_1 \circ \overleftarrow{\cL_{H\to G}(X)} \\
            &= (S_{H\to G}(X))^{-1} \cL_{H\to G}(X)_0 = (S_{H\to G}(X))^{-1}.
        \end{aligned}
    \end{equation*}
    If $\sigma$ is an ordered reparametrisation, then $\cL_{H\to G}(X)\circ\sigma$ is the lift of $X\circ\sigma$. Hence uniqueness and $\sigma(1)=1$ imply $S_{H\to G}(X\circ\sigma)=S_{H\to G}(X)$.
    Finally, let $g$ in $G$ and choose a pointed finite $p$-variation path $Z$ in $G$ with $Z_1=g$. By bijectivity of $\cL_{H\to G}$,
    \begin{equation*}
        Z=\cL_{H\to G}(\pi\circ Z),
    \end{equation*}
    and therefore $S_{H\to G}(\pi\circ Z)=g$.
\end{proof}
It follows from the previous discussion that the $\cL_k$'s and $S_k$'s defined via \cref{def:LeDonne-unique-path-lifting-p-variation} fulfil the endpoint map requirements \eqref{eq:Endpoint-Map-Properties} one would want to impose on a homomorphism $\Phi \colon WG\Omega^p \to H$ for any group $H$. \\
Returning to $\cG_p$ equipped with $\rho^p = (\rho^p_k)_{k\in\N_p}\in\metlift(\cG_p)$ and the corresponding lifts $\cL_k$, we define the endpoint map $S_k:=\eva_1\circ\cL_k$ and call it the \emph{level-$k$ signature}. By Chow's theorem \cite[Proposition 7.28]{Friz_Victoir_2010_fullbook}, every point $g$ of $G^k(V)$ can be joined to the unit $1_k$ by a path $X$ in the admissible set
\begin{equation} \label{eq:Horizontal-Paths-of-CC-metric-p=1}
    \begin{aligned}
        \mathcal{A}_k^{\CC}(g) &:= \mathcal{A}^{\CC}_k(1_k,g) \\
        &:= \{ X \mid \int_0^1 \|X_t^{-1}\dot X_t\|\,dt < \infty, \ X_t^{-1}\dot X\in V\subseteq \mathfrak g^k(V)\text{ for a.e. }t\in[0,1], \ X_1 = g \}.
    \end{aligned}
\end{equation}
Paths in this set are referred to as \emph{horizontal paths} in the sub-Finsler literature, where $\|\cdot\|$ is the norm coming from $V$ and $\pvd{X}{1}{\rho^1_k} < \infty$. Since for any $p\geq1$ and fixed $k$, the homogeneous metrics $\rho^p_k$ and $\rho^1_k$ are equivalent, every path in $\mathcal{A}_k^{\CC}(G^k(V))$ has finite $1$-variation w.r.t.~$\rho^p_k$ and hence finite $p$-variation by \ref{rem:more-on-p-variation-nesting-of-variations} of \cref{rem:more-on-p-variation}, the surjectivity condition in \cref{lem:Homomorphism-And-Image-of-endpoint-map} is satisfied for any $k \in \N_p$ w.r.t to $\mathcal{A}_k^{\CC}(G^k(V)) = \bigcup_{g \in G} \mathcal{A}_k^{\CC}(g)$ and left translation. Therefore $S_k$ is surjective, invariant under ordered reparametrisations, and satisfies the homomorphism property
\begin{equation} \label{eq:Chens-relation-as-a-homomorphism}
    S_k(X\sqcup Y)=S_k(X)S_k(Y),
\end{equation}
commonly referred to as \emph{Chen's relation}, as well as $S_k(\rev X)=S_k(X)^{-1}$. \\

\paragraph{\textbf{Lifts into the group-like elements}.} Lyons extension and the signature further admit a formulation into the group-like elements $G((V))$. Let
\begin{equation*}
    \Pi_k \colon G((V)) \to G^k(V), \quad (1,g^1,...,g^k,...) \mapsto (1,g^1,...,g^k)
\end{equation*}
be the canonical projection for the group-like elements and note that $\Pi_k \circ \Pi^p = \pi^\infty_k$ holds by \eqref{eq:Inverse-limit-to-product-conversion-map-p} for $k \geq \floor p$.
\begin{remark}
In \cite{BOEDIHARDJO_LYONS_2016720} the codomain for the signature map was chosen to be
    \begin{equation} \label{eq:Codomain-Exponential-decay}
        G_{p.r.c.}(V) := \{ g \in G((V)) \mid \sup_{k \in \N} (||\proj_k(g)||_k)^{1/k} < \infty \}
    \end{equation}
    as it includes all group-like elements of exponential decay w.r.t.~the norms, which covers the image of the signature map for every $p \geq 1$ as they decay factorially. As $p \geq 1$ is considered arbitrary but fixed in this article, it is possible to refine the codomain where the lifts and signatures for a given $p \geq 1$ take their values, namely we can take instead
    \begin{equation} \label{eq:Max-metric-rho-p.r.c.-set-codomain}
        G_{\rho^p\textnormal{-}p.r.c.}(V) := \{ g\in G((V)) \mid \sup_{k \in \N_p} \rho^p_k(1_k,\Pi_k(g)) < \infty \}
    \end{equation}
    or equivalently all $g \in G((V))$, such that $\limsup_{k \to \infty} ((k/p)! \|\proj_k(g)\|_k)^{1/k} < \infty$.
\end{remark}
Now, given any sequence $d = (d_k)_{k \in \N_p}$ in $\metlift(\cG_p)$, the appropriate set where the lifted paths and endpoint map take their values is
\begin{equation} \label{eq:Group-like-elements-with-more-completion-under-metrics}
    G_{d\textnormal{-}p.r.c.}(V) :=  \{ g\in G((V)) \mid \tilde d_\infty(1,g) := \lim_{k \to \infty} d_k(1_k,\Pi_k(g)) < \infty \}.
\end{equation}
For any $d,d'$ in $\metlift(\cG_p)$ the sets $G_{d\textnormal{-}p.r.c.}$ and $G_{d'\textnormal{-}p.r.c.}$ do not necessarily coincide due to \cref{ex:counterexample}. However, whether $G_{\rho^1-p.r.c.}(V)$ with $\rho^1=(\rho^1_k)_{k \in \N}$ coincides with $ G_{d^{1-\CC}-p.r.c.}(V)$ where $d^{1-\CC} = (d_k^{1-\CC})_{k \in \N}$ is given by
\begin{equation} \label{eq:CC-distance-p=1}
    d^{1-\CC}_{k}(1,g) := \min_{X \in \mathcal{A}^{\CC}_k(g)} \int_0^1 \|X_t^{-1}\dot X_t\| \mathrm{d}t
\end{equation}
is an open question.

Let $\mathcal{L}$  be the corresponding lift to the group-like elements. The set $G_{d\textnormal{-}p.r.c.}(V)$ is naturally equipped with the left-invariant metric $\tilde d_\infty(g,h) := \tilde d_\infty(1,g^{-1}h)$, which we can identify with $d_\infty$ for the corresponding inverse limit via the definition of $\Pi_k$. \\

\begin{definition}
    With the above notations we call the endpoint map by $S := \eva_1 \circ \mathcal{L}$ the \emph{(full) signature}.
\end{definition}

For the choice $\rho^p = (\rho^p_k)_{k \in \N_p}$ we recover the signature introduced by T.~Lyons.\footnote{We refrain from the two-parameter setup as done in \cite{Lyons1998}, and take the view of \cite[Chapter 8]{Friz_Victoir_2010_fullbook} instead, which is justified by Chen's relation for increments $\cL(X)_{s,t}$.} Let us define the \emph{signature group} w.r.t.~$d \in \metlift(\cG_p)$ to be
\begin{equation}\label{eq:Signature-Group-Definition}
    SG^p_d := S_{G^{\floor p}(V) \to G_{d\textnormal{-}p.r.c.}(V)}(WG\Omega^p_{d_{\floor p}}) \subseteq G_{d\textnormal{-}p.r.c.}(V)
\end{equation}
and note that $SG^p_d$ is closed w.r.to~the multiplication of $G_{d\textnormal{-}p.r.c.}$ by \cref{lem:Homomorphism-And-Image-of-endpoint-map}. \\
Let for now $\dim V \geq 2$.

\begin{claim}
    Recall that $\rho^p_k$ is defined via any admissible tensor norms \eqref{eq:Max-metric-for-any-p} and the definitions of this paragraph with $\dim V \geq 2$. Then,
    \[
        G_{\rho^p\textnormal{-}p.r.c.}(V) \subsetneq G((V))
    \]
\end{claim}
\begin{proof}
    We give a construction of a group-like element $g$, such that $g \not\in G_{\rho^p\textnormal{-}p.r.c.}(V)$. For that, let $g = \exp(l)$ with $\Pi_k(g) = g_k = \exp_k(l_k)$ and $l_k = \sum_{m = 1}^k a_m l^m_k$ where $l_k^m \in \mathfrak v_m, \, \|l^m_k\|_m = 1$ for $m = 1,...,k$ and $a_m \neq 0$ are real numbers. 

    The compatibility condition \(\pi^k_{k-1}(g_k)=g_{k-1}\) is equivalent to \(\pi^k_{k-1}(l_k)=l_{k-1}\) under the identification of both $G^k(V)$ and $\mathfrak g^k(V)$ as subspaces of $T^k(V)$. As the \(a_m\) are nonzero, uniqueness of the homogeneous decomposition then yields \(l_k^m=l_{k-1}^m\) for every \(m=1,\ldots,k-1\).

    Let us observe that we can express
    \begin{equation*}
        \proj_j(g_k) = a_j l_k^j + R^j,
    \end{equation*}
    where $R^k$ is a polynomial formed by entries of degree lower than $k$ of $g_k$. Since each $l_k^m$ is normed, the reverse triangle inequality lets us bound
    \[
        \|\proj_j(g_k)\|_j \geq |\,|a_j| \|l_k^j\| - \|R^j\|_j| = |\,|a_j| - \|R^j\|_j|
    \]
    and by choosing $a_j > \|R_j\|_j + \frac{j^j}{\beta_p (j/p)!} \geq 0$ to be bigger than the $\R^j$ contribution as well as the root decay and factorial weights imposed by $\rho^p_k$, we end up with a controllable lower bound
    \begin{align*}
        \rho^p_k(1_k,g_k) &= \max_{j = 1,...,k} \left(\beta_p (j/p)! \|\proj_j(g_k)\|_j\right)^{1/j} > \max_{j = 1,...,k} \left(\beta_p (j/p)! \frac{j^j}{\beta_p (j/p)!}\right)^{1/j} \\
        &> \max_{j = 1,...,k} j = k
    \end{align*}
    which results in $\sup_{k \in \mathbb N_p} \rho_k^p(1_k,g_k) = \infty$ and hence $g \not\in G_{\rho^p\textnormal{-}p.r.c.}(V)$ by construction.
\end{proof}
Writing $SG^p := SG^p_{\rho^p}$, we also get $SG^p \subsetneq G((V))$.\\

\paragraph{\textbf{Tree structures for signatures.}} Although the unique $p$-path lifting provides the desired algebraic structure at the end of \cref{sec:Paths-over-pointed-metric-groups} in an abstract manner, it does not state anything about the kernel of $S_{F\to E}$. \\
Since a characterisation of $\metlift(\cG_p)$ in accordance to our question on \hyperref[question:Metric-characterisation]{p.~\pageref*{question:Metric-characterisation}} as of was is too broad for general metrics, we restrict in this paragraph to the choice $d = \rho^p = (\rho^p_k)_{k \in \N_p}$. Then, the kernel of the signature characterises the tree-like paths into $G^{{\floor{p}}}(V)$. In addition, between any two distinct points in $G_{\rho^p\textnormal{-}p.r.c.}(V)$ there exists a unique arc of finite $p$-variation in the following sense.
\begin{lemma}[comp.~\protect{\cite[Lemma 4.6]{BOEDIHARDJO_LYONS_2016720}}] \label{lem:Uniqueness-arcs-p-variation-limit}
    Let $S\colon[0,T]\to G_{p.r.c.}(V)$ be a continuous path with finite $p$-variation. There exists an injective path
    \begin{equation*}
        \widetilde S\colon[0,\widetilde T]\to G_{p.r.c.}(V)
    \end{equation*}
    such that
    \begin{equation*}
        \widetilde S_0=S_0, \qquad \widetilde S_{\widetilde T}=S_T.
    \end{equation*}
    Moreover, $\widetilde S$ is unique up to reparametrisation.
\end{lemma}
Since there exists a $C_p = \sup_{k \in \N} (\beta_p (k/p)!)^{-1/k}$ independent of the level $k$, s.t.
\begin{equation*}
    \|\proj_k(g) \|_k^{1/k} \leq C_p \sup_{k \in \N} \rho^p_k(1_k,\Pi_k(g)), \qquad \forall g \in G_{\rho^p\textnormal{-}p.r.c.},
\end{equation*}
the proof of \cref{lem:Uniqueness-arcs-p-variation-limit} in \cite{BOEDIHARDJO_LYONS_2016720} carries over for fixed $p \geq 1$ when replacing $G_{p.r.c.}(V)$ with $G_{\rho^p\textnormal{-}p.r.c.}(V)$. Hence, we can reformulate \cref{lem:Uniqueness-arcs-p-variation-limit} for $G_{\rho^p\textnormal{-}p.r.c.}(V)$ (under reparametrisations of the intervals) into the metric geometric language.
\begin{lemma} \label{lem:Reformulation-Uniqueness-to-tree}
    Let for any $g,h$ in  $G_{\rho^p\textnormal{-}p.r.c.}$ define
    \begin{equation*}
        \begin{aligned}
            d^p(g,h) = \inf_{Y \in \mathcal{A}_{\operatorname{inj}}(g,h)} \pv{Y}{p},
        \end{aligned}
    \end{equation*}
    where if $g \neq h$ we set 
    \begin{equation} \label{eq:Injective-set}
        \mathcal{A}_{\operatorname{inj}}(g,h) := 
               \{ Y \in C^{p\textnormal{-}\var}([0,1],G_{\rho^p\textnormal{-}p.r.c.}(V)) \mid g \neq h,\, Y_0 = g, \ Y_1 = h, \ Y \text{ injective} \}
    \end{equation}
  and otherwise
    \begin{equation*}
        \mathcal{A}_{\operatorname{inj}}(g,g) := \{ t \mapsto g\}.
    \end{equation*}
    Then, the finite $p$-variation path connected component
    \begin{equation*}
        PG_{\rho^p\textnormal{-}p.r.c.} = \{ g \in G_{\rho^p\textnormal{-}p.r.c.}(V) \mid d^p(1,g) < \infty \}
    \end{equation*}
    is a $p$-variation geodesic space with admissible set $\mathcal{A}_{\operatorname{inj}}(g,h)$ and a topological tree.
\end{lemma}
We will provide a proof for completeness after giving a comment on the choice of injective paths considered in the set of admissible paths.
\begin{remark} \label{rem:Injectivity-reason}
    Notice that the restriction to injective paths in \eqref{eq:Injective-set} is necessary to single out a unique minimiser of $d^p$, which is subsequently used to prove the UAC property in \cref{lem:Reformulation-Uniqueness-to-tree}.
    Indeed, contrary to the $p = 1$ case, the minimisers for $p>1$ need not be unique, ref.~\cite[Remark 4.1]{BOEDIHARDJO_LYONS_2016720} where \cite[Example 2.1]{Geng_2017_Reconstruction} and \cite[Example 3.1]{cass2024topologiesunparameterisedroughpath} show that certain retracings of path segments can leave the $p$-variation of a path unchanged. Therefore, $p$-variation alone can not distinguish the unique injective minimiser, and thus the injectivity condition needs to be imposed explicately.

    \cref{lem:Uniqueness-arcs-p-variation-limit} provides a unique arc under loop erasure of finite $p$-variation paths (up to reparametrisation), contrary to the case of general continuous paths as remarked after \cref{thm:loop-erasure}.  
    Letting $Y$ be a path of finite $p$-variation $G_{\rho^p\textnormal{-}p.r.c.}$, we denote the arc obtained by loop erasure as $Y^\circlearrowleft$ with the corresponding map by $\cdot^{\circlearrowleft} \colon Y \mapsto Y^\circlearrowleft$. Notice that $\cdot^{\circlearrowleft}$ constitutes a projection onto the arcs of finite $p$-variation. As loop erasure does not increase $p$-variation, it is possible to obtain the unique injective minimiser in \eqref{eq:Injective-set} by first solving the minimisation problem without imposing injectivity and then applying loop-erasure afterwards.

    Due to the UAC property of $(PG_{\rho^p\textnormal{-}p.r.c.},d^p)$, we may also employ the notation introduced in \cref{sec:Paths-over-pointed-metric-groups} to characterise any such arc of finite $p$-variation by its ordered pair of its start and endpoint. Then, one may define analogously to the tree-like equivalence in \cref{def:Tree-like-equivalence}, an equivalence of two pointed paths\footnote{Recall the paragraph on pointed paths in \cref{sec:Paths-over-pointed-metric-groups}}, i.e.~$X,Y$, of finite $p$-variation by
    \begin{equation*}
        X \operatorname{inj} Y :\Leftrightarrow (\rev X \sqcup Y)^\circlearrowleft = O \Leftrightarrow X_1 = Y_1
    \end{equation*}
    where $O \colon t \mapsto 1$ and the last equivalence follows by the uniqueness of the arcs of finite $p$-variation. One can then check that also the corresponding equivalence class $[X]_{\operatorname{inj}}$ (by understanding the arcs as the images of injective paths) will make the induced concatenation associative. We can then define the injective pointed paths of finite $p$-variation as the quotient $\cI^p := C^{p-\var}([0,1],G_{\rho^p\textnormal{-}p.r.c.})/\operatorname{inj}$ and notice that this set naturally constitutes a group under the induced concatenation under loop erasure as the group multiplication and the induced path reversal as the group inversion. The group identity is then given by $[O]_{\operatorname{inj}}$.
\end{remark}

\begin{proof}[Proof of \protect{\cref{lem:Reformulation-Uniqueness-to-tree}}]
    By \cref{lem:Uniqueness-arcs-p-variation-limit}, every finite $p$-variation path admits a loop-erasure with the same endpoints, not increasing $p$-variation, and unique up to reparametrisation. Hence $\mathcal{A}_{\operatorname{inj}}(g,h)$ for any $g,h $ in $PG_{\rho^p\textnormal{-}p.r.c.}$ is non-empty and consists of one reparametrisation class. Thus the minimum defining $d^p$ is attained and well-defined. \\
    The triangle inequality follows by concatenating the pairwise connecting injective arcs of $g,h,k$ from $g$ to $h$ and from $h$ to $k$ under loop-erasure, which does not increase the $p$-variation. \\

    Let $Y $ in $\mathcal{A}_{\operatorname{inj}}(g,h)$ be the unique injective representative, up to reparametrisation. Since for every $[s,t] \subseteq [0,1]$ the restriction $Y|_{[s,t]}$ is the unique representative between $Y_s$ and $Y_t$, it follows
    \begin{equation*}
        d^p(Y_s,Y_t) = \|Y\|_{p\textnormal{-}\var,[s,t]}
    \end{equation*}
    and hence using summing over partitions that
    \begin{align*}
        \pvd{Y}{p}{d^p}^p &= \sup_{\mathcal{D}} \sum_{t_i \in \mathcal{D}} d^p(Y_{t_{i-1}},Y_{t_i})^p \\
        &= \sup_{\mathcal{D}} \sum_{t_i \in \mathcal{D}} \|Y\|_{p\textnormal{-}\var,[t_{i-1},t_i]}^p \\
        &= \pv{Y}{p}^p = d^p(g,h)^p.
    \end{align*}
    Reparametrising $Y$ to be $1/p$-Hölder shows that $Y$ is a Hölderian geodesic w.r.t.~$p$-variation and hence $(PG_{\rho^p\textnormal{-}p.r.c.},d^p)$ is a $p$-variation geodesic space. \\

    To show the UAC and LAC property we use interpolation in the spirit of \cite[Section 5.2, Section 8.2]{Friz_Victoir_2010_fullbook} in the $p$-variation case. Let $[g,h] := Y([0,1])$ be the image of the unique injective representative $Y$ constructed before. For $k $ in $[g,h]$, we set $d^p(g,k) \leq d^p(g,h)$. Now, let $A$ be any arc from $g$ to $h$ that is continuous w.r. to ~the induced $d^p$-topology parametrised by a homeomorphism $Z \colon [0,1] \to A$. Since $[0,1]$ is compact, $Z$ is uniformly continuous. Then, for a given $\epsilon > 0$  by uniform continuity  we choose  $\delta > 0$ and correspondingly a  partition $\mathcal{D} : 0=t_0 < t_1 < ... < t_n = 1$ with $\max_{i = 1,...,n} |t_{i-1} - t_i| < \delta$. Then, $d^p(Z_{t_{i-1}},Z_{t_{i}}) < \epsilon$ and let $Y^{(i)}$ be the unique injective finite $p$-variation representative with arc $[Z_{t_{i-1}},Z_{t_i}] := Y^{(i)}([0,1])$ joining $Z_{t_{i-1}}$ and $Z_{t_i}$, for all $i \in \{1,...,n\}$. For any $i$, let $w \in [Z_{t_{i-1}},Z_{t_i}]$ and $u \in [0,1]$ such that $Y^{(i)}_u = w$. By monotonicity of $p$-variation along arcs, and since $Y^{(i)}|_{[0,u]}$ is the unique injective finite $p$-variation representative connecting $Z_{t_{i-1}}$ to $w$, we get
    \begin{equation} \label{eq:epsilon-ball}
        d^p(Z_{t_{i-1}},w) = \pvi{Y^{(i)}}{p}{[0,u]} \leq \pv{Y^{(i)}}{p} = d^{p}(Z_{t_{i-1}},Z_{t_i})< \epsilon.
    \end{equation}
    Hence, $[Z_{t_{i-1}},Z_{t_i}] \subset B_{d^p}(Z_{t_{i-1}},\epsilon)$. Now, consider the path $Z^{(\epsilon)} = (Y^{(1)} \sqcup ...) \sqcup Y^{(n)}$, which has finite $p$-variation by sub-additivity. Furthermore, due to \eqref{eq:epsilon-ball} we get
    \begin{equation*}
        Z^{(\epsilon)}([0,1]) \subseteq \bigcup_{i = 1}^n B_{d^p}(Z_{t_{i-1}},\epsilon).
    \end{equation*}
    Now consider the loop erasure $(Z^{(\epsilon)})^\circlearrowleft$ giving an injective arc of finite $p$-variation connecting $g$ and $h$. By \cref{thm:loop-erasure} we have $(Z^{(\epsilon)})^\circlearrowleft([0,1]) \subseteq Z^{(\epsilon)}([0,1])$. By \cref{lem:Uniqueness-arcs-p-variation-limit} (also ref.~\cite[Proof of Lemma 4.6]{BOEDIHARDJO_LYONS_2016720}) we have $(Z^{(\epsilon)})^\circlearrowleft([0,1]) = Y([0,1]) = [g,h]$ and hence
    \begin{equation} \label{eq:inclusion-of-injective-representative}
        [g,h] \subseteq \bigcup_{i = 1}^n B_{d^p}(Z_{t_{i-1}},\epsilon).
    \end{equation}
    Recall that by construction $Z_{t_i} \in A$. Hence, for any $k \in [g,h]$ and $\epsilon > 0$ it follows by \eqref{eq:inclusion-of-injective-representative} that there is an $i \in \{1,...,n\}$ such that $d^p(k,Z_{t_{i-1}}) < \epsilon$. Hence, we have the distance
    \begin{equation} \label{eq:vanishing-distance}
        \operatorname{dist}(k,A) := \inf_{a \in A} d^p(k,a) = 0.
    \end{equation}
    Since $A = Z([0,1])$ is compact (since $Z$ is continuous) w.r.t.~the topology induced by $d^p$,  it is closed and hence \eqref{eq:vanishing-distance} implies $k \in \overline{A} = A$. Since $k$ was any point of $[g,h]$, it follows that $[g,h] \subseteq A$. But $A$ was an arc from $g$ to $h$, $Z$ a homeomorphism with $Z_0 = g$ and $Z_1 = h$ and $[g,h]$ is pathwise connected containing $g$ and $h$. Thus, $Z^{-1}([g,h]) \subseteq [0,1]$ is connected and contains $\{0\}$ and $\{1\}$. Therefore, $Z^{-1}([g,h]) = [0,1]$. Hence, $[g,h] = A$ and so $PG_{\rho^p\textnormal{-}p.r.c.}$ is UAC. \\
    Finally, $d^p$-balls are arcwise connected. Indeed, take any $g \in PG_{\rho^p\textnormal{-}p.r.c.}$ and let $h,k\in B_{d^p}(g,r)$. Then, the set $[g,h] \cup [g,k]$ is arcwise connected with the common point $g$. Hence,
    \begin{equation*}
        [h,k]\subseteq [g,h]\cup[g,k].
    \end{equation*}
    If $l \in [h,k]$, then $l\in[g,h]$ or $l\in[g,k]$. In the first case,
    \begin{equation*}
        d^p(g,l)\leq d^p(g,h) < r,
    \end{equation*}
    since $[g,l] \subseteq [g,h]$ by monotonicity. Analogously in the other case. Hence, $[h,k]\subseteq B_{d^p}(g,r)$, so $B_{d^p}(g,r)$ is arcwise connected and $PG_{\rho^p\textnormal{-}p.r.c.}$ is LAC.
\end{proof}
Since $p$-variation for $p > 1$ can fail to capture local information, any non-injective minimiser includes a tree-like piece as can be seen from the examples mentioned in \cref{rem:Injectivity-reason}. In fact, by \cref{lem:Uniqueness-arcs-p-variation-limit} the set $\mathcal{A}_{\operatorname{inj}}(g,h)$ for any $g \neq h$ consists of only one element (up to reparametrisation), namely the minimiser.
\begin{remark}[comp.~\protect{\cite[Remark 4.1]{BOEDIHARDJO_LYONS_2016720}}, \protect{\cite[Example 3.1]{cass2024topologiesunparameterisedroughpath}}]
    The correspondence between \emph{tree-reduced} paths on $WG\Omega^p$ and the corresponding lifts with values in $G_{\rho^p\textnormal{-}p.r.c.}(V)$ is defined to be
    \begin{equation*}
        X \text{ is tree-reduced} : \Longleftrightarrow \cL(X) \text{ is injective},
    \end{equation*}
    due to the lack of uniqueness for the $p$-variation minimiser, in contrast to the $p=1$ case (ref.~\cite{HamblyLyonsUniqueness_2010}). Notice that tree-like reduction, that is, deletion of tree-like pieces in paths corresponds to the erasure of loops of the corresponding lifted path. We shall denote the map that deletes tree-like pieces of a path $X$ of finite $p$-variation in $WG\Omega^p$ by $\cdot^\tau \colon X \mapsto X^\tau$.
\end{remark}
Lastly, it is the main result of \cite{BOEDIHARDJO_LYONS_2016720} to characterise the kernel of the signature.
\begin{theorem}[comp.~\protect{\cite[Theorem 1.1]{BOEDIHARDJO_LYONS_2016720}}] \label{prop:Kernel-signature-characterisation}
    Let $S$ be the signature as defined before. Then
    \begin{equation*}
        \ker S = \{X \in C^{p-\var}([0,1],G^{{\floor{p}}}(V)) \mid X \text{ is tree-like}\},
    \end{equation*}
    where tree-like is to be understood in the sense of \cref{def:Tree-like-equivalence}. We write the corresponding equivalence classes then $[X]_\tau \in \mathcal{C}^p := WG\Omega^p / \sim$ and $[Y]_\circlearrowleft \in \cI^p$. 
\end{theorem}
\begin{remark}
  With some abuse of notation, we note that $\cL \colon \mathcal{C}^p \to \cI^p$ yields a group isomorphism.
\end{remark}
We term $\mathcal{C}^p$ the reduced path group (of weakly geometric rough paths), which can be identified with the signature group $SG^p$.
\begin{lemma} \label{lem:Image-of-ev1}
    Equip $\cI^p$ with the metric
    \begin{equation*}
        D^p([X]_\circlearrowleft,[Y]_{\circlearrowleft}) = \pv{\left((\rev{X})^\circlearrowleft\sqcup Y^\circlearrowleft\right)^\circlearrowleft}{p}\ , \qquad \forall [X]_\circlearrowleft,[Y]_{\circlearrowleft} \in \cI^p.
    \end{equation*}
    Then, $\eva_1|_{\cI^p} \colon (\cI^p,D^p) \to (G_{\rho^p\textnormal{-}p.r.c.},d^p)$ is an isometry onto its image, which is given as $SG^p = PG_{\rho^p\textnormal{-}p.r.c.}$.
\end{lemma}
\begin{proof}
   By definition of \(\mathcal I^p\), the image of \(\eva_1\) is precisely the finite \(p\)-variation path component of the identity, i.e.
    \begin{equation*}
        \eva_1(\mathcal I^p)=PG_{\rho^p\textnormal{-}p.r.c.}.
    \end{equation*}
    It remains to identify this set with \(SG^p\).\\
    Let $x:=X^\circlearrowleft_1$ and $y:=Y^\circlearrowleft_1$. By definition of $d^p$, the distance $d^p(x,y)$ is the $p$-variation of the injective path joining $x$ to $y$. This injective path is $$\left(\rev{X^\circlearrowleft}\sqcup Y^\circlearrowleft\right)^\circlearrowleft.$$
    Hence
    \begin{equation*}
        d^p(\eva_1([X]_\circlearrowleft),\eva_1([Y]_\circlearrowleft))
        = d^p(x,y)
        = \pv{(\rev{X^\circlearrowleft}\sqcup Y^\circlearrowleft)^\circlearrowleft}{p}
        = D^p([X]_\circlearrowleft,[Y]_{\circlearrowleft}).
    \end{equation*}
    Thus $\eva_1|_{\cI^p}$ is an isometric embedding.

    Let $[X]_\circlearrowleft$ belong to $\cI^p$, so that $X^\circlearrowleft$ is a finite $p$-variation path in $G_{\rho^p\textnormal{-}p.r.c.}$. Its projection $\Pi_{\floor p}\circ X^\circlearrowleft$ is therefore a weakly geometric $p$-rough path. By the unique $p$-path lifting property, $X^\circlearrowleft$ is the unique lift of this projected rough path to $G_{\rho^p\textnormal{-}p.r.c.}$. Hence
    \begin{equation*}
        X^\circlearrowleft_1 = S(\Pi_{\floor p}\circ X^\circlearrowleft)
    \end{equation*}
    and $\eva_1(\cI^p)\subseteq SG^p$.

    Conversely, given $g$ in $ SG^p$, there exists $X\in WG\Omega^p$ with $S(X)=g$.
    The lifted path $\cL(X)$ is a finite $p$-variation path in $G_{\rho^p\textnormal{-}p.r.c.}$ and loop erasure in $G_{\rho^p\textnormal{-}p.r.c.}$ gives an injective representative $\cL(X)^\circlearrowleft$ with the same endpoints. Hence $[\cL(X)]_\circlearrowleft\in\cI^p$ and
    \begin{equation*}
        \eva_1([\mathcal L(X)]_\circlearrowleft)=g.
    \end{equation*}
    Thus $\eva_1(\cI^p)=SG^p$.
\end{proof}
In the same way as above, we can find the corresponding metric on $\mathcal{C}^p$ using $S^{-1}$ to be
\begin{equation} \label{eq:Tree-metric-on-reduced-path-group}
    \delta^p([X]_\tau,[Y]_\tau) = \pv{(\rev{X}^\tau \sqcup Y^\tau)^{\tau}}{p}{},
\end{equation}
such that $\cL$ considered on the quotients will be an isometry as well. \\

\paragraph{\textbf{Diagram.}} Summing up, this establishes the right triangle of the diagram \eqref{eq:Diagram-p-variation-LZ}. \\

\paragraph{\textbf{The choice of the $\R$-tree metric.}} \cref{lem:Reformulation-Uniqueness-to-tree} showed that $(PG_{\rho^p\textnormal{-}p.r.c.}(V),d^p)$ is a topological tree.
From \cref{thm:Toptree-Rtree} it  then follows that there is a corresponding $\mathbb R$-tree of the form shown in \cref{thm:Height-Function-existence-Rtree-Toptree}, which will make the resulting metric space geodesic.

\begin{proposition} \label{prop:Tree-metric-p=1-case}
Given $[X]_\circlearrowleft,[Y]_{\circlearrowleft}$ in $\cI^1$, we have
    \begin{equation*}
        D^1([X]_\circlearrowleft,[Y]_{\circlearrowleft}) = \pv{(\rev X^\circlearrowleft \sqcup Y^\circlearrowleft)^\circlearrowleft}{1} = \pv{X^\circlearrowleft}{1} + \pv{Y^\circlearrowleft}{1} - 2 \pv{X^\circlearrowleft \land Y^\circlearrowleft}{1}
    \end{equation*}
    with height function $H(X) = \pv{X}{1}$.
\end{proposition}
We kept the $\cdot^\circlearrowleft$ for notational consistency, but may very well drop it in the $p=1$ case.
\begin{proof}
    The proof follows by additivity shown in \cref{sec:Canonical-metric-relations-goals}. Let both $X^\circlearrowleft,Y^\circlearrowleft$ be pointed injective paths of finite $1$-variation and the representatives of their respective equivalence class, with $X^\circlearrowleft_1 = x$ and $Y^\circlearrowleft_1 = y$. 
    By the UAC property, we have $X^\circlearrowleft = (X^\circlearrowleft\land Y^\circlearrowleft) \sqcup A$ and $Y^\circlearrowleft = (X^\circlearrowleft \land Y^\circlearrowleft) \sqcup B$ with $A([0,1]) \cap B[(0,1)] = \{x \land y\}$.
    Both $A$ and $B$ are injective paths of finite $1$-variation, and a direct computation yields
    \begin{equation*}
        \pv{X^\circlearrowleft}{1} = \pv{X^\circlearrowleft\land Y^\circlearrowleft}{1} + \pv{A}{1}, \qquad \pv{Y^\circlearrowleft}{1}{} = \pv{X^\circlearrowleft\land Y^\circlearrowleft}{1} + \pv{B}{1}.
    \end{equation*}
    Loop erasure implies that the path $X^\circlearrowleft \land Y^\circlearrowleft$ does not contribute as it is traversed twice and using the invariance of $1$-variation under $\rev{\cdot}$, we have
    \begin{equation*}
        \begin{aligned}
            \pv{(\rev X^\circlearrowleft \sqcup Y^\circlearrowleft)^\circlearrowleft}{1} &= \pv{\rev A \sqcup B}{1} = \pv{A}{1} + \pv{B}{1} \\
            &= \pv{X^\circlearrowleft}{1} + \pv{Y^\circlearrowleft}{1} - 2 \pv{X^\circlearrowleft\land Y^\circlearrowleft}{1}
        \end{aligned}
    \end{equation*}
    as claimed.
\end{proof}
It follows that the natural candidate for the height function in the case of $p$-variation is either $H(X^\circlearrowleft) = \pvd{X^\circlearrowleft}{p}{}^p$ or $\tilde H(X^\circlearrowleft) = \pvd{X^\circlearrowleft}{p}{}$ by the monotonicity properties of $p$-variation laid out in \cref{sec:Paths-over-pointed-metric-groups}. Due to the sub-additivity of $p$-variation or, equivalently, the super-additivity of the corresponding control, the proof of \cref{prop:Tree-metric-p=1-case} does not hold for $D^p$. However, it can be seen that the failure is caused by the lack of additivity of $p$-variation along arcs, which was already pointed out in \cref{sec:Paths-over-pointed-metric-groups}. Making the choice $H$ (or $\tilde H$) forces the additivity of $p$-variation along arcs at the branching points and one obtains the metric $D'$ (or $\tilde D$) as shown in \cite[Proposition 4.1]{BOEDIHARDJO_LYONS_2016720}.
\begin{lemma} \label{lem:euivalent-topologies}
    The metrics $D^p$ and
    \begin{equation*}
        \begin{aligned}
        D'([X]_\circlearrowleft,[Y]_{\circlearrowleft}) &= \pv{X^\circlearrowleft}{p}^p + \pv{Y^\circlearrowleft}{p}^p - 2 \pv{X^\circlearrowleft\land Y^\circlearrowleft}{p}^p\,, \\ 
        \tilde D([X]_\circlearrowleft,[Y]_{\circlearrowleft}) &= \pv{X^\circlearrowleft}{p} + \pv{Y^\circlearrowleft}{p} - 2 \pv{X^\circlearrowleft\land Y^\circlearrowleft}{p}
        \end{aligned}
    \end{equation*}
    induce the same topology.
\end{lemma}
The proof uses the sub-additivity property of $p$-variation and the super-additivity property of the control. Recall that for a set $E$ with metrics $d$ and $d'$, the induced topologies agree if convergence of sequences w.r.t.~$d$ implies convergence w.r.t.~$d'$ and vice versa.
\begin{proof}
    Let both $X^\circlearrowleft,Y^\circlearrowleft$ be pointed injective paths of finite $p$-variation and the representatives of their respective equivalence class, with $X^\circlearrowleft_1 = x$ and $Y^\circlearrowleft_1 = y$. We notice that $H(X) = (\tilde H(X))^p$. Let $Z = X^\circlearrowleft\land Y^\circlearrowleft$ and decompose  the paths  writing $X^\circlearrowleft = Z \sqcup A$ and $Y^\circlearrowleft = Z \sqcup B$ as in the proof of \cref{prop:Tree-metric-p=1-case}  with $A([0,1]) \cap B[(0,1)] = \{x \land y\}$ and both $A$ and $B$ injective paths of finite $p$-variation. \\
    By $\|X|_{[s,t]}\|_{p\textnormal{-}\var} \leq \|X\|_{p\textnormal{-}\var}$ for any $s,t \in [0,1]$ with $s \leq t$, subadditivity of $p$-variation and loop erasure, we have
    \begin{equation*}
        \max\{\pv{A}{p}{},\pv{B}{p}{}\}\leq D^p([X]_\circlearrowleft,[Y]_{\circlearrowleft}) \leq \pv{A}{p}{} + \pv{B}{p}{}
    \end{equation*}
    and by superadditivity
    \begin{equation*}
        \begin{aligned}
            \pv{A}{p}^p &\leq \pv{X^\circlearrowleft}{p}^p - \pv{Z}{p}^p, \\
            \pv{B}{p}^p &\leq \pv{Y^\circlearrowleft}{p}^p - \pv{Z}{p}^p.
        \end{aligned}
    \end{equation*}
   Let $[Y_n]_\circlearrowleft$ be a sequence and set $Z_n:=X^\circlearrowleft\land Y_n^\circlearrowleft$, with decompositions $X^\circlearrowleft=Z_n\sqcup A_n$ and $Y_n^\circlearrowleft=Z_n\sqcup B_n$.

    We first assume that $D'([X]_\circlearrowleft,[Y_n]_\circlearrowleft)\to0$. Then, both non-negative summands $\pv{X^\circlearrowleft}{p}{}^p-\pv{Z_n}{p}{}^p$ and $\pv{Y_n^\circlearrowleft}{p}{}^p-\pv{Z_n}{p}{}^p$ converge to zero. Hence the estimates above imply $\pv{A_n}{p}{}\to0$ and $\pv{B_n}{p}{}\to0$. Therefore
    \begin{equation*}
        D^p([X]_\circlearrowleft,[Y_n]_\circlearrowleft)
        \leq
        \pv{A_n}{p}+\pv{B_n}{p}
        \to0.
    \end{equation*}

    Conversely, suppose that $D^p([X]_\circlearrowleft,[Y_n]_\circlearrowleft)\to0$. Then, $\pv{A_n}{p}{}\to0$ and $\pv{B_n}{p}{}\to0$. Since $X^\circlearrowleft=Z_n\sqcup A_n$, subadditivity and the decomposition of $X$ imply
    \begin{equation*}
        0
        \leq
        \pv{X^\circlearrowleft}{p}{}-\pv{Z_n}{p}{}
        \leq
        \pv{A_n}{p}{}
        \to0.
    \end{equation*}
    Similarly, $0\leq \pv{Y_n^\circlearrowleft}{p}{}-\pv{Z_n}{p}{}\leq \pv{B_n}{p}{}\to0$. Since $r\mapsto r^p$ is continuous on bounded intervals, it follows that $D'([X]_\circlearrowleft,[Y_n]_\circlearrowleft)\to0$. Thus $D^p$ and $D'$ induce the same topology.

    Finally, $\tilde D([X]_\circlearrowleft,[Y_n]_\circlearrowleft)$ is the sum of the two non-negative terms $\pv{X^\circlearrowleft}{p}{}-\pv{Z_n}{p}{}$ and $\pv{Y_n^\circlearrowleft}{p}{}-\pv{Z_n}{p}{}$, whereas $D'([X]_\circlearrowleft,[Y_n]_\circlearrowleft)$ is the same expression with each term replaced by the corresponding difference of $p$-th powers. Since the involved quantities are non-negative and uniformly bounded near $X^\circlearrowleft$, continuity and strict monotonicity of $r\mapsto r^p$ imply
    \begin{equation*}
        \tilde D([X]_\circlearrowleft,[Y_n]_\circlearrowleft)\to0
        \quad\Longleftrightarrow\quad
        D'([X]_\circlearrowleft,[Y_n]_\circlearrowleft)\to0.
    \end{equation*}
    Hence $D^p$, $D'$ and $\tilde D$ induce the same topology.
\end{proof}

\paragraph{\textbf{Continuity of lifts.}}  
Since \cref{def:LeDonne-unique-path-lifting-p-variation} provided a $p$-variation analogue of the unique path lifting in \cite[Definition 2.5]{LeDonneZuest_public}, also the continuous path lifting \cite[Definition 2.8]{LeDonneZuest_public} can be given a $p$-variation analogue as follows.
\begin{definition}
    Let $(E,d_E)$ and $(F,d_F)$ be pointed metric spaces. A map $\pi \colon E \to F$ has the \emph{continuous $p$-path lifting property} if it has the unique path $p$-lifting property and for each sequence of paths $X^{(n)} \colon [0,1] \to F$, $n \in \N$, with
    \begin{enumerate}[label = (\roman*)]
        \item $X^{(n)}_0 = \xi$ for all $n$ in  $\N$,
        \item $\sup_{n\geq 1} \pvd{X^{(n)}}{p}{d_F} < \infty$,
        \item $X^{(n)} \to X$ pointwise for all $t$ in $[0,1]$,
    \end{enumerate}
    it holds that
    \begin{equation*}
        Y^{(n)} \to Y \quad \textrm{ pointwise for all $t$ in} [0,1],
    \end{equation*}
    where $Y^{(n)}$ and $Y$ are lifts of $X^{(n)}$ and $X$ respectively, starting at $x$ with $\pi(x) = \xi$.
\end{definition}
It is then possible to also carry over \cite[Proposition 2.9]{LeDonneZuest_public} to the $p$-variation setting.
\begin{proposition} \label{prop:Geodesic-triangular-argument}
    Assume that for every $k $ in $\N$, the pointed metric space $(E_k, d_k)$ is $p$-variation geodesic and proper. If all the projections $\pi_{k+1} \colon E_{k+1} \to E_k$ are base point-preserving\footnote{A map between pointed spaces $\pi\colon (E,e) \to (F,f)$ is basepoint preserving if $\pi(e) = f$.} and have the continuous $p$-path lifting property, then the limit $(E_\infty,d_\infty)$ is a complete $p$-variation geodesic metric space and for each pair of points $x = (x_1,x_2,...),y = (y_1,y_2,...) \in E_\infty$ there exist $Y^{(\infty)} \colon [0,1] \to E_\infty$ and $X^{(\sigma(n))} \colon [0,1] \to E_{\sigma(n)}$ for a strictly increasing sequence $\sigma \colon \N \to \N$ such that:
    \begin{enumerate}[label = (\roman*)]
        \item $Y^{(\infty)}$ is a Hölderian geodesic path w.r.t.~$p$-variation parametrized to be $1/p$-Hölder connecting $x$ with $y$ in $E_\infty$.
        \item Each $X^{\sigma(n)}$ is a Hölderian geodesic path w.r.t.~$p$-variation parametrized to be $1/p$-Hölder connecting $x^{\sigma(n)}$ with $y^{\sigma(n)}$.
        \item $\pi^\infty_k \circ Y = \lim_{n \to \infty} \pi^{\sigma(n)}_k\circ X^{(\sigma(n))}$ for all $k $ in $\N$.
    \end{enumerate}
\end{proposition}
For completeness we provide a sketch of the proof adapting to the $p$-variation case the proof of \cite[Proposition 2.9]{LeDonneZuest_public} in the rectifiable case.

\begin{proof}[Proof]
    Completeness of $(E_\infty,d_\infty)$ follows by properness of each $(E_k,d_k)$ and \cite[Lemma 2.2]{LeDonneZuest_public}. Similarly to the length case, let $X^{(k)} \colon [0,1] \to E_k$ be a Hölderian geodesic w.r.t.~$p$-variation connecting $x_k$ to $y_k$. Since $X^{(k)}$ has finite $p$-variation we can choose the minimal control $\omega_X$ so that $[X^{(k)}]_{1/p} \leq d_k(x_k,y_k)$ for all $k \in \N$, where $[X_k]_{1/p} = \pv{X}{p}{}$ is the Hölder constant.\footnote{Note that $[X]_1 = \operatorname{Lip(X)}$ as in \cite{LeDonneZuest_public}.} 
    Since the unique $p$-path lifting holds for the limiting projections and $1$-Lipschitz maps, the inequality
    \begin{equation*}
        [Y^{(k)}]_{1/p} \leq [X^{(k)}]_{1/p} \leq d_k(x_k,y_k) \leq d_\infty(x,y)
    \end{equation*}
    also holds for the lift $Y^{(n)}\colon [0,1] \to E_\infty$.
    The path $Y^{(k)}$ is therefore contained in the closed ball $\bar{B}(x,d_\infty(x,y))$. By the $1$-Lipschitz property of each limiting projection $\pi^\infty_k \colon (E_\infty,d_\infty) \to (E_k,d_k)$, $\pi^\infty_n \circ Y^{(k)}_t$ lies in $\bar{B}(x_n,d_\infty(x,y))$, starts at $x_n$ and also satisfies $[\pi^\infty_n \circ Y^{(k)}]_{1/p} \leq d_\infty(x,y)$ for all $n \in \N$. Since $\bar{B}(x_n,d_\infty(x,y)) \subseteq E_n$ is compact by assumption, the theorem of Arzelà-Ascoli \cite[Theorem 1.4]{Friz_Victoir_2010_fullbook} guarantees the existence of a subsequence $(\sigma(n,k))_{k \in \N}$ with $\sigma(1,k) = k$ and $\sigma(n,1) = n$, such that we can conclude recursively that $(\pi^\infty_{n} \circ Y^{(m(n,k))})_{n \in \N}$ converges uniformly to some path $Z^{(k)}$. For the diagonal sequence, we have
    \begin{equation*}
        \lim_{k \to \infty} \pi^\infty_n \circ Y^{\sigma(k,k)} = Z^{(n)}
    \end{equation*}
    uniformly for all $n$ in $ \N$.
    As in \cite[Proposition 2.9]{LeDonneZuest_public}, the following properties of these limits hold for all $n$ in $\N$:
    \begin{enumerate}[label = (\arabic*)]
        \item $Z^{(n)}_0 = x_n$, \label{prop:Geodesic-triangular-argument-property-limit-startpoint}
        \item $Z^{(n)}_1 = y_n$, \label{prop:Geodesic-triangular-argument-property-limit-endpoint}
        \item $[Z^{(n)}]_{1/p} \leq d_\infty(x,y)$, \label{prop:Geodesic-triangular-argument-property-limit-Hoelder}
        \item $\pi^{n+1}_n(Z^{(n+1)}) = Z^{(n)}$. \label{prop:Geodesic-triangular-argument-property-limit-projection}
    \end{enumerate}
Properties \ref{prop:Geodesic-triangular-argument-property-limit-startpoint} and \ref{prop:Geodesic-triangular-argument-property-limit-endpoint} are clear by construction, ref. \cite[Proposition 2.9]{LeDonneZuest_public}.
Property \ref{prop:Geodesic-triangular-argument-property-limit-Hoelder} follows from $[\pi_n^\infty \circ Y^{(k)}]_{1/p} \leq d(x,y)$ for all $n,k $ in $\N$. 
Property \ref{prop:Geodesic-triangular-argument-property-limit-projection} follows from the unique $p$-path lifting property with \cref{rem:LeDonneZuest-proof-2.6-and-2.7} with fixed starting point. The unique path lifting assumption then implies that $Z^{(n+1)}$ is the lift of $Z^{(n)}$. \\
    From \ref{prop:Geodesic-triangular-argument-property-limit-endpoint}, \ref{prop:Geodesic-triangular-argument-property-limit-Hoelder} and \ref{prop:Geodesic-triangular-argument-property-limit-projection} and \cref{rem:LeDonneZuest-proof-2.6-and-2.7} it follows that the limit path $Z^{(\infty)} = (Z^{(1)},Z^{(2)},...)$ over $E_\infty$ for each $n \in \N$ is the lift of $Z^{(n)}$ starting at $x$. Moreover, $[Z^{(\infty)}]_{1/p} \leq d_\infty(x,y)$ by \ref{prop:Geodesic-triangular-argument-property-limit-Hoelder} and \cref{rem:LeDonneZuest-proof-2.6-and-2.7}. Furthermore,
    \begin{equation*}
        d_\infty(x,y) \leq \pvd{Z^{(\infty)}}{p}{d_\infty} \leq [Z^{(\infty)}]_{1/p} \leq d_\infty(x,y).
    \end{equation*}
    Together with \ref{prop:Geodesic-triangular-argument-property-limit-endpoint} we have $Z^{(\infty)}_1 = y$ and hence $Z$ is a Hölderian geodesic w.r.t.~$p$-variation connecting $x$ and $y$. The bound $[Z^{(\infty)}]_{1/p} \leq d_\infty(x,y)$ implies also that $Z^{(\infty)}$ is also parametrised to be $1/p$-Hölder.
\end{proof}
\begin{remark}
    Note that the above statement shows that \cite[Proposition 2.9]{LeDonneZuest_public} only depends on the admission of geodesics between endpoints and hence could be stated in even greater generality under the minimisation of general cost functionals as in optimal control.
\end{remark}

\paragraph{\textbf{Topologies on the signature group.}} To conclude this section, we briefly discuss further topologies on the reduced path group.
One straightforward choice is the product topology on $SG^p$ pulled back via $S^{-1}$. Since the signature group is closed w.r.~to its inherited multiplication, both the signature group and reduced path group are topological groups in this case. Another candidate was given by $\delta^p$ in \eqref{eq:Tree-metric-on-reduced-path-group} for the reduced path group and $D^p$ for the signature group, which induced the tree topology. In this sense, there are many topologies one might consider, some of which have already been investigated in detail in \cite{cass2024topologiesunparameterisedroughpath}, where it was shown in \cite[Remark 3.7]{cass2024topologiesunparameterisedroughpath} that the topology induced by $\delta^p$ makes the resulting topological space non-separable. 
This is not a favourable situation for probabilists when investigating measures on these infinite dimensional spaces since many results in probability require the underlying metric space to be Polish. Another metric that one might equip path space with is \cite{cass2024topologiesunparameterisedroughpath}
\begin{equation*}
    \delta^p_\star([X]_\tau,[Y]_\tau) = \max_{j = 1,...,{\floor{p}}} \left(\sup_{\mathcal D} \sum_{t_i \in \mathcal{D}} \norm{\proj_j(X^\tau_{t_{i},t_{i+1}}) - \proj_j(Y^\tau_{t_i,t_{i+1}})}_j^{p/j} \right)^{j/p}.
\end{equation*}
Note that it was proven in \cite[Proposition 4.10]{cass2024topologiesunparameterisedroughpath} that $(\mathcal{C}^1,\delta^1_\star)$ is Polish. We suspect that this will fail for $p > 1$ due to the fact that Hölder spaces in general are not separable.\\
In addition, given the above discussion and the next section, the initial diagram \eqref{eq:Diagram-p-variation-LZ} of \cref{thm:Main-Theorem-Commutative-Diagram} determines a canonical metric on the corresponding path space associated to $\cG_p$ together with the fixed choice $d \in \metlift(\cG_p)$ by the group isomorphism between the signature group and the inverse limit.

\section{The canonical metric relations}

\label{sec:Canonical-metric-relations-goals}

In this section we establish the left hand side of the diagram \eqref{eq:Diagram-p-variation-LZ}, give a characterisation of the homogeneous metrics in $\metlift(\cG_p)$ and conclude the approximation corollaries. \\

\paragraph{\textbf{Intrinsic $p$-variation.}} The $p = 1$ case considered in \cite{LeDonneZuest_public} together with \cref{prop:Tree-metric-p=1-case} and the considerations of the intrinsic $p$-variation w.r.t.~$(G_{\rho^p\textnormal{-}p.r.c.}(V),d^p)$ with \cite[Remark 4.1]{BOEDIHARDJO_LYONS_2016720} motivate the introduction of the analogue of the Carnot-Carathéodory norm for the weakly geometric rough paths on each level, which we define in the abstract sense via the corresponding control
\begin{equation} \label{eq:CC-p-variation-norm}
    \cost_{k}^p(g) = \inf_{X \in \mathcal{A}_k(g)} \pvd{X}{p}{d_{\floor{p}}}^p, \quad \mathcal{A}_k(g) = \{ X \in WG\Omega^p_{d_{\floor{p}}} \mid \, S_k(X) = g \}
\end{equation}
where $S_k$ is understood via the unique $p$-path lifting defined in \cref{sec:Rough-Path-Theory} w.r.t.~a family of metric $d \in \metlift(\cG_p)$. In the case of $d = \rho^p$ or a homogeneous family of metrics the truncated signature $S_k$ is the Lyons lift. \\
Notice that $\mathcal{A}_k$ is well-defined and reduces to $\mathcal{A}^{\CC}_k$ for $p=1$, which is the Carnot-Carathéodory setting. Hence, we call by analogy $\mathcal{A}_k(g)$ again the set of \emph{admissible (horizontal) paths}. Set $\mcost^p_k(g,h) := \cost^p_k(g^{-1} h)$. Note that each $\mcost^{p}_k$ is not necessarily homogeneous\footnote{If we require all $d_k$'s to be homogeneous, it follows immediately that $\mcost^p_k(\delta_\lambda(g),\delta_{\lambda}(h)) = \lambda^p \mcost^p_k(g,h)$ for $\lambda \geq 0$.}, let alone a metric \textit{a priori} by \eqref{eq:p-Variation-inequalities-super}. \\
Since \cite[Section 4]{LeDonneZuest_public} uses knowledge of \cite[Chapter 7]{Friz_Victoir_2010_fullbook}, for completeness, we find it necessary to provide proofs of several propositions, in particular, \cite[Proposition 7.59]{Friz_Victoir_2010_fullbook} for any $p \geq 1$.
\begin{lemma} \label{lem:properties-of-dpCC}
    Let $\mcost^p_k$ be defined as above.
    For $d $ in $\metlift(\cG_p)$, let $d^{p\textnormal{-}\CC}_k(\cdot,\cdot) := (\mcost_k^p(\cdot,\cdot))^{1/p}$. Then
    \begin{enumerate}[label = (\alph*)]
        \item $\mcost_k^p$ is symmetric and $\left(\mcost_k^p(1,g) = 0\right) \Longleftrightarrow \left(g = 1_k\right)$ for every $g$ in $G^k(V)$,
        \item $d^{p\textnormal{-}\CC}_k$ satisfies the triangle inequality and therefore is an extended\footnote{Recall that this refers to a metric $d \colon E\times E \to [0,\infty) \cup \{\infty\}$.} metric,
        \item if each $d_k$ is homogeneous, then $d^{p\textnormal{-}\CC}_k$ is homogeneous.
    \end{enumerate}
\end{lemma}
For completeness let us add the elementary proof which is analogous to the one for the $\CC$-metric.
\begin{proof}
    \begin{enumerate}
        \item By definition of $p$-variation, we get that $\mcost_k^p$ has range $[0,\infty]$ and is therefore well-defined. Since for any path $X$ and $\rev{X}$ have the same $p$-variation and reversal of the path is the group inversion of the end point, we get that $\cost_k^p(g) = \cost^p_k(g^{-1})$ for all $g \in G^k(V)$, and hence $\mcost$ is symmetric. Since $o\colon t \mapsto 1_{\floor{p}} \in \mathcal{A}_k(1_k)$, we get
        \begin{equation*}
            0 \leq \cost^p_k(1_k) \leq \pvd{o}{p}{d_{{\floor{p}}}} = 0.
        \end{equation*}
        For the converse, assume $\cost^p_k(g) = 0$, then for any $X \in \mathcal{A}_k(g)$ we have
        \begin{equation*}
            d_k(1_k,g) \leq \pvd{\cL_k(X)}{p}{d_k} = \pvd{X}{p}{d_{{\floor{p}}}}.
        \end{equation*}
        and hence by using the infimum on the right hand side, we get $d_k(1_k,g)^p \leq \cost_k^p(g) = 0$. Since $d_k$ is a metric and $p \geq 1$, it follows that $g = 1_k$.

        \item Let $(\varepsilon_n)_{n \in \N}$ be a sequence of positive reals with $\varepsilon_n \downarrow 0$, $X^{(n)}$ a sequence in $\mathcal A_k(g)$ and $Y^{(n)}$ a sequence in $\mathcal A_k(h)$ 
        such that
        \begin{equation*}
            \pvd{X^{(n)}}{p}{d_{\floor{p}}} \leq d_k^{p\textnormal{-}\CC}(1_k,g) + \varepsilon_n, \quad \pvd{Y^{(n)}}{p}{d_{\floor{p}}} \leq d_k^{p\textnormal{-}\CC}(1_k,h) + \varepsilon_n, \quad 
            \forall n\in \N.
        \end{equation*}
        Since $\rev{X^{(n)}} \sqcup Y^{(n)}$ 
        lies in $\mathcal{A}_k(g^{-1}h)$ for any $n$ in $\N$, by \eqref{eq:p-Variation-inequalities-sub} 
       it follows that
        \begin{equation*}
            \begin{aligned}
                d^{p\textnormal{-}\CC}_k(g,h) &\leq \pvd{\rev{X^{(n)}}\sqcup Y^{(n)}}{p}{d_{{\floor{p}}}} \leq \pvd{X^{(n)}}{p}{d_{{\floor{p}}}} + \pvd{Y^{(n)}}{p}{d_{{\floor{p}}}} \\
                &\leq d^{p\textnormal{-}\CC}_k(1_k,g) + d^{p\textnormal{-}\CC}_k(1_k,h) + 2 \varepsilon_n.
            \end{aligned}
        \end{equation*}
        Letting $n \to \infty$ proves the triangle inequality.

        \item This is a consequence of the homogeneity of $d_{\floor{p}}$ combined with the identity $S_k \circ \delta_\lambda = \delta_\lambda \circ S_k$ and 
        \begin{equation*}
            \pvd{\delta_\lambda(X)}{p}{d_{\floor{p}}} = |\lambda| \pvd{X}{p}{d_{\floor{p}}}
        \end{equation*}
    \end{enumerate}
\end{proof}
\begin{remark} \label{rem:Category-change}
        Let $X$ be a path of finite $p$-$d_k$-variation into $(G^k(V),d_k)$,
        by construction
        \begin{equation*}
            \pvd{X}{p}{d^{p\textnormal{-}\CC}_k}^p = \pvd{X}{1}{\mcost^p_k}
        \end{equation*}
        and the above holds for any restriction of $X$ to a subinterval $[s,t] \subseteq [0,1]$, as well.
\end{remark}

\paragraph{\textbf{Unique path lifting.}} Just as in the case $p=1$, using the lifting of paths  given by the Lyons extension, we can establish the lifting property w.r.t.~each $d^{p\textnormal{-}\CC}_k$ in the case of general path lifts.
\begin{lemma} \label{lem:p-variation=1-variation-p-CC}
   Given $X$ in $WG\Omega^p_d$, the following identity holds
   \begin{equation*}
        \pvd{X}{p}{d_{\floor{p}}} = \pvd{X}{p}{d^{p\textnormal{-}\CC}_{\floor{p}}}.
   \end{equation*}
\end{lemma}
\begin{proof}
With the notations of the lemma, by definition we have the identity
    \begin{equation*}
       \mcost^p_{\floor{p}}(X_s,X_t) \leq \pvdi{X}{p}{d_{{\floor{p}}}}{[s,t]}^p\quad  \forall (s,t) \in \Delta
    \end{equation*}
    with $\Delta$ the standard simplex over the unit interval.
    Hence, for any partition $\mathcal{D}$ of the interval $[0,1]$, we get
    \begin{equation*}
        \sum_{t_i \in \mathcal{D}} d^{p\textnormal{-}\CC}_{\floor{p}}(X_{t_{i-1}},X_{t_i})^p \leq \sum_{t_i \in \mathcal{D}} \pvdi{X}{p}{d_{{\floor{p}}}}{[t_{i-1},t_i]}^p.
    \end{equation*}
    Taking the supremum over the partitions on both sides then shows that $$\pvd{X}{p}{d^{p\textnormal{-}\CC}_{\floor{p}}}\leq \pvd{X}{p}{d_{\floor{p}}}.$$ To show the other direction, for $s,t$ in $\Delta$ we take the infimum over the admissible paths starting at $t$ and ending at $s$, which yields
    \begin{equation*}
        d_{\floor{p}}(X_s,X_t)^p \leq \mcost^p_{\floor{p}}(X_s,X_t).
    \end{equation*}
 Taking the supremum over partitions on both sides yields
 the desired result.
\end{proof}
Given any $g \in G^k(V)$ for a $k \geq {\floor{p}}$, notice that the sets $(\mathcal A_j(\pi^k_j(g)))_{j = {\floor{p}}}^k$ are nested, 
\begin{equation*}
    \mathcal A_k(g) \subseteq \mathcal A_{k-1}(\pi^k_{k-1}(g)) \subseteq \cdots \subseteq \mathcal A_{\lfloor p\rfloor}(\pi^k_{\lfloor p\rfloor}(g)).
\end{equation*}
Hence, the sequence $d^{p\textnormal{-}\CC} = (d^{p\textnormal{-}\CC}_k)_{k \in \N_p}$ makes the canonical projections $1$-Lipschitz.
\begin{proposition} \label{prop:Isometry-Property-Lift}
    Given $X$ in $WG\Omega^p_{d_{\lfloor p\rfloor}}$, we have
    \begin{equation*}
        \pvd{\cL_k(X)}{p}{d_k^{p\textnormal{-}\CC}} = \pvd{X}{p}{d^{p\textnormal{-}\CC}_{\floor{p}}}.
    \end{equation*}
\end{proposition}
\begin{proof}
The proof follows the same arguments as the proof of \cref{lem:p-variation=1-variation-p-CC}.
Given $X$ in $WG\Omega^p_d$, from the nested property of the sets of admissible paths, it follows that
    \begin{equation*}
        \pvd{\cL_k(X)}{p}{d_k^{p\textnormal{-}\CC}} \geq \pvd{\pi^k_{\floor{p}} \circ \cL_k(X)}{p}{d^{p\textnormal{-}\CC}_{\floor{p}}} = \pvd{X}{p}{d^{p\textnormal{-}\CC}_{\floor{p}}}.
    \end{equation*}
    Using \cref{lem:p-variation=1-variation-p-CC}, we moreover have that
    \begin{equation*}
        \pvd{\cL_k(X)}{p}{d_k^{p\textnormal{-}\CC}}^p = \pvd{X}{p}{{d^{p\textnormal{-}\CC}_{\floor{p}}}}^p.
    \end{equation*}
    Conversely, for any $s,t$ in $\Delta$ it holds that
    \begin{equation*}
        d^{p\textnormal{-}\CC}_k(\cL_k(X)_s,\cL_k(X)_t)^p \leq \pvdi{X}{p}{d_{\floor{p}}}{[s,t]}^p.
    \end{equation*}
    Taking again the supremum over partitions on both sides ends the proof.
\end{proof}
\begin{remark} \label{rem:equality-of-p-variations}
    The two propositions above recover \ref{def:LeDonne-unique-path-lifting-p-variation-isometry} of \cref{def:LeDonne-unique-path-lifting-p-variation}. 
    Indeed, by combining the previous results we get back that
    \begin{align*}
            \pvd{\cL_k(X)}{p}{d^{p\textnormal{-}\CC}_k} = \pvd{X}{p}{d^{p\textnormal{-}\CC}_{\floor{p}}} = \pvd{X}{p}{d_{\floor{p}}} = \pvd{\cL_k(X)}{p}{d_k}.
    \end{align*}
\end{remark}
We now return to the characterisation question following \cref{eq:Set-of-left-invariant-metrics-with-Lifts}.

\begin{proposition}[Metric characterisation] \label{prop:homogeneous-metric-characterisaton}

By equivalence of homogeneous metrics, so that $\cL_k$ and $S_k$ are understood via Lyons' extension theorem independently of the choice of the family of homogeneous metrics $d$ in $\met(\cG_p)$, we have the characterisation
    \begin{equation*}
        d = (d_k)_{k \in \N_p} \in \metlift(\cG_p) \Longleftrightarrow d_k \leq d_k^{p\textnormal{-}\CC}, \quad \forall k \in \N_p.
    \end{equation*}
\end{proposition}

\begin{proof}
    Let $d = (d_k)_{k \in \N_p} \in \metlift(\cG_p)$. Fix a $k$ in $\N_p$, $g$ in $G^k(V)$ and let $X \in \mathcal{A}_k(g)$. By definition $S_k(X) = g$ and we get by the isometry property of the unique $p$-path lifting \ref{def:LeDonne-unique-path-lifting-p-variation-isometry} of \cref{def:LeDonne-unique-path-lifting-p-variation}
    \begin{equation*}
        d_k(1_k,g) \leq \pvd{\cL_k(X)}{p}{d_k} = \pvd{X}{p}{d_{\floor p}}
    \end{equation*}
    Taking the infimum over all paths of $\mathcal{A}_k(g)$, using the definition of $d^{p\textnormal{-}\CC}_k$ and noting that $k$ was arbitrary, it follows pointwise that $d_k \leq d_k^{p\textnormal{-}\CC}$. \\
    
    For the converse, let $d = (d_k)_{k \in \N_p} \in \met(\cG_p)$ satisfying $d_k \leq d^{p\textnormal{-}\CC}_k$ pointwise for all $k \in \N_p$. Recall that $\cL_k$ is here understood via the Lyons' lift. Then, we need to check the properties of \cref{def:LeDonne-unique-path-lifting-p-variation}, where due to the Lyons' lift it suffices to check the isometry property \ref{def:LeDonne-unique-path-lifting-p-variation-isometry}. Let $X \in WG\Omega^p_{d_{\floor p}}$. Since by assumption $d \in \met(\cG_p)$, the canonical projections $\pi^{k+1}_k$ are $1$-Lipschitz and it holds $X = \pi^k_{\floor p} \circ \cL_k(X)$. Thus, using the $1$-Lipschitz bound for the $p$-variation, the pointwise metric inequality, \cref{def:LeDonne-unique-path-lifting-p-variation-isometry} under the Lyons' lift, and finally \cref{lem:p-variation=1-variation-p-CC} gives
    \begin{equation*}
        \pvd{X}{p}{d_{\floor p}} \leq \pvd{\cL_k(X)}{p}{d_k} \leq \pvd{\cL_k(X)}{p}{d^{p\textnormal{-}\CC}} = \pvd{X}{p}{d^{p\textnormal{-}\CC}} = \pvd{X}{p}{d_{\floor p}}.
    \end{equation*}
    Hence, $d \in \metlift(\cG_p)$.
\end{proof}

To attain minimisers in the $p>1$ case for general metrics requires both an assumption on $(G^{\floor p},d_{{\floor{p}}})$ and that $S_k$ is sequentially continuous for fixed $k \in \N_p$. We shall formulate the necessary conditions using controls instead, cf. \cref{lem:properties-of-dpCC}. A sufficient version mimicking the $p=1$ case, would be the following.

\begin{proposition} \label{prop:Attainment-of-minimisers}
    For $k$ in $\N_p$,  
    assume $(G^{\floor p}(V),d_{{\floor{p}}})$ is proper as a metric space, $S_k$ is sequentially continuous w.r.t.~to uniform convergence in $d_{\floor p}$ and each $G^k(V)$ is path connected, where there exists at least one path connecting any pair of points with finite $p$-variation.\footnote{We specifically keep \eqref{eq:Diagram-p-variation-LZ} in mind for which continuity is a given by the isometry property.} Then, for any $g \in G^k(V)$, $\cost^p_k(g)$ is finite and the infimum attained.
\end{proposition}
Recall that any finite dimensional vector space $V$ equipped with a norm is a proper metric space and therefore the properness requirement is the analogue for the ${\floor{p}} \geq 2$ setting.
\begin{proof}
    The finiteness follows by \cref{lem:Homomorphism-And-Image-of-endpoint-map}. By the definition of the infimum, there exists a sequence $(X^{(n)})_{n \in \N} \subset WG\Omega_d^p$ with $S_k(X^{(n)}) = g$ and with $c_n = \pvd{X^{(n)}}{p}{d_{\floor{p}}}^p$ we have $c_n \downarrow \cost^p_k(g)$ as $n\rightarrow \infty$. \\
    Since the sequence is decreasing and each path has finite $p$-$d_{\floor{p}}$-variation, we can set $M := \sup_{n \in \N} c_n < \infty$ (the sequence is bounded). Hence, for all $t$ in $ [0,1]$ and any natural number, the inequality $d_{\floor{p}}(1,X^{(n)}_t)^p \leq c_n \leq M$ holds. It follows that $d_{\floor{p}}(1,X^{(n)}_t) \leq M^{1/p}$. 
    Furthermore, $X^{(n)}$ lies inside a closed ball which is compact by the properness of $(G^{\floor p}(V),d_{\floor p})$. Using the Hölder reparametrisation for each $X^{(n)}$, we further conclude that the sequence is equicontinuous and infer the existence of a subsequence $(X^{(n_m)})_{m \in \N}$ that converges uniformly (w.r.t.~$d_{{\floor{p}}}$) to a continuous limit path $X^*$.
    Without loss of generality, we can choose $(X^{(n)})_{n \in \N}$ to be that sequence. 
    We have uniform convergence and hence point-wise convergence (w.r.t.~$d_{{\floor{p}}}$). 
    Consequently,
    \begin{equation*}
        \pvd{X^*}{p}{d_{{\floor{p}}}}^p \leq \liminf_{n \to \infty} \pvd{X^{(n)}}{p}{d_{{\floor{p}}}}^p,
    \end{equation*}
    and $X^*$ is also Hölder with regularity $1/p$. By sequential continuity, it follows from $S_k(X^{(n)}) = g$ that $S_k(X^*) = g$. \\
    To see that there exists such a path $X^*$, so that $\cost^p_k(g) = \pvd{X^*}{p}{d_{{\floor{p}}}}^p$, note that by definition $\cost^p_k(g) \leq \pvd{X^*}{p}{d_{{\floor{p}}}}^p$. On the other hand from the lower semicontinuity it follows that
    \begin{equation*}
        \pvd{X^*}{p}{d_{{\floor{p}}}}^p \leq \liminf_{n \to \infty} c_n = \cost^p_k(g).
    \end{equation*}
    This ends the proof.
\end{proof}
\begin{remark}
    In \cref{def:LeDonne-unique-path-lifting-p-variation} the projections $\pi^{k+1}_k \colon (G^{k+1}(V),\rho^p_{k+1}) \to (G^k(V),\rho^p_k)$ were not assumed to be submetries. 
    Indeed, this feature was shown in \cref{counterexample:Submetry-Rho} to fail in general for the choice $\rho^1 = (\rho^1_k)_{k \in \N}$. However, the choice of $d^{p\textnormal{-}\CC} = (d^{p\textnormal{-}\CC}_k)_{k \in \N_p}$ for any $p \geq 1$ makes projections turn into submetries under the attainment conditions of \cref{prop:Attainment-of-minimisers}. 
    For that it suffices to show that, for every $g_k \in G^k(V)$, there exists a $g_{k+1} \in G^{k+1}(V)$ with $\pi^{k+1}_k(g_{k+1}) = g_k$ and
    \begin{equation*}
        d^{p\textnormal{-}\CC}_k(1_k,g_k) = d^{p\textnormal{-}\CC}_{k+1}(1_{k+1},g_{k+1})
    \end{equation*}
    Suppose now $X^\star$ to be the minimiser of $d^{p\textnormal{-}\CC}_k(1_k,g_k)$ such that $g_k = S_k(X^\star)$. Then choosing the point $g_{k+1} = S_{k+1}(X^\star)$, proves the submetry property, since for any other path $Y$ yielding $S_{k+1}(Y) = g_{k+1} = S_{k+1}(X^\star)$, it holds
    \begin{equation*}
        \pv{X^\star}{p} \leq \pv{Y}{p}.
    \end{equation*}
    Thus $X^\star$ is also minimising for step $k+1$ and the projections become submetries.
\end{remark}
For the subsequent discussion we do not require that the minimizer is attained. However, recall that some metrics of interest, such as $\rho^p_k$, make $(G^k(V),\rho^p_k)$ a proper metric space. \\

\paragraph{\textbf{Identifying group-like elements and signatures.}}
In the $p = 1$ case, as a result of \cite[Theorem 4.4]{LeDonneZuest_public}, given the family of $\CC$ metrics $d^{1\textnormal{-}\CC} = (d^{1\textnormal{-}\CC}_k)_{k \in \N}$, we have
\begin{equation*} \label{eq:Intrinsic-vs-Extrinsic}
    SG^1 = G_{d^{1\textnormal{-}\CC}-p.r.c.},
\end{equation*}
as (pointed) groups, where $G_{d^{1\textnormal{-}\CC}-p.r.c.}$ is group isomorphic (under $\Pi^1$ of \eqref{eq:Inverse-limit-to-product-conversion-map-p}) to the projective limit $(G^\infty_1(V),d_\infty^{1\textnormal{-}\CC})$ in the category of pointed metric groups. In fact, $\Pi^p$ can even be made an isomorphism in the category of pointed metric groups for any $p \geq 1$ when choosing $d^{p\textnormal{-}\CC} = (d^{p\textnormal{-}\CC}_k)_{k \in \N_p}$. 
In that case, we know that the inverse limit carries the topology of an $\R$-tree, hence the associated topological space is a topological tree, which corresponds to 
the $p=1$ case for $D^p$ of \cref{sec:Rough-Path-Theory}.
Lastly, it follows from \hyperref[thm:NoGo-Theorem]{NoGo-Theorem~\ref*{thm:NoGo-Theorem}} that the resulting signature group is not a pointed topological group for $\dim V \geq 2$.

We are now ready to establish the same relation for $p > 1$ as in \cite{LeDonneZuest_public}, namely
\begin{equation*}
    G^\infty_p \xrightarrow{\Pi^p,\sim} G_{d^{p\textnormal{-}\CC}\textnormal{-}p.r.c.} = SG^p
\end{equation*}
with the induced metric, which resembles the $p = 1$ case and establishes the diagram \eqref{eq:Diagram-p-variation-LZ}.

\begin{proposition} \label{prop:Isomorphisms}
    $\Pi^p|_{G^\infty_p(V)}$ is a group isomorphism. If
    in addition  
    $(G^{\floor p}(V),d_{\floor{p}})$ is proper and for all $k$ in $\mathbb N_p$, $S_k$ is sequentially continuous, then $\Pi^p$ is an isomorphism in the category of pointed metric spaces, in the sense of \eqref{eq:Diagram-p-variation-LZ}, where the induced (left-invariant) metric on $SG^p_d$ is given by
    \begin{equation} \label{eq:Inverse-Metric}
        d^p(1,g) = \inf_{X \in \mathcal{A}(g)} \pvd{X}{p}{d_{{\floor{p}}}}, \quad \mathcal{A}(g) = \{ X \in WG\Omega^p_{d_{\floor{p}}} \mid S(X) = g \}, \qquad \forall g \in SG^p_d.
    \end{equation}
\end{proposition}

We provide a proof in full detail.
\begin{proof}
    Recall that $\Pi^p$ is a group isomorphism between the set $G((V))$ of group-like elements and $\prod_{k \in \N_p} G^k(V)$ with $\pi^{k+1}_k(g_{k+1}) = g_k$ for every $k \in \N_p$ and $(g_k)_{k \in \N_p} \in\prod_{k \in \N_p} G^k(V) $. Then by the definition of $G^\infty_{p}(V)$ and $G_{d^{p\textnormal{-}\CC}\textnormal{-}p.r.c.}(V)$ together with the compatibility $g = \Pi^p((g_{k})_{k\in \N_p}))$, the restriction $\Pi^p|_{G^\infty_p(V)} \colon G^\infty_{p}(V) \to G_{d^{p\textnormal{-}\CC}\textnormal{-}p.r.c.}$ is a bijection and it remains to show $G_{d^{p\textnormal{-}\CC}\textnormal{-}p.r.c.} = SG^p_d$ as well as the isometry. \\
    
    For that let $g = S(X) \in SG^p_d$ where $X \in WG\Omega^p_{d_{\floor{p}}}$ realises $g$. Then, for every $k \in \N_p$ set $g_k = \Pi_k(g) = S_k(X)$ and note that by definition
    \begin{equation*}
            d^{p\textnormal{-}\CC}_k(1_k,g_k) \leq \pvd{X}{p}{d_{\floor{p}}} < \infty
    \end{equation*}
    and by taking the supremum over $k \in \N_p$ on the left hand side we conclude $g \in G_{d^{p\textnormal{-}\CC}\textnormal{-}p.r.c.}(V)$ and hence $SG^p_d \subseteq G_{d^{p\textnormal{-}\CC}\textnormal{-}p.r.c.}(V)$. \\
    
    For the converse direction and the isometry,
    \begin{equation*}
        \lim_{k \to \infty} \inf_{X \in \mathcal{A}_k(\Pi_k(g))} \pvd{X}{p}{d_{{\floor{p}}}}^p = \inf_{X \in \mathcal{A}(g)} \pvd{X}{p}{d_{{\floor{p}}}}^p, \qquad \forall g \in G_{d^{p\textnormal{-}\CC}\textnormal{-}p.r.c.}(V),
    \end{equation*}
    define
    \begin{equation*}
        a_k = \inf_{X \in \mathcal{A}_k(g_k)} \pvd{X}{p}{d_{{\floor{p}}}}^p = d^{p\textnormal{-}\CC}_k(1_k,\Pi_k(g))^p \in [0,\infty).
    \end{equation*}
    By the nesting of the $\mathcal{A}_k(g_k)$'s (or equivalently $1$-Lipschitz property of the $d_k^{p\textnormal{-}\CC}$'s), and since $g \in G_{d^{p\textnormal{-}\CC}\textnormal{-}p.r.c.}(V)$, it follows that
    \begin{enumerate}
        \item $a_k$ is monotonically increasing, and
        \item $a_k$ is bounded, since $g \in G$.
    \end{enumerate}
    Hence, $a = \lim_{k \to \infty} a_k < \infty$. \\
    
    By the definition of infimum, there exists a sequence\footnote{One could choose $\varepsilon_k = 1/k$.} $(\varepsilon_k)_{k \in \N_p}$ with $\varepsilon_k \downarrow 0$, and a sequence $(X^{(k)})_{k \in \N_p}$ with $X^{(k)} \in \mathcal{A}_k(\Pi_k(g))$ such that
    \begin{equation*}
        \pvd{X^{(k)}}{p}{d_{{\floor{p}}}}^p \leq a_k + \varepsilon_k.
    \end{equation*}
    Then, since $a_k \leq a < \infty$ and $\varepsilon_k \leq \varepsilon_{\floor p}$, we can take the Hölder reparametrisation of each $X^{(k)}$ with uniform Hölder constant $(a + \varepsilon_{\floor p})^{1/p}$, to show that the sequence $(X^{(k)})_{k \in \N_p}$ is equicontinuous and 
    \begin{equation*}
        X^{(k)}([0,1]) \subset \overline{B}_{(a + \varepsilon_{\floor p})^{1/p}}(1_{\floor{p}}) = \{ x \in G^{\floor{p}}(V) \mid d_{\floor{p}}(1_{\floor{p}},x) \leq (a + \varepsilon_{\floor{p}})^{1/p} \}.
    \end{equation*}
    Since closed balls are compact, we can invoke the Arzelà–Ascoli \cite[Theorem 1.4]{Friz_Victoir_2010_fullbook} theorem to infer the existence of a subsequence $(X^{(k_j)})$ that converges uniformly in $d_{\floor{p}}$--, hence also pointwise-- to a limit $X^*$.
    Using the subsequence $(X^{(k_j)})_{j \in \N}$, we have
    \begin{equation} \label{eq:Sequence-Upperbound-Isomorphism-Pi-p}
        \begin{aligned}
            \pvd{X^*}{p}{d_{\floor{p}}}^p &\leq \liminf_{j \to \infty} \pvd{X^{(k_j)}}{p}{d_{\floor{p}}}^p \leq \liminf_{j \to \infty} (a_{k_j} + \varepsilon_{k_j}) \\
            &= \lim_{k \to \infty} (a_k + \varepsilon_k) = a < \infty
        \end{aligned}
    \end{equation}

    To conclude the converse direction, notice that by the compatibility of the projections, once $k_j \geq m$ for some $m \in \N_p$ with $g_k = \Pi_k(g)$, we obtain
    \begin{equation*}
        S_m(X^{(k_j)}) = \pi^{k_j}_m(S_{k_j}(X^{(k_j)})) = \pi^{k_j}_m(g_{k_j}) = g_m,
    \end{equation*}
    so that by sequential continuity it follows that
    \begin{equation} \label{eq:signature-memebership}
        g_m = S_m(X^{(k_j)}) \to S_m(X^*) = g_m, \qquad \forall m \in \N_p
    \end{equation}
    and hence $X^* \in \mathcal{A}(g)$ and $g \in SG^p_d$. Therefore, $G_{d^{p\textnormal{-}\CC}\textnormal{-}p.r.c.}(V) \subseteq SG^p_d$.

    To conclude the isometry, notice that by \eqref{eq:signature-memebership} and
    \begin{equation*}
        \mathcal{A}(g) = \lim_{k \to \infty} \mathcal{A}_k(\Pi_k(g)) = \bigcap_{k \in \N_p} \mathcal{A}_k(\Pi_k(g)), \qquad \forall g \in G((V)),
    \end{equation*}
    we get $\mathcal{A}(g) \neq \varnothing$ for all $g \in G_{d^{p\textnormal{-}\CC}\textnormal{-}p.r.c.}(V)$. Therefore,
    \begin{equation*}
        \underline{a} = \inf_{X \in \mathcal{A}(g)} \pvd{X}{p}{d_{{\floor{p}}}}^p < \infty.
    \end{equation*}
    and $a \leq \underline{a}$. Since by definition and \eqref{eq:Sequence-Upperbound-Isomorphism-Pi-p}, $\underline{a} \leq \pvd{X^*}{p}{d_{\floor{p}}}^p \leq a$ the proof follows.
\end{proof}

\begin{remark} The diagram \eqref{eq:Diagram-p-variation-LZ} therefore gives an explicit candidate for the resulting metric topology. In particular, given any object in \eqref{eq:Diagram-p-variation-LZ}, we can compute the corresponding metrics by pushing it forward along the respective isomorphism.
\end{remark}

\paragraph{\textbf{Signatures form an $\R$-tree.}}

Together with \cref{sec:Rough-Path-Theory}, \cref{prop:Isomorphisms} now lets us establish the concrete version of \cref{thm:Main-Theorem-Commutative-Diagram}.
\begin{proposition}\label{prop:diag_poitmetgp}
    Let $p \geq 1$ and $V$ be a finite dimensional real vector space. Then the following commutative diagram holds in the category of pointed metric groups:
    \begin{equation} \label{eq:Diagram-p-variation-LZ-precise}
            \begin{tikzcd}
            {((WG\Omega^p/\sim,[o]_\tau),\delta^p)} \arrow[r, "{\cL,\sim}"]\arrow[dd, "\sim"] \arrow[rdd, "{S,\sim}"] & {((\mathcal{I}^p,[O]_\circlearrowleft),D^p)} \arrow[dd, "{\operatorname{ev}_1,\sim}"] \\
            & \\
            {((G^\infty_p(V),1_\infty),d_{\infty}^{p\textnormal{-}\CC})} \arrow[r, "{\Pi^p,\sim}"] & {((SG^p,1),d^p)}
        \end{tikzcd}
    \end{equation}
\end{proposition}

Along the same lines of the second part of \cite[Theorem 4.4]{LeDonneZuest_public}, we get the following result.
\begin{corollary}\label{cor:SGRtree}
    The signature group $(SG^p,d)$ is homeomorphic to an $\R$-tree.
\end{corollary}
\begin{proof}
    This follows from \cref{lem:Reformulation-Uniqueness-to-tree}, \cref{lem:Image-of-ev1} and \cref{prop:Isomorphisms}.
\end{proof}
\begin{remark}
    \begin{enumerate}
        \item Since left translations are isometries, the \hyperref[thm:NoGo-Theorem]{NoGo-Theorem~\ref*{thm:NoGo-Theorem}}, forces the failure of continuity of some right translation of $SG^p$ w.r.t.~the induced metric (for $\dim V \geq 2$).
        \item The sets $(SG^p,d^p)$ (recall \cref{lem:Reformulation-Uniqueness-to-tree}) and $(SG^p,d)$ coincide as pointed metric groups.
    \end{enumerate}
\end{remark}

\paragraph{\textbf{More general limits as trees.}} The construction of \eqref{eq:Diagram-p-variation-LZ-precise} and the above discussion as well as in \cref{sec:Rough-Path-Theory} yield a negative answer to the question in \cite[Remark 8.20]{Schmeding_2022} whether a suitable limiting metric on the untruncated signature group will turn it into a Lie group: Swapping $\rho^p_k$ with any other compatible family of metrics $d_k$, such that the construction yields the diagram \eqref{eq:Diagram-p-variation-LZ} such that 
\begin{enumerate}[label = (\arabic*)]
    \item the UAC property of $(SG^p_d,d)$ holds, since $\eva_1|_{\cI^p}$ is an isomorphism,
    \item the LAC property follows if the minimiser is attained. By properness of $(G^{\floor{p}}(V),d_{\floor{p}})$ for a suitable metric $d_{\floor{p}}$ (i.e.~choose any homogeneous metric), any $p$-variation geodesic space induces an $\R$-tree structure.
\end{enumerate}
We conclude that for all common choices of metrics from rough path analysis the resulting group cannot be a topological group (much less a Lie group) for $\dim V \geq 2$. \\

\paragraph{\textbf{Approximation corollaries.}}
Lastly, we derive a similar approximation result as \cite[Corollary 4.5]{LeDonneZuest_public} in the following way. Let us assume that the diagram \eqref{eq:Diagram-p-variation-LZ} holds in the case of general homogeneous metrics (we may set $d_k = \rho^p_k$ at any point).
The following statement holds.
\begin{corollary} \label{cor:LeDonne-Cor} 
For any (continuous) path $X \colon [0,1] \to
G^{\floor{p}}(V)$ and any positive $\epsilon $ there is some $k \geq \floor{p}$ and a Hölderian geodesic $Y^{(k)} \colon [0,1] \to G^k(V)$ w.r.t.~$p$-variation such that $\pvd{Y^{(k)}}{p}{d_k} \leq K_{d,p} \pvd{X}{p}{d_\up}$ (possibly infinite) for some constant $K_{d,p} > 0$, only depending on $p$ and $d_{\floor p}$, $\pi_\up^k (Y^{(k)} (\delta)) = X(\delta)$
for $\delta = 0, 1$ and
\begin{equation*}
    \|X; \pi_\up^k \circ Y^{(k)}\|_\infty := \sup_{t \in [0,1]} d_{\floor{p}}( X_t, (\pi_{\floor{p}}^k \circ Y^{(k)})_t) < \epsilon.
\end{equation*}
\end{corollary}
\begin{proof}
    Since any continuous path $X$ can be approximated uniformly w.r.~to $d_{\floor p}$ by tree-reduced paths\footnote{The corresponding tree-reduced approximation is as in \cite[Corollary 4.5]{LeDonneZuest_public}. Since the same argument also applies to paths in $G^k(V)$ and any rectifiable path has finite $p$-variation, every continuous path can be approximated by tree-reduced paths of finite $p$-variation. If $X$ is a path of finite $p$-variation, then \cite[Proposition 8.12]{Friz_Victoir_2010_fullbook} gives rectifiable approximants with uniformly bounded $p$-variation. Since the tree-reducing perturbations can be chosen arbitrarily small in $1$-variation, the resulting tree-reduced approximants still have uniformly bounded $p$-variation. Hence, by equivalence of homogeneous metrics at the fixed step $\floor{p}$, we obtain such a constant $K_{d,p}>0$, depending only on $p$ and $d_{\floor{p}}$.}
    of finite $p$-variation. By a triangular argument, we may assume without loss of generality that $X$ is tree-reduced and of finite $p$-variation. Since $X$ is tree-reduced, its lifted path $X_\infty$ to the inverse limit $(G^\infty_d(V),d^{p\textnormal{-}\CC}_\infty)$ via the diagram is injective and by the arguments above a Hölderian geodesic w.r.t.~$p$-variation (not necessarily parametrised to be $1/p$-Hölder). By uniqueness of $X_\infty$ for its endpoints, we derive from \cref{prop:Geodesic-triangular-argument} increasing steps $k_j$ and geodesics $Y^{(k_j)}$ such that $\pi^{k_j}_{\floor p} \circ Y^{(k_j)} \to X$ uniformly (w.r.t.~$d_{\floor p}$). Hence, by choosing $j$ large enough, we obtain the desired $\epsilon$.
\end{proof}

This enables an approximation in $p$-variation by interpolation. Before giving the statement, let us set for two paths of finite $p$-variation $X,Z \colon [0,1] \to (E,d)$ with values in a metric space $(E,d)$ the function
\begin{equation*} 
    \begin{aligned}
        \|X; Z\|_p &:= \left(\sup_{\mathcal{D}\colon 0 = s_0 < s_1 < ... < s_n = 1}\sum_{[s,t] \in \mathcal{D}} |d(X_t,Z_t) - d(X_s,Z_s)|^p\right)^{1/p} \\
    \end{aligned}
\end{equation*}
with $\|X; Z\|_\infty = \sup_{t \in [0,1]} d(X_t,Z_t)$ as defined in \cref{cor:LeDonne-Cor}.
Equivalently, $\|X; Z\|_p$ is the $p$-variation of the real valued path $t \mapsto d(X_t,Z_t)$ with induced distance given by the absolute value. Then recall that for any $q > p$, we have the following interpolation property:
\begin{equation} \label{eq:Interpolation-estimate}
    \|X; Z\|_q \leq (2\|X; Z\|_\infty)^{1-p/q} \left( \pv{X}{p}{} + \pv{Z}{p}{} \right)^{p/q}.
\end{equation}
This follows from the reverse triangle inequality, assuming $s<t$, so that
\begin{equation} \label{eq:Triangle-inequality}
    |d(X_t,Z_t)-d(X_s,Z_s)| \leq d(X_t,Z_t) + d(X_s,Z_s) \leq 2 \|X; Z\|_\infty,
\end{equation}
the standard interpolation \cite[Proposition 5.5]{Friz_Victoir_2010_fullbook} for any $p < q$ with any path $X\colon [0,1] \to (E,d)$, given by
\begin{equation} \label{eq:Standard-Interpolation-estimate}
    \pvd{X}{q}{d} \leq \left(\sup_{0 \leq u < v \leq 1} d(X_u,X_v) \right)^{1-p/q} \pvd{X}{p}{}^{p/q}
\end{equation}
and again by the inverse triangle inequality
\begin{equation*}
    |d(X_t,Z_t) - d(X_s,Z_s)| \leq d(X_t,X_s) + d(Z_t,Z_s)
\end{equation*}
together with the Minkowski inequality, giving for any finite partition $0 = t_0 < t_1 < ... < t_n = 1$ of $[0,1]$
\begin{equation*}
    \left(\sum_{j = 0}^{n-1} |d(X_{t_{j+1}},Z_{t_{j+1}})-d(X_{t_j},Z_{t_j})|^p\right)^{1/p} \leq \pv{X}{p} + \pv{Z}{p}.
\end{equation*}
This further yields an approximation in the following sense.
\begin{corollary} \label{cor:q-variation-approximation-via-p-geodesics}
For any $X \colon [0,1] \to G^\up(V)$ of finite $p$-variation and any $\epsilon > 0$ there is some $k \geq \floor p$ and a Hölderian geodesic $Y^{(k)} \colon [0,1] \to G^k(V)$ w.r.t.~$p$-variation, with $\pi_\up^k (Y^{(k)} (\delta)) = X(\delta)$
for $\delta = 0, 1$ and $\pvd{Y^{(k)}}{p}{d_{k}} < \infty$, such that
\begin{equation*}
    \|X; \pi_\up^k \circ Y^{(k)}\|_q < \epsilon
\end{equation*}
for any $q > p$.
\end{corollary}
Note that the interpolation theorem for $p=1$ (see \cite[Corollary 4.5]{LeDonneZuest_public}) is genuinely stronger by the inverse triangle inequality. Therefore, we state it as an extra corollary. Recall that $V$ is finite dimensional.
\begin{corollary} \label{cor:cor:q-variation-approximation-via-1-geodesics}
    For any path $X \colon [0,1] \to G^1(V)$ of finite $1$-variation and any $\epsilon>0$ there is some $k \geq 1$ and a geodesic $Y^{(k)} \colon [0,1] \to G^k(V)$ such that $\pvd{Y^{(k)}}{1}{d^{1\textnormal{-}\CC}_k} \leq \pvd{X}{1}{d_1}$, $\pi^k_1(Y^{(k)}_\delta) = X_\delta$ for $\delta = 0, 1$ and
    \begin{equation*}
        \|X - \pi^k_1 \circ Y^{(k)}\|_{q\textnormal{-}\var} < \epsilon
    \end{equation*}
    for any $q > 1$.
\end{corollary}
\begin{proof}
    Let $\delta > 0$ and fix $k \geq 1$, such that
    \begin{equation*}
        \sup_{t \in [0,1]} \|X_t - \pi^k_1( Y_t^{(k)})\| <\delta.
    \end{equation*}
    We set $Z = X - \pi^k_1 \circ Y^{(k)}$ and apply the standard estimate \eqref{eq:Standard-Interpolation-estimate} with $d(x,y) = \|x-y\|$ (by the identification $G^1(V)$ with $V$) to $\pv{Z}{1}$. We then estimate the right hand side terms of \eqref{eq:Standard-Interpolation-estimate} in this case. First, we infer from the norm properties
    \begin{equation*}
        \sup_{0 \leq s < t \leq 1}d(Z_t,Z_s) \leq 2 \sup_{t \in [0,1]} \|Z_t\| < 2 \delta.
    \end{equation*}
    The remaining term is finite, since
    \begin{equation*}
        \begin{aligned}
            \|Z_t - Z_s\| &= \| (X_t - X_s) - (\pi^k_1(Y^{(k)}_t) - \pi^k_1(Y^{(k)}_s)) \| \\
            &\leq \| X_t - X_s \| + \| \pi^{k}_1( \, (Y_s^{(k)})^{-1} \ Y_t^{(k)}\, ). \|
        \end{aligned}
    \end{equation*}
    By the continuous path lifting, this implies that
    \begin{equation*}
        \pv{Z}{1} \leq \pv{X}{1} + \pv{\pi^{k}_1 \circ Y^{(k)}}{1} \leq \pv{X}{1} + \pv{Y^{(k)}}{1} =: M_k < \infty.
    \end{equation*}
    Moreover, by assumption, $\pv{Y^{(k)}}{1} \leq \pv{X}{1} < \infty$ and thus $M_k \leq 2 \pv{X}{1} =: M < \infty$ for any initially chosen $k$.
    Interpolation then yields
    \begin{equation*}
        \pv{Z}{q} \leq (2\delta)^{1-1/q} M^{1/q}.
    \end{equation*}
    Then choosing $\delta$ such that $(2\delta)^{1-1/q} M^{1/q} < \epsilon$ closes the proof.
\end{proof}
\begin{remark}
    The proof for the Banach case $p=1$ implicitly uses that right translations are isometries. For a general metric group in the sense of our category, this might not be the case. However, if there exist uniform estimates bounding right translation (hence also making it continuous) according to $d(g l,h l) \leq C d(g,h)$, then a similar result to the Banach case can be deduced in a direct way.
\end{remark}

\section{Transfer to Character Groups}
\label{sec:Character-Groups}
In this section we sketch how the results obtained in the rest of the article transfer to the setting of character groups of Hopf algebras. Various flavours of rough paths (for example branched rough paths, cf. \cite{Schmeding_2022}) can be realised as rough paths taking values in (the character group) of a suitable Hopf algebra. Let us first recall:
\begin{definition}
Let $\cH=(\cH,m,\Delta,\varepsilon,S)$ be a real Hopf algebra. We say that $\cH$ is
\emph{connected, graded and of finite type} if
\begin{equation}
    \cH=\bigoplus_{n\in\N_0}\cH_n,\qquad \cH_0\cong\R,\qquad \dim(\cH_n)<\infty
    \quad\forall n\in\N_0,
\end{equation}
and the algebra and coalgebra structures are compatible with the grading, i.e.
\begin{equation}
    m(\cH_n\otimes \cH_m)\subseteq \cH_{n+m},
    \qquad
    \Delta(\cH_n)\subseteq \bigoplus_{k+\ell=n}\cH_k\otimes \cH_\ell
\end{equation}
for all $n,m\in\N_0$.
\end{definition}
The character group of  a connected graded Hopf algebra $\cH$ is defined by
\begin{equation}
    G(\cH) :=
    \left\{
        \varphi\in \cH^\ast \;\middle|\;
        \langle \varphi,m(u,v)\rangle=\langle \varphi,u\rangle\langle \varphi,v\rangle
        \ \forall u,v\in \cH,\ 
        \varphi(1_\cH)=1
    \right\},
\end{equation}
where we write $\langle \varphi,h\rangle:=\varphi(h)$ for the canonical pairing. The group
law is given by convolution,
\begin{equation*}
    (\varphi\star\psi)(h) := \langle \varphi\otimes\psi,\Delta h\rangle, \qquad h\in \cH.
\end{equation*}
In the same way, for $k$  in $\N_0$ we set
\begin{equation}
    \cH_{\leq k}:=\bigoplus_{j=0}^k \cH_j
\end{equation}
and define the truncated character group of step $k$ by
\begin{equation*}
    G_k(\cH) :=
    \left\{
        g\in \cH_{\leq k}^\ast \;\middle|\;
        \langle g,m(u,v)\rangle=\langle g,u\rangle\langle g,v\rangle
        \ \forall u\in\cH_n,\ v\in\cH_m,\ n+m\leq k,\
        g(1_\cH)=1
    \right\}.
\end{equation*}
Equipped with the induced convolution product, $(G_k(\cH),\star_k)$ is a nilpotent finite-dimensional
Lie group with respect to the subspace topology induced by the vector topology on the convolution algebra. Note that this turns the Lie group exponential map into a diffeomorphism. The projection maps
\begin{equation*}
    \pi_j^k:G_k(\cH)\to G_j(\cH),\qquad \pi_j^k(g):=g|_{\cH_{\leq j}},
    \qquad k\geq j,
\end{equation*}
yield again an inverse system of Lie groups
\begin{equation*}
    \cG_p(\cH):=\bigl(G_k(\cH),(\pi_j^k)_{k\geq j\geq \lfloor p\rfloor}\bigr)_{k\in\N_p}.
\end{equation*}
Assume now that there is a metric family
\begin{equation*}
    d=(d_k)_{k\in\N_p}\in \metlift(\cG_p(\cH)),
\end{equation*}
i.e.~the projections $\pi_j^k$ are $1$-Lipschitz and the same extension/lifting theorem
as in Section \ref{sec:Canonical-metric-relations-goals} holds from level ${\floor{p}}$ to all higher levels. Then all constructions and arguments for rough paths carry over directly after replacing the free nilpotent groups with the truncated character groups $G_k(\cH)$. This gives rise to the notion of Hopf algebra valued rough paths. For the choice of the shuffle algebra we recover the weakly geometric rough paths. In case of the Butcher--Connes--Kreimer, we obtain the branched rough paths \cite[Example 8.27]{Schmeding_2022}.
\begin{remark}\label{rem:hopfrp}
    In the Hopf algebraic rough path setting considered in \cite{Tapia_2020}, a lifting theorem for rough paths from the truncated groups to the character group of the Hopf algebra is available, see~\cite[Theorem 3.4 and Theorem 3.9]{Tapia_2020}. To the best of our knowledge no general replacement for \cref{prop:Kernel-signature-characterisation}, i.e.~a characterization of the kernel of the corresponding signature by tree-like equivalence, is known in the literature. However, for certain Hopf algebras (encompassing in particular the Hopf algebras associated to branched, planarly branched and quasi-geometric rough paths), the results from \cite{BaFaT26} can be used to characterise the kernel: Given a suitable Hopf algebra $\cH$ (cofree over the primitive elements, see loc.cit. for the conditions), \cite[Theorem 3.4]{BaFaT26} constructs an isomorphism from $\cH$ to the shuffle Hopf algebra over the primitive elements. This identifies Hopf algebra valued paths with (anisotropic) geometric rough paths. Then \cite[Theorem 1.1]{BOEDIHARDJO_LYONS_2016720} applies for the geometric regime and characterises the kernel. Note however, that due to the change of the Hopf algebra using the isomorphism, one loses the direct interpretation of tree-like equivalence. One would have to track the elements in the new coordinates to interpret tree-likeness as a condition on the original Hopf algebra valued rough paths.
\end{remark}
    Once the kernel of the signature map is known, the remaining constructions and the tree-topological conclusions transfer to the setting of Hopf algebra valued rough paths. So we deduce from \Cref{rem:hopfrp} that in particular a suitable version of the conclusions of \Cref{prop:diag_poitmetgp} holds.
    
\section{Further Remarks}

The (full) isomorphism property of $\Pi^1$ was already present in \cite{LeDonneZuest_public}, but not explicitly used as the discussion was restricted to the group isomorphism only. This work makes this part more explicit.

Furthermore, by establishing the diagram together with the abstract results, we hope that the reader now has deeper insights into the connection between \cite{BOEDIHARDJO_LYONS_2016720} and \cite{LeDonneZuest_public}.
\begin{remark}
    We stress that one must be aware that the notion of paths is still hidden inside the choice of $d^{p\textnormal{-}\CC}$ for $G_{d^{p\textnormal{-}\CC}\textnormal{-}p.r.c.}$.
\end{remark}
While \cref{prop:Isomorphisms} together with the NoGo-Theorem by a discussion of Felix Medwed with Terry Lyons during the Oxford-Berlin-Meeting in December 2025 is certainly known in some form, this work gives a \textquote{top-down view} on the signature group from the perspective of metric geometry, ref.~\eqref{eq:Intrinsic-vs-Extrinsic}, in the form \textquote{$\{ g \textrm{ group-like}\mid \textrm{Property of } g \}$}, which can be a preferred setup for an algebraist or geometer.

\section*{Tool and computational resource disclosure}

During the preparation of the manuscript we used standard LLM models (ChatGPT 5.0 to 5.6 Sol) to discuss, generate and improve proofs and counterexamples. This led in particular to the current versions of \cref{lem:Reformulation-Uniqueness-to-tree}, \cref{lem:euivalent-topologies}, \cref{prop:homogeneous-metric-characterisaton} and \cref{prop:Isomorphisms}.

\printbibliography

\end{document}